\documentclass[reqno, 12pt]{amsart} 

\usepackage[T1]{fontenc}
\usepackage[expansion=false]{microtype} 
\usepackage{amsfonts,amsthm,amsmath,amssymb,amscd,mathrsfs,mathtools} 
\allowdisplaybreaks
\usepackage{latexsym} 
\usepackage{cite} 
\usepackage[colorlinks=true,linkcolor=blue,citecolor=red,urlcolor=black]{hyperref} 
\usepackage{graphicx}
\usepackage{indentfirst} 
\usepackage{color} 

\usepackage{lmodern} 
\usepackage[pagewise]{lineno} 

\usepackage{enumitem}
\usepackage[nameinlink,noabbrev]{cleveref}

\theoremstyle{plain} 
\newtheorem{Thm}{Theorem}[section] 
\newtheorem{Lem}[Thm]{Lemma}     
\newtheorem{Prop}[Thm]{Proposition}
\newtheorem{Cor}[Thm]{Corollary}

\theoremstyle{definition}
\newtheorem{Def}[Thm]{Definition}

\theoremstyle{remark}
\newtheorem{Rem}[Thm]{Remark}

\numberwithin{equation}{section} 

\newcommand{\beq}{\begin{equation}}
	\newcommand{\eeq}{\end{equation}}
\newcommand{\ben}{\begin{eqnarray}}
	\newcommand{\een}{\end{eqnarray}}
\newcommand{\beno}{\begin{eqnarray*}}
	\newcommand{\eeno}{\end{eqnarray*}}

\usepackage[
    letterpaper,                 
    textheight=8.35in,           
    headsep=20pt,                
    footskip=36pt,               
    marginparwidth=0pt,          
    marginparsep=0pt,            
    left=0.75in, 
    right=0.75in
]{geometry}

\newcommand{\R}{\mathbb R}
\newcommand{\PP}{\mathbb P}
\newcommand{\dd}{\,\mathrm d}
\newcommand{\diver}{\nabla\!\cdot}
\newcommand{\BUC}{\mathrm{BUC}}

\newcommand{\Acal}{\mathcal A}

\setlist{nosep}
\Crefname{Def}{Definition}{Definitions}
\crefname{Thm}{theorem}{theorems}
\Crefname{Thm}{Theorem}{Theorems}
\crefname{Lem}{lemma}{lemmas}
\Crefname{Lem}{Lemma}{Lemmas}
\crefname{Prop}{proposition}{propositions}
\Crefname{Prop}{Proposition}{Propositions}
\crefname{Cor}{corollary}{corollaries}
\Crefname{Cor}{Corollary}{Corollaries}
\crefname{Rem}{remark}{remarks}
\Crefname{Rem}{Remark}{Remarks}
\hypersetup{
 pdftitle={Global Regularity for the Chemotaxis--Stokes System with Signal Consumption in R3},
 pdfauthor={Fengqiang Shi, Wendong Wang, Guoxu Yang},
 pdfsubject={Three-dimensional chemotaxis with signal consumption and Stokes coupling},
 pdfkeywords={Chemotaxis, Stokes equations, signal consumption, global classical regularity},
 bookmarksnumbered=true
 }

\title[Global Regularity for Chemotaxis--Stokes]{Global Regularity for the
 Chemotaxis--Stokes System with Signal Consumption in $\mathbb{R}^3$}

\author[F. Shi, W. Wang and G. Yang]{
Fengqiang~Shi$^{1}$ \and
Wendong~Wang$^{1,*}$ \and
Guoxu~Yang$^{1}$
}

\thanks{$^{1}$School of Mathematical Sciences, Dalian University of Technology, 
Dalian 116024, China.}

\thanks{$^{*}$Author to whom correspondence should be addressed.}

\thanks{E-mail addresses:
2192593228@mail.dlut.edu.cn (Fengqiang~Shi),
wendong@dlut.edu.cn (Wendong~Wang), 
guoxu\_dlut@outlook.com (Guoxu~Yang)}

\date{\today}

\begin{document}

\begin{abstract}

Motivated by the continuum model of aerobic bacteria proposed by Tuval et al. in 2005, we study the large-data regularity problem for the three-dimensional chemotaxis--Stokes system with signal consumption. We prove the global existence and uniqueness of classical solutions to the Cauchy problem in $\mathbb R^3$ for smooth initial data with finite density mass and moment and finite signal Fisher information, without any smallness assumptions.

The proof is based on three main ideas. First, we use a terminal heat-kernel representation to define a finite signal value at every spatial point and prove strong $L^a$ convergence of the backward parabolic signal rescalings to this value for every
$1\le a<\infty$. Second, we prove a localized entropy estimate that remains uniform under density normalization, even when the resulting consumption coefficient becomes arbitrarily large. These estimates give two complementary mechanisms:
consumption excludes normalized density concentration at points with positive terminal signal, while a localized weighted estimate improves density integrability at points
with zero terminal signal. Third, we establish a local $\varepsilon$-regularity criterion based on smallness of scale-invariant density, velocity, and pressure quantities. A heat decomposition recovers the required signal-gradient smallness, removing it as a separate assumption. Together, these yield the decay of a scale-invariant density quantity and then exclude finite-time singularities.

\end{abstract}

\maketitle

\tableofcontents

\section{Introduction}\label{sec:intro}

\subsection{Chemotaxis--fluid models}

Chemotaxis is the directed movement of cells in response to a chemical
gradient. When the cells are suspended in a fluid, their migration is
coupled to fluid transport, and their density can in turn drive the
fluid through buoyancy. Experiments on swimming bacteria reveal
collective motion, self-concentration, and enhanced transport
\cite{Dombrowski2004}. In 2005, Tuval--Cisneros--Dombrowski--Wolgemuth--Kessler--Goldstein \cite{Tuval2005} developed a continuum model for aerobic bacteria
swimming toward oxygen and consuming it in a viscous incompressible
fluid. This model motivated the study of chemotaxis--fluid systems
with signal consumption. The full chemotaxis--Navier--Stokes model takes the form
\begin{equation}\label{eq:NS-background}
\begin{cases}
n_t+u\cdot\nabla n=\Delta n-\nabla\cdot(n\nabla c),\\
c_t+u\cdot\nabla c=\Delta c-nc,\\
u_t+(u\cdot\nabla)u-\Delta u+\nabla\pi=n\nabla\Phi,
 \qquad \nabla\cdot u=0,
\end{cases}
\end{equation}
where \(n\ge0\) is the cell density, \(c\ge0\) the oxygen or signal
concentration, \(u\) the fluid velocity, and \(\pi\) the pressure.
The prescribed potential \(\Phi\) represents the external force.
The drift \(-\nabla\cdot(n\nabla c)\) drives cells toward higher
signal concentrations, whereas \(-nc\) describes the consumption.
All diffusion, chemotactic, and consumption coefficients in this
model are fixed to one. Omitting the velocity self-advection term of
\((u\cdot\nabla)u\) gives the chemotaxis--Stokes system.
The two transport terms \(u\cdot\nabla n\) and \(u\cdot\nabla c\)
remain, so the coupling between chemotaxis and fluid motion persists.

\subsection{Weak solutions, small-data regularity, and logistic damping}
In $\R^2$, Liu--Lorz~\cite{LL2011} constructed global weak solutions
to a chemotaxis--Navier--Stokes system for large initial data under
structural assumptions on the chemotactic sensitivity $\chi$ and the
consumption rate $k$. Chae--Kang--Lee~\cite{ChaeKangLee2013} obtained
global regular solutions for smooth initial data under assumptions that
include $\chi,k,\chi',k'\geq 0$ and smallness of
$\|\chi-\mu k\|_{L^\infty}$ for some constant $\mu$.
Zhang--Zheng~\cite{ZZ2014} proved existence and uniqueness of global
weak solutions to the Cauchy problem for a class of initial data.
In the 2D smooth bounded convex domain, Winkler~\cite{Winkler2012}
proved global existence and uniqueness of classical solutions for large
initial data under coefficient assumptions that cover constant
sensitivity and linear signal consumption. In smooth bounded convex domains in $\R^3$, Winkler~\cite{Winkler2016}
constructed global weak solutions to the chemotaxis--Navier--Stokes
system. He later proved in \cite{Winkler2017} that every eventual
energy solution becomes classical after a finite waiting time and
converges uniformly to $(\overline{n}_0,0,0)$ as $t\to\infty$, where
$\overline{n}_0$ is the spatial mean of the initial density.
For the Cauchy problem in $\R^3$ with constant sensitivity and linear
signal consumption, Kang--Lee--Winkler~\cite{KangLeeWinkler2022}
constructed global weak solutions under integrability and entropy
assumptions on the initial data.
Chen--Li--Wang--Wang~\cite{ChenLiWangWang2025} constructed global
suitable weak solutions satisfying a local energy inequality under
structural assumptions on the coefficients and integrability and energy
assumptions on the initial data. Here we study global classical
solvability for smooth large initial data in the chemotaxis--Stokes
system in $\R^3$ with constant sensitivity and linear signal consumption.

The existence of strong solution was established by Duan--Lorz--Markowich \cite{DuanLorzMarkowich2010} for the chemotaxis--Navier--Stokes
system in $\mathbb R^3$ with small perturbations of constant states. Note that
for systems with signal consumption, the maximum principle gives
$\|c(t)\|_\infty\leq\|c_0\|_\infty$. For the chemotaxis system
without fluid coupling in a smooth bounded domain
$\Omega\subset\mathbb R^d$, $d\geq 2$, Tao \cite{Tao2011} proved
global existence and uniform boundedness when $\|c_0\|_\infty$ is
sufficiently small. Under a smallness condition on $\|c_0\|_\infty$,
Chae--Kang--Lee \cite[Theorem~2]{ChaeKangLee2014} obtained global
classical solutions to the chemotaxis--Navier--Stokes system in
$\mathbb R^2$ and the chemotaxis--Stokes system in $\mathbb R^3$.
In smooth bounded domains in dimensions two and three,
Cao--Lankeit \cite{CaoLankeit2016} proved global classical
solvability for small initial data in a chemotaxis--Navier--Stokes
system with matrix-valued sensitivity.
For the chemotaxis--Navier--Stokes system in $\mathbb R^3$,
Chen--Li--Wang \cite{ChenLiWang2025} used local entropy estimates
and an $\varepsilon$-regularity criterion to show that the spatial
singular set at the first blow-up time has zero $5/3$-dimensional
Hausdorff measure.


A logistic source $\kappa n-\mu n^2$ with $\mu>0$ yields a
space-time $L^2$ estimate for the cell density through the mass
balance. For the signal-consumption system without fluid coupling
in smooth bounded domains of any dimension, Lankeit--Wang
\cite{LankeitWang2017} constructed global weak solutions for every
$\mu>0$ and proved uniform boundedness of global classical solutions
when $\mu$ is sufficiently large in terms of the dimension and
$\chi\|c_0\|_\infty$. For the consumption system coupled to the
Navier--Stokes equations in smooth bounded domains in $\mathbb R^3$,
Lankeit \cite{Lankeit2016} constructed global weak solutions that
become classical after finite time and converge to
$(\kappa/\mu,0,0)$ under suitable decay assumptions on the external
force. For the signal-production equation
$c_t+u\cdot\nabla c=\Delta c-c+n$ in smooth bounded convex domains
in $\mathbb R^3$, Winkler \cite{Winkler2019} obtained global
generalized solutions for every $\mu>0$ and proved stabilization of
$n$ and $c$ when $\mu>\frac{\chi}{4}\sqrt{\max\{\kappa,0\}}$.


\subsection{Keller--Segel concentration and mechanisms of suppression}

The competition between diffusion and chemotactic attraction is
central to Keller--Segel theory, originating in
\cite{KellerSegel1970}. For the parabolic--elliptic Keller--Segel system on $\R^2$ with
chemotactic sensitivity $\chi>0$, Blanchet--Dolbeault--Perthame
\cite{BlanchetDolbeaultPerthame2006} considered initial densities
of mass $M$ with finite entropy and finite second moment. They proved
global existence of weak solutions for $\chi M<8\pi$ and finite-time
blow-up for $\chi M>8\pi$.
For the fully parabolic system in a ball in $\R^d$, $d\geq 3$,
Winkler \cite{Winkler2013} showed that for every $M>0$ there exist
positive radial initial data with cell mass $M$ for which the solution
blows up in finite time.
For the Keller--Segel--Navier--Stokes system on $\R^3$ with buoyancy
and elliptic signal equation $-\Delta c=n$, Li--Zhou
\cite{LiZhou2025} constructed smooth initial data with nonnegative
cell density and finite mass for which the solution blows up in
finite time.

Fluid transport can also suppress concentration. Kiselev--Xu
\cite{KiselevXu2016} proved suppression of chemotactic explosion by
suitable mixing flows. Bedrossian--He \cite{BedrossianHe2017} and \cite{BedrossianHe2018} established suppression
through strong shear flows for parabolic--elliptic
Patlak--Keller--Segel equations, and He \cite{He2018} treated a
parabolic--parabolic system with a strictly monotone shear.
These results use prescribed flows with quantitative mixing or
shear properties. For an active coupling, Hu--Kiselev--Yao
\cite{HuKiselevYao2025} proved global regularity for a
two-dimensional parabolic--elliptic Patlak--Keller--Segel system
coupled through buoyancy to Darcy's law in a periodic strip,
even for arbitrarily small nonzero coupling strength. Recently, Cui-Wang-Wang \cite{CWW2026}, Cui-Wang-Wang-Wei-Yang \cite{CWWWY2026} suppressed the blow-up in the 3D Patlak--Keller--Segel--Navier-Stokes system via Couette flow. Such results
show that the equation determining the fluid, the geometry, and the
signal law all matter in deciding between blow-up and regularity.

For the standard consumption law, the term \(-nc\) suggests a
different mechanism: a large density depletes the signal responsible
for attraction. Turning this observation into a regularity argument
requires control of spatial gradients and concentration scales;
small signal values at individual points do not directly control
\(\nabla c\). Tao--Winkler \cite{TaoWinkler2012} constructed
large-data weak solutions of the fluid-free consumption system
that become smooth and stabilize in three-dimensional bounded
convex domains. Lankeit--Winkler \cite{LankeitWinkler2023}
surveyed this theory and subsequently established weak solvability
and eventual regularity in general higher-dimensional bounded
domains \cite{LankeitWinkler2025}. 

\subsection{The large-data regularity question and the main results}
The gap between global weak existence and global classical
regularity has been explicitly emphasized in the literature.
In the introduction to \cite{Winkler2017}, Winkler discusses the
possibility of finite-time singularities even after removing
\((u\cdot\nabla)u\), as well as for the fluid-free consumption
subsystem, and said that "{\it Even in this simplified setting,
all these solutions constructed so far are merely weak solutions, with widely
unknown boundedness and regularity properties which in fact might be poor so as
to be consistent with several conceivable types of blow-up phenomena}". For the three-dimensional fluid-free  problem,
Jiang--Wu--Zheng \cite{JiangWuZheng2018}
state that "{\it However, when d = 3, to the best of our knowledge, it is still an open problem whether
the classical solution to problem (1.1)–(1.3) exists globally or blows up in finite time with
arbitrary large regular initial data.}"
Their work gives
blow-up criteria; it does not construct a blow-up solution for the
standard consumption system. The same large-data issue is
identified in the survey \cite[Section~2]{LankeitWinkler2023} by  Lankeit--Winkler. The introduction to
\cite{LankeitWinkler2025} by    Lankeit--Winkler again explains the role of small initial
signal in the available higher-dimensional classical theory.

These discussions concern several geometries and solution classes.
The question we address here is the following whole-space version:
\begin{quote}
\itshape {\bf Can smooth, arbitrarily large initial data for the
three-dimensional chemotaxis--Stokes consumption system
produce a finite-time singularity?}
\end{quote}
In the following, we establish global classical solvability for this problem
under the finite-mass, moment, and signal-energy assumptions.


Consider the Cauchy problem
\begin{equation}\label{eq:system}
\begin{cases}
n_t+u\cdot\nabla n=\Delta n-\nabla\cdot(n\nabla c),
 &x\in\R^3,\ t>0,\\
c_t+u\cdot\nabla c=\Delta c-nc,
 &x\in\R^3,\ t>0,\\
u_t-\Delta u+\nabla\pi=nF,\quad \nabla\cdot u=0,
 &x\in\R^3,\ t>0,
\end{cases}
\end{equation}
with \((n,c,u)|_{t=0}=(n_0,c_0,u_0)\), where \(F=\nabla\Phi\) is prescribed. Fix \(0<\beta<1\) and assume
\begin{align}
&n_0,c_0\in C_b^{2+\beta}(\R^3),\qquad n_0,c_0\ge0,
\label{eq:data-smooth}\\
&M:=\int_{\R^3}n_0<\infty,\qquad
 \int_{\R^3}|x|^2n_0(x)\dd x<\infty,\qquad
 I_0:=2\int_{\R^3}|\nabla\sqrt{c_0}|^2<\infty,
\label{eq:data-integral}\\
&u_0\in H^4(\R^3;\R^3),\qquad \nabla\cdot u_0=0,\qquad
 F=\nabla\Phi\in W^{3,\infty}(\R^3;\R^3).
\label{eq:data-fluid}
\end{align}
Here \(C_b^{2+\beta}\) denotes bounded functions with bounded derivatives
through order two and a globally \(\beta\)-H\"older continuous second
derivative. 

We specify the mild-solution class before stating the main theorem.
Let \(\BUC\) denote bounded uniformly continuous functions, and let
\(\BUC^1\) require this property for the function and its first
derivatives. All the following spaces are over \(\R^3\). Define
\begin{equation}\label{eq:Xspace}
X:=(L^1\cap\BUC)\times\BUC^1
\times(L^2\cap\BUC)_\sigma ,
\end{equation}
where the subscript $\sigma$ means distributionally divergence-free, and equip
this space with the norm
\begin{equation}\label{eq:Xnorm}
\|(n,c,u)\|_X
=\|n\|_1+\|n\|_\infty+\|c\|_{W^{1,\infty}}
+\|u\|_2+\|u\|_\infty .
\end{equation}
Write 
\(\PP\) for the Helmholtz--Leray
projection onto divergence-free vector fields in \(L^2(\R^3;\R^3)\).

\begin{Def}\label{def:mild}
Let \(0<T\le\infty\) and \((n_0,c_0,u_0)\in X\).
A triple \((n,c,u)\) is a bounded mild solution of
\eqref{eq:system} on \([0,T)\) with initial data $(n_0,c_0,u_0)$ if
\((n,c,u)\in C([0,T'];X)\) for every \(0<T'<T\)
and for every \(0<t<T\),
\begin{equation}\label{eq:mild-definition}
\begin{aligned}
n(t)&=e^{t\Delta}n_0
-\int_0^t\nabla\cdot e^{(t-s)\Delta}
       [n(\nabla c+u)](s)\dd s,\\
c(t)&=e^{t\Delta}c_0
-\int_0^t e^{(t-s)\Delta}
       [nc+u\cdot\nabla c](s)\dd s,\\
u(t)&=e^{t\Delta}u_0
+\int_0^t e^{(t-s)\Delta}\PP(nF)(s)\dd s.
\end{aligned}
\end{equation}
The associated pressure is recovered, up to a
function of time, from
\(\nabla\pi=(I-\PP)(nF)\) in the sense of distributions.
\end{Def}

\begin{Thm}\label{thm:main}
Under \eqref{eq:data-smooth}--\eqref{eq:data-fluid}, the system
\eqref{eq:system} with initial data \((n_0,c_0,u_0)\) has a unique
global classical solution in the bounded mild-solution class
defined in Definition \ref{def:mild}. 
For all \(t\ge0\),
\begin{equation}\label{eq:mass-max}
n(t)\ge0,\qquad 0\le c(t)\le A:=\|c_0\|_\infty,\qquad
\int_{\R^3}n(t)=M.
\end{equation}
For every finite \(T>0\),
\begin{equation}\label{eq:main-bound}
\sup_{0\le t\le T}\!
\left(\|n(t)\|_\infty+\|c(t)\|_{W^{1,\infty}}
+\|u(t)\|_2+\|u(t)\|_\infty\right)<\infty .
\end{equation}
\end{Thm}

\begin{Rem}\label{rem:scope}
The above theorem answers the questions in the whole space asked by Winkler in \cite{Winkler2017}, Jiang--Wu--Zheng in \cite{JiangWuZheng2018},   Lankeit--Winkler in \cite{LankeitWinkler2023} and so on.
For the three-dimensional Stokes coupling on \(\R^3\),
Chae--Kang--Lee \cite[Theorem~2]{ChaeKangLee2014} establish global
classical solvability under a sufficiently small
\(\|c_0\|_{L^\infty}\), together with their Sobolev-data and
coefficient assumptions. Their admissible coefficients include
the choice of constant sensitivity and linear consumption in
\eqref{eq:system}. Under the present data hypotheses
\eqref{eq:data-smooth}--\eqref{eq:data-fluid},
Theorem \ref{thm:main} removes this initial-signal smallness requirement.
The comparison concerns global classical existence; the finite-mass,
moment, and Fisher assumptions of the present theorem are retained.
The theorem allows a positive constant signal background.
It gives boundedness on each finite time interval; it does not assert
a bound uniform as \(T\to\infty\), nor the algebraic decay asserted
in \cite[Theorem~2]{ChaeKangLee2014} under its own hypotheses.
\end{Rem}

Taking \(F=0\) and \(u_0=0\) gives the fluid-free consequence.

\begin{Cor}\label{cor:ks}
Suppose \eqref{eq:data-smooth} and \eqref{eq:data-integral}. The Cauchy problem
\begin{equation}\label{ks:system}
 n_t=\Delta n-\diver(n\nabla c),\qquad
 c_t=\Delta c-nc,\qquad (n,c)|_{t=0}=(n_0,c_0)
\end{equation}
has a unique global classical solution. Moreover, \eqref{eq:mass-max} holds for all $t\ge0$ and for every finite $T>0$, it holds
\begin{equation}\label{ks:finite-bound}
 \sup_{0\leq t\leq T}
 \bigl(\|n(t)\|_\infty+\|c(t)\|_{W^{1,\infty}}\bigr)<\infty.
\end{equation}
\end{Cor}

\begin{Rem}\label{rem:ks-comparison}
Tao \cite{Tao2011} proves global bounded classical solvability for
the Neumann problem on bounded smooth domains under smallness of
the initial signal relative to the chemotactic sensitivity.
For the corresponding whole-space Stokes system, the smallness
condition in \cite[Theorem~2]{ChaeKangLee2014} also applies when
\(F=0\) and \(u_0=0\). Corollary~\ref{cor:ks} gives global
classical solvability of \eqref{ks:system} for arbitrary
\(\|c_0\|_\infty\) under
\eqref{eq:data-smooth}--\eqref{eq:data-integral}, without logistic
damping.
\end{Rem}

\begin{Cor}\label{cor:stability}
Under the hypotheses of Corollary~\ref{cor:ks}, let $(n,c)$ be its solution
and let $(n_\mu,c_\mu)$ denote the maximal classical solution with the same
initial data of
\begin{equation}\label{ks:damped}
 \begin{cases}
 (n_\mu)_t=\Delta n_\mu-\diver(n_\mu\nabla c_\mu)-\mu n_\mu^2,\\
 (c_\mu)_t=\Delta c_\mu-n_\mu c_\mu.
 \end{cases}
\end{equation}
For every finite $T>0$, there exist $\mu_T>0$ and $C_T<\infty$ such
that these solutions exist on $[0,T]$ for $0\leq\mu\leq\mu_T$ and
\begin{equation}\label{ks:stability}
 \sup_{0\leq t\leq T}
 \left(\|n_\mu(t)-n(t)\|_\infty
       +\|c_\mu(t)-c(t)\|_{W^{1,\infty}}\right)\leq C_T\mu.
\end{equation}
\end{Cor}

\begin{Rem}\label{rem:damping-scope}
Lankeit--Wang \cite[Theorem~1.1]{LankeitWang2017} obtain global
bounded classical solutions of a bounded-domain consumption model
when the quadratic damping coefficient is sufficiently large
relative to the initial signal and the sensitivity.
Corollary~ \ref{cor:stability} addresses a different question:
convergence on each prescribed finite time interval as damping
vanishes, after the undamped solution has been constructed by
Corollary~ \ref{cor:ks}. No initial-signal smallness is used in this
corollary. Its constants may depend on \(T\), and the admissible
range \(0\le\mu\le\mu_T\) may shrink as \(T\) increases.
Thus \eqref{ks:stability} is a consequence of the undamped theorem,
as proved in Section~\ref{ks:section}; it is not an assumption used to
prove that theorem by a vanishing-damping limit.
\end{Rem}

We next state a local $\varepsilon$-regularity criterion used to prove
Theorem~\ref{thm:main}.
We write \(f(t)=f(t,\cdot)\) for the spatial profile at time \(t\). For \(z_0=(t_0,x_0)\), set
\begin{equation}\label{eq:cylinders}
Q_r(z_0)=(t_0-r^2,t_0)\times B_r(x_0),\qquad
(\pi)_{B_r}(t)=\frac{1}{|B_r|}\int_{B_r(x_0)}\pi(t,x)\dd x .
\end{equation}
Using \eqref{eq:cylinders}, we define
\begin{equation}\label{eq:scale-quantities}
\begin{split}
\Acal_p(r;z_0)&:=r^{2p-5}\iint_{Q_r(z_0)}n^p,\\
U(r;z_0)&:=\|u\|_{L^5(Q_r(z_0))},\\
P(r;z_0)&:=\|\pi-(\pi)_{B_r}(t)\|_{L_t^2L_x^3(Q_r(z_0))}.
\end{split}
\end{equation}

\begin{Thm}\label{thm:epsilon}
Let \(0\leq G<\infty\) be given and \(3/2<p\le5/3\). There are
\(\varepsilon_*=\varepsilon_*(A,G,p)>0\) and
\(C=C(A,G,p)\) with the following property.
Suppose that \eqref{eq:system} holds classically in \(Q_R(z_0)\), with
\(n>0\), \(0\le c\le A\), and \(R\|F\|_\infty\le G\).
If
\begin{equation}\label{eq:epsilon-hyp}
\Acal_p(R;z_0)+U(R;z_0)+P(R;z_0)\le\varepsilon_*,
\end{equation}
then
\begin{equation}\label{eq:epsilon-conc}
\sup_{Q_{\frac{R}{2}}(z_0)}\left(n+|\nabla c|^2+|u|^2\right)
\le CR^{-2}.
\end{equation}
\end{Thm}

\begin{Rem}\label{rem:epsilon-comparison}
The local entropy calculation in Section~\ref{sec:epsilon} is related to
Chen--Li--Wang \cite[Lemma~15 and Theorem~7]{ChenLiWang2025}.
For the construction of suitable weak solutions satisfying a
local energy inequality, see Chen--Li--Wang--Wang
\cite{ChenLiWangWang2025}.
Here the fluid equation is Stokes, the pressure oscillation is
measured in \(L_t^2L_x^3\), and the signal-gradient input is
derived in \eqref{eq:gradient-reduction} from the density and
velocity inputs. Consequently \eqref{eq:epsilon-hyp} contains no
separate smallness assumption on \(\nabla c\). It is a local
smallness criterion at an arbitrary bounded signal amplitude
\(A\), not a smallness assumption on \(c_0\). The strict range
\(p>3/2\) in Theorem~\ref{thm:epsilon} is retained throughout the proof.
\end{Rem}

\subsection{Main ideas of the proof}

Suppose that the maximal classical existence time is $T<\infty$.
The global energy estimates give
$n\in L^1((0,T)\times\R^3)\cap L^{5/3}((0,T)\times\R^3)$,
but do not directly imply smallness of
$\Acal_p(r;T,x_0)=r^{2p-5}\iint_{Q_r(T,x_0)}n^p$,
because its prefactor is singular as $r\rightarrow0$.
Moreover, local consumption does not by itself give smallness
of the signal gradient.
We therefore construct a terminal signal representative,
exclude density concentration separately at its positive and
zero values, and recover the gradient smallness needed for
local regularity.
Throughout the argument, fix
\begin{equation}\label{eq:fixed-p}
p=\frac{19}{12},\qquad
p'=\frac{19}{7},\qquad
\gamma=5-2p=\frac{11}{6}.
\end{equation}
Here $p>3/2$ permits the regularity criterion, whereas $p<5/3$
permits strong convergence by interpolation with the entropy gain.
The three main ingredients are as follows.

\smallskip
\noindent\textbf{(1) Terminal signal structure and strong
convergence of the rescaled signal.}
The first ingredient is the construction of a terminal signal
representative at every spatial point, together with strong
convergence of all backward parabolic rescalings.
{\bf A key observation is
\begin{equation}\label{eq:strong conver}
c_r(s,y):=c(T+r^2s,x_0+ry)
\longrightarrow c_*(T,x_0)
\quad\text{strongly in }L^a(Q_1),
\qquad 1\le a<\infty,
\end{equation}
for every fixed $x_0\in\R^3$ as $r\downarrow0$.}
This convergence supplies the compactness needed to pass to
the consumption term in the normalized signal equation.

The critical bound
$u\in L^\infty(0,T;L^3(\R^3))$
permits the construction of a positive fundamental solution
$\Gamma$ for $\partial_t-\Delta+u\cdot\nabla$, satisfying the
Gaussian bounds \eqref{eq:Gaussian}.
With $t_b=T/2$, define
\[
\begin{aligned}
H_0(t,x)
&:=\mathcal U(t,t_b)c(t_b)(x),\\
V_0(t,x)
&:=\int_{t_b}^{t}\int_{\R^3}
\Gamma(t,x;s,y)n(s,y)c(s,y)\,dy\,ds,
\qquad t_b<t\le T.
\end{aligned}
\]
At $t=T$, the second expression is initially understood as a
nonnegative extended-valued integral.
For $t<T$, the Duhamel formula and the maximum principle give
\[
c(t,x)=H_0(t,x)-V_0(t,x),
\qquad
0\le V_0(t,x)\le H_0(t,x)\le A.
\]
The continuity of $H_0$ and Fatou's lemma imply
\[
0\le V_0(T,x)
\le\liminf_{t\uparrow T}V_0(t,x)
\le H_0(T,x)\le A
\qquad\text{for every }x\in\R^3.
\]
Thus the terminal value
\[
c_*(T,x):=H_0(T,x)-V_0(T,x)\in[0,A]
\]
is well defined everywhere.
Since $V_0$ is lower semicontinuous, extending $c$ by this
terminal value gives an upper semicontinuous function on
$(0,T]\times\R^3$; see Proposition~\ref{prop:terminal}.

We next explain the strong convergence \eqref{eq:strong conver}.
Fix $x_0$ and write $c_\sharp=c_*(T,x_0)$.
The semigroup identity gives the exact terminal-tail formula
\[
\begin{aligned}
&\mathcal U(T,T-h)c(T-h)(x_0)-c_\sharp\\
&\qquad=
\int_{T-h}^{T}\int_{\R^3}
\Gamma(T,x_0;s,y)n(s,y)c(s,y)\,dy\,ds
\longrightarrow0
\end{aligned}
\]
as $h\downarrow0$, because $V_0(T,x_0)<\infty$.
To control the absolute difference, upper semicontinuity gives,
for every $\varepsilon>0$, a sufficiently small $d>0$ such that
\[
c(t,y)\le c_\sharp+\varepsilon
\qquad\text{on }
(T-d^2,T)\times B_d(x_0).
\]
The Gaussian upper bound therefore yields
\[
\mathcal U(T,T-h)
\bigl(c(T-h)-c_\sharp\bigr)_+(x_0)
\le\varepsilon+CAe^{-d^2/(Ch)},
\qquad 0<h<d^2.
\]
Using the unit-mass property and $|z|=2z_+-z$, we obtain
\[
\begin{aligned}
\omega(h)
&:=\mathcal U(T,T-h)|c(T-h)-c_\sharp|(x_0)\\
&=2\mathcal U(T,T-h)
\bigl(c(T-h)-c_\sharp\bigr)_+(x_0)\\
&\quad-
\bigl[\mathcal U(T,T-h)c(T-h)(x_0)-c_\sharp\bigr]
\longrightarrow0,
\end{aligned}
\]
which is the weighted convergence
\eqref{eq:weighted-convergence}.
The Gaussian lower bound converts this weighted estimate
into convergence on backward cylinders.
Indeed, for $0<\sigma<1$,
\[
\Gamma(T,x_0;T-h,x)\ge C_\sigma^{-1}r^{-3}
\quad\text{if }\quad
\sigma r^2\le h\le r^2,\qquad |x-x_0|\le r.
\]
Splitting $Q_1$ into $(-1,-\sigma)\times B_1$ and
$(-\sigma,0)\times B_1$, we obtain
\[
\|c_r-c_\sharp\|_{L^1(Q_1)}
\le
C_\sigma\sup_{0<h\le r^2}\omega(h)
+A|B_1|\sigma.
\]
Letting $r\downarrow0$ and then $\sigma\downarrow0$ gives
strong $L^1$ convergence.
Since $0\le c_r,c_\sharp\le A$, it follows that
\[
\|c_r-c_\sharp\|_{L^a(Q_1)}^a
\le (1+A)^{a-1}\|c_r-c_\sharp\|_{L^1(Q_1)}
\longrightarrow0,
\qquad 1\le a<\infty.
\]
The convergence holds along the full limit $r\downarrow0$
at every prescribed center.
Together with upper semicontinuity, it also gives
\[
\limsup_{\substack{(t,x)\to(T,x_0)\\t<T}}c(t,x)
=c_*(T,x_0).
\]
In particular, choosing $a=p'$ provides exactly the strong
convergence required to pass to products with densities
converging only weakly in $L^p$.
This is the key use of the terminal representation in the
positive-signal argument below.

At a zero terminal value, upper semicontinuity supplies a
different estimate.
If $c_*(T,x_0)=0$, there is a fixed $R>0$ such that
$0\le c\le1$ in $Q_R(T,x_0)$.
Testing the density equation with
$p n^{p-1}\varphi(c)$ times a cutoff, where
$\varphi(c)=(\cos c)^{-(p-1)}$, and using the signal equation
yield \eqref{eq:cos-identity} and
\eqref{eq:small-energy}.
Parabolic Sobolev interpolation then gives
\[
\Acal_p(r;T,x_0)
\le
C\left(
1+\|u\|_{L^\infty(0,T;L^3(\R^3))}^2
\right)
\Acal_p(R;T,x_0)
\left(\frac rR\right)^{1/6}
\longrightarrow0
\]
for fixed $R$ and $r\downarrow0$;
see \eqref{eq:zero-rate}.

\smallskip
\noindent\textbf{(2) Uniform normalized entropy and density
decay at positive terminal-signal points.}
Fix $x_0$ with $c_*(T,x_0)>0$ and use the strong signal
convergence \eqref{eq:strong conver} established above.
Let
\[
\begin{aligned}
(n_r,c_r,u_r)(s,y)
&=(r^2n,c,ru)(T+r^2s,x_0+ry),\\
\lambda_r
&=\max\{1,\Acal_p(r;T,x_0)^{1/p}\},
\qquad N_r=n_r/\lambda_r.
\end{aligned}
\]
Then $\iint_{Q_1}N_r^p\,dy\,ds\le1$, and
\[
\begin{cases}
(N_r)_s+u_r\cdot\nabla N_r
=\Delta N_r-\nabla\cdot(N_r\nabla c_r),\\
(c_r)_s+u_r\cdot\nabla c_r
=\Delta c_r-\lambda_rN_rc_r,
\end{cases}
\qquad \nabla\cdot u_r=0.
\]
The normalization bounds the density in $L^p(Q_1)$,
but introduces the possibly unbounded consumption
coefficient $\lambda_r$.
The estimate must therefore remain uniform as
$\lambda_r\to\infty$.
Proposition~\ref{prop:uniform-entropy} achieves this by coupling
the density entropy to the signal Fisher energy with weight
$\lambda_r^{-1}$:
\[
\int_{B_1}
\left(
N_r\log N_r-N_r+1+
\frac{|\nabla c_r|^2}{2\lambda_r(1+c_r)}
\right)\psi\,dy.
\]
The cancellation in \eqref{eq:normalized-identity} leaves
$(1+c_r)^{-1}\nabla N_r\cdot\nabla c_r$,
which is absorbed by
\[
\frac{|\nabla N_r||\nabla c_r|}{1+c_r}
\le
\frac{|\nabla N_r|^2}{2N_r}
+\frac{N_r|\nabla c_r|^2}{2(1+c_r)^2}.
\]
The surviving mixed dissipation
$N_r|\nabla c_r|^2/(1+c_r)$ has no factor
$\lambda_r^{-1}$.
For the localized entropy and dissipations $Y_\tau$
defined in \eqref{eq:Ytau}, localization gives
\[
Y_\tau\le C_0+
C_1\left(
\|u_r\|_{L^5(Q_1)}^2+
\|\nabla u_r\|_{L^{5/2}(Q_1)}
\right)Y_\tau,
\qquad -1<\tau<0.
\]
The critical fluid norms tend to zero by
\eqref{eq:uniform-fluid-tail}.
Consequently,
\[
\sup_{-1<\tau<0}Y_\tau\le2C_0,
\qquad
\sup_{0<r<r_0}
\iint_{Q_{1/2}}
\left(
|\nabla\sqrt{N_r}|^2+N_r^{5/3}
\right)\,dy\,ds\le C.
\]
Given any $r_k\downarrow0$, weak compactness gives
$N_{r_k}\rightharpoonup N\ge0$ in $L^p(Q_1)$ along a
subsequence.
The first ingredient gives
$c_{r_k}\to c_*(T,x_0)$ strongly in $L^{p'}(Q_1)$.
For $0\le\varphi\in C_c^\infty(Q_1)$,
testing the signal equation yields
\[
\begin{aligned}
0\le
\iint_{Q_1}\lambda_{r_k}N_{r_k}c_{r_k}\varphi\,dy\,ds
&=
\iint_{Q_1}
(c_{r_k}-c_*(T,x_0))
(\varphi_s+\Delta\varphi)\,dy\,ds\\
&\quad+
\iint_{Q_1}
c_{r_k}u_{r_k}\cdot\nabla\varphi\,dy\,ds
\longrightarrow0.
\end{aligned}
\]
Since $\lambda_{r_k}\ge1$, the same limit holds with
$\lambda_{r_k}$ omitted from the left-hand side.
Strong signal convergence and weak density convergence
therefore imply
\[
c_*(T,x_0)\iint_{Q_1}N\varphi\,dy\,ds=0.
\]
Thus $N=0$. {\bf A new observation: since $N_{r_k}\geq 0$, testing its weak convergence against the constant function $1$ gives
strong $L^1(Q_1)$ convergence to zero, i.e.,
$\|N_{r_k}\|_{L^1(Q_1)}\to0$.} Then interpolation with the
uniform entropy gain yields
\[
\|N_{r_k}\|_{L^p(Q_{1/2})}
\le
\|N_{r_k}\|_{L^1(Q_{1/2})}^{3/38}
\|N_{r_k}\|_{L^{5/3}(Q_{1/2})}^{35/38}
\longrightarrow0.
\]
Since the initial sequence was arbitrary, the convergence
holds along the full family.

Finally, \eqref{eq:exact-recurrence} gives
\[
\frac{\Acal_p(r/2;T,x_0)}
{\max\{1,\Acal_p(r;T,x_0)\}}
=
2^\gamma\iint_{Q_{1/2}}N_r^p\,dy\,ds
\longrightarrow0.
\]
Once this ratio is at most $1/2$, dyadic iteration first
bounds $\Acal_p$ along dyadic radii.
The vanishing ratio then gives decay along those radii,
and cylinder inclusion yields
\[
\Acal_p(r;T,x_0)\longrightarrow0
\qquad\text{as }r\downarrow0.
\]

\smallskip
\noindent\textbf{(3) A new regularity criterion with no independent
signal-gradient smallness assumption.}
Theorem~\ref{thm:epsilon} holds for $3/2<p\le5/3$ and gives
\[
\Acal_p(R;z_0)+U(R;z_0)+P(R;z_0)\le\varepsilon_*
\quad\Longrightarrow\quad
\sup_{Q_{R/2}(z_0)}
\bigl(n+|\nabla c|^2+|u|^2\bigr)
\le CR^{-2}.
\]
The joint local entropy argument builds on
Chen--Li--Wang \cite[Lemma~15]{ChenLiWang2025}.
The weighted mass estimate \eqref{eq:weighted-mass} controls
negative entropy on the full test support, while the pressure
estimate \eqref{eq:pressure-flux} closes the backward-kernel
induction \eqref{eq:induction}.
To remove the independent gradient term, scale to $Q_1$ and
decompose $c=a+w$ on $Q_{3/4}$, where
\[
a_t-\Delta a=0,\qquad
w_t-\Delta w=-nc-u\cdot\nabla c,\qquad
w|_{\partial_pQ_{3/4}}=0.
\]
Small $\|u\|_{L^5(Q_1)}$ permits absorption of the critical
drift in the maximal $L^p$ estimate.
Combining this estimate with interpolation for $w$ and the
interior gradient bound for $a$ gives, for
$\|n\|_{L^p(Q_1)}\le1$ and $0<\kappa\le1/4$,
\[
\begin{aligned}
&\kappa^{2p-5}\iint_{Q_\kappa}
\left(
n^p+|\nabla\sqrt{1+c}|^{2p}
\right)\,dx\,dt\\
&\qquad\le
C_{A,p}\kappa^{2p-5}
\left(
\|n\|_{L^p(Q_1)}^p+\|u\|_{L^5(Q_1)}^p
\right)
+C_pA^{2p}\kappa^{2p}.
\end{aligned}
\]
One first fixes $\kappa$ to control the caloric contribution,
then chooses the density and fluid terms small enough to apply
the joint criterion.
All estimates are obtained below a strict top time,
with constants independent of that time.
Finally, the density decay above and
\eqref{eq:uniform-fluid-tail} imply local bounds at every
terminal point.
A fixed-radius spatial-tail estimate and a finite covering give
\eqref{eq:global-top-bound}, allowing a restart strictly before
$T$ with a common lifespan.

The rest of the article is organized as follows: Section~\ref{sec:energy}
states local well-posedness and proves the global budgets; the local construction is in Appendix~\ref{app:local}.
Section~\ref{sec:epsilon} proves the local regularity criterion.
Section~\ref{sec:terminal-normalization} develops the terminal
signal and the uniform normalized entropy, and
Section~\ref{sec:concentration} treats the positive- and zero-signal
branches. Section~\ref{sec:global} proves the main theorem, and
Section~\ref{ks:section} gives the fluid-free and vanishing-damping
consequences. 

Write \(\|\cdot\|_q\) for spatial \(L^q(\R^3)\) norms when
the domain is clear. A cylinder norm without separate space and
time indices is a space-time norm, and \(Q_r=Q_r(0,0)\).
Integrals without a specified domain are over \(\R^3\), or over
the support of the displayed test in a local calculation.
Constants \(C\) may change from line to
line; dependence on fixed parameters is indicated when needed.

\section{Local theory and global energy estimates}\label{sec:energy}

\subsection{Local construction and continuation}\label{sec:local}



\begin{Lem}\label{lem:local}
Let $(n_0,c_0,u_0)\in X$ defined in \eqref{eq:Xspace}. Then \eqref{eq:system} admits a unique bounded mild solution
in the sense of Definition~\ref{def:mild} in $[0,T_{\max})$.
Moreover, 
\begin{equation}\label{eq:continuation}
T_{\max}<\infty
\quad\Longrightarrow\quad
\limsup_{t\uparrow T_{\max}}\|(n,c,u)(t)\|_X=\infty.
\end{equation}
If \eqref{eq:data-smooth} and \eqref{eq:data-fluid} hold,
the solution is classical.
\end{Lem}

The proof is given in Appendix~\ref{app:local}.
For nonnegative initial data, the maximum principle and the
mild formulation yield \eqref{eq:mass-max} for $0\le t<T_{\max}$.
If $n_0\not\equiv0$, the strong maximum principle further gives
$n(t,x)>0$ for all $x\in\R^3$ and $0<t<T_{\max}$.

\subsection{A priori energy estimates}

\begin{Lem}\label{lem:energy}
Assume \eqref{eq:data-smooth}--\eqref{eq:data-fluid}.
For every $T>0$, there exists $C_T>0$, depending only on
the initial data $(n_0,c_0,u_0)$, $F$, and $T$, such that the mild  solution of \eqref{eq:system} satisfies
\begin{equation}\label{eq:energy-budget}
\begin{aligned}
&\sup_{0<t<\min\{T,T_{\max}\}}
\int_{\R^3}\left[
 n\log_+n+\frac{|\nabla c|^2}{2(c+1)}
 +|u|^2+\langle x\rangle n\right]\dd x\\
&\quad+\int_0^{\min\{T,T_{\max}\}}\int_{\R^3}
\left[\frac{|\nabla n|^2}{n}
 +\left(c+1\right)
 \left|\nabla^2\log\left(c+1\right)\right|^2
 +\frac{n|\nabla c|^2}{c+1}
 +|\nabla u|^2\right]\dd x\,\dd t
\le C_T
\end{aligned}
\end{equation}
and
\begin{equation}\label{eq:global-signal-norms}
\begin{aligned}
\sup_{0<t<\min\{T,T_{\max}\}}\|\nabla c(t)\|_2^2+\int_0^{\min\{T,T_{\max}\}}
\left(\|\nabla^2c\|_2^2+\|\nabla c\|_4^4
 +\int_{\R^3}n|\nabla c|^2\,\dd x\right)\dd t
\le C_T,
\end{aligned}
\end{equation}
where $\langle\cdot\rangle=(1+(\cdot)^2)^\frac12$.
\end{Lem}

\begin{proof}
Fix $0<s<\tau<T_{\max}$. We first derive the estimates for
positive density on $[s,\tau]$. The spatial cutoff limit,
the passage $s\rightarrow0$, and the removal of the positivity
assumption are justified in {\bf Step 6.}

\smallskip
\noindent\textit{\bf Step 1: The density entropy.}
Multiplying $\eqref{eq:system}_1$ by $1+\log n$ and
integrating over $\R^3$, we obtain
\begin{align}
\frac{\dd}{\dd t}\int_{\R^3}n\log n\,\dd x
={}&\int_{\R^3}(1+\log n)\Delta n\,\dd x
 -\int_{\R^3}(1+\log n)\diver(n\nabla c)\,\dd x
 \nonumber\\
&-\int_{\R^3}(1+\log n)u\cdot\nabla n\,\dd x
 =:I_1+I_2+I_3.\label{eq:I0}
\end{align}
Integration by parts gives
\begin{align}
I_1&=-\int_{\R^3}\frac{|\nabla n|^2}{n}\,\dd x,
\label{eq:I1}\\
I_2&=\int_{\R^3}n\nabla c\cdot\nabla\log n\,\dd x
 =\int_{\R^3}\nabla n\cdot\nabla c\,\dd x.
\label{eq:I2}
\end{align}
Since $\diver u=0$, we also have
\begin{equation}\label{eq:I3}
I_3=-\int_{\R^3}u\cdot\nabla(n\log n)\,\dd x=0.
\end{equation}
Combining \eqref{eq:I0}--\eqref{eq:I3}, we obtain
\begin{equation}\label{eq:density-entropy-global}
\frac{\dd}{\dd t}\int_{\R^3}n\log n\,\dd x
+\int_{\R^3}\frac{|\nabla n|^2}{n}\,\dd x
=\int_{\R^3}\nabla n\cdot\nabla c\,\dd x.
\end{equation}

\smallskip
\noindent\textit{\bf Step 2: The Fisher identity.}
Write $L_u:=\partial_t-\Delta+u\cdot\nabla$.
Differentiating $L_uc=-nc$ in $x_i$ gives
\[
L_u(\partial_i c)
=-\partial_i(nc)-\sum_{j=1}^3\partial_i u_j\,\partial_j c.
\]
Multiplying by $2\partial_i c$ and summing over $i$, we obtain
\begin{equation}\label{eq:Fisher-gradient}
L_u|\nabla c|^2
=-2\nabla c\cdot\nabla(nc)-2|\nabla^2c|^2
 -2(\nabla c\otimes\nabla c):\nabla u.
\end{equation}
The chain rule also gives
\begin{equation}\label{eq:Fisher-reciprocal}
L_u\!\left(c+1\right)^{-1}
=\frac{nc}{(c+1)^2}
 -\frac{2|\nabla c|^2}{(c+1)^3}.
\end{equation}
Applying the product rule to $|\nabla c|^2/[2(c+1)]$
and using \eqref{eq:Fisher-gradient}--\eqref{eq:Fisher-reciprocal},
we have
\[
\begin{aligned}
L_u\!\left(\frac{|\nabla c|^2}{2(c+1)}\right)
={}&-\frac{|\nabla^2c|^2}{c+1}
 +\frac{2\nabla^2c:(\nabla c\otimes\nabla c)}{(c+1)^2}
 -\frac{|\nabla c|^4}{(c+1)^3}\\
&-\frac{\nabla c\cdot\nabla(nc)}{c+1}
 +\frac{nc|\nabla c|^2}{2(c+1)^2}
 -\frac{\nabla c\otimes\nabla c}{c+1}:\nabla u.
\end{aligned}
\]
Since
\[
\nabla^2\log\left(c+1\right)
=\frac{\nabla^2c}{c+1}
 -\frac{\nabla c\otimes\nabla c}{(c+1)^2},
\]
it follows that
\begin{equation}\label{eq:Fisher-identity}
\begin{aligned}
&L_u\!\left(\frac{|\nabla c|^2}{2(c+1)}\right)
 +\left(c+1\right)
 \left|\nabla^2\log\left(c+1\right)\right|^2\\
&\qquad=-\frac{\nabla c\cdot\nabla(nc)}{c+1}
 +\frac{nc|\nabla c|^2}{2(c+1)^2}
 -\frac{\nabla c\otimes\nabla c}{c+1}:\nabla u.
\end{aligned}
\end{equation}
The consumption terms satisfy
\begin{equation}\label{eq:Fisher-consumption}
\begin{aligned}
-\frac{\nabla c\cdot\nabla(nc)}{c+1}
 +\frac{nc|\nabla c|^2}{2(c+1)^2}=-\frac{c}{c+1}\nabla n\cdot\nabla c
 -\frac{n|\nabla c|^2}{2(c+1)}
 -\frac{n|\nabla c|^2}{2(c+1)^2}.
\end{aligned}
\end{equation}
Substituting \eqref{eq:Fisher-consumption} into
\eqref{eq:Fisher-identity}, integrating over $\R^3$, and
using $\diver u=0$, we obtain
\begin{equation}\label{eq:signal-Fisher-global}
\begin{aligned}
&\frac{1}{2}\frac{\dd}{\dd t}
 \int_{\R^3}\frac{|\nabla c|^2}{c+1}\,\dd x
 +\int_{\R^3}\left(c+1\right)
 \left|\nabla^2\log\left(c+1\right)\right|^2\dd x\\
&\quad+\frac{1}{2}\int_{\R^3}\frac{n|\nabla c|^2}{c+1}\,\dd x
 +\frac{1}{2}\int_{\R^3}\frac{n|\nabla c|^2}{(c+1)^2}\,\dd x\\
&=-\int_{\R^3}\frac{c}{c+1}
 \nabla n\cdot\nabla c\,\dd x
 -\int_{\R^3}\frac{\nabla c\otimes\nabla c}{c+1}
 :\nabla u\,\dd x.
\end{aligned}
\end{equation}

\smallskip
\noindent\textit{\bf Step 3: The strain term.}
Adding \eqref{eq:density-entropy-global} and
\eqref{eq:signal-Fisher-global}, we obtain
\begin{equation}\label{eq:three trems1}
\begin{aligned}
&\frac{\dd}{\dd t}\int_{\R^3}
 \left[n\log n+\frac{|\nabla c|^2}{2(c+1)}\right]\dd x
 +\int_{\R^3}\frac{|\nabla n|^2}{n}\,\dd x\\
&\quad+\int_{\R^3}\left(c+1\right)
 \left|\nabla^2\log\left(c+1\right)\right|^2\dd x
 +\frac{1}{2}\int_{\R^3}\frac{n|\nabla c|^2}{c+1}\,\dd x
 +\frac{1}{2}\int_{\R^3}\frac{n|\nabla c|^2}{(c+1)^2}\,\dd x\\
&=\int_{\R^3}
 \frac{\nabla n\cdot\nabla c}{c+1}\,\dd x
 -\int_{\R^3}\frac{\nabla c\otimes\nabla c}{c+1}
 :\nabla u\,\dd x.
\end{aligned}
\end{equation}
By Young's inequality,
\begin{equation}\label{eq:three trems2}
\left|\int_{\R^3}
 \frac{\nabla n\cdot\nabla c}{c+1}\,\dd x\right|
\le\frac{1}{2}\int_{\R^3}\frac{|\nabla n|^2}{n}\,\dd x
 +\frac{1}{2}\int_{\R^3}
 \frac{n|\nabla c|^2}{(c+1)^2}\,\dd x.
\end{equation}
Combining \eqref{eq:three trems1} and \eqref{eq:three trems2}, we obtain
\begin{equation}\label{eq:global-mixed}
\begin{aligned}
&\frac{\dd}{\dd t}\int_{\R^3}
 \left[n\log n+\frac{|\nabla c|^2}{2(c+1)}\right]\dd x
 +\frac{1}{2}\int_{\R^3}\frac{|\nabla n|^2}{n}\,\dd x\\
&\quad+\int_{\R^3}\left(c+1\right)
 \left|\nabla^2\log\left(c+1\right)\right|^2\dd x
 +\frac{1}{2}\int_{\R^3}\frac{n|\nabla c|^2}{c+1}\,\dd x\\
&\le\left|\int_{\R^3}
 \frac{\nabla c\otimes\nabla c}{c+1}:\nabla u\,\dd x\right|.
\end{aligned}
\end{equation}
Using $0\le c\le A$ and H\"older's inequality, we have
\begin{equation}\label{eq:global-strain-holder}
\begin{aligned}
\left|\int_{\R^3}
 \frac{\nabla c\otimes\nabla c}{c+1}:\nabla u\,\dd x\right|
&\le\left(\int_{\R^3}
 \frac{|\nabla c|^4}{(c+1)^3}\,\dd x\right)^{\frac{1}{2}}
 \left(\int_{\R^3}\left(c+1\right)|\nabla u|^2\,\dd x\right)^{\frac{1}{2}}\\
&\le\left(A+1\right)^{\frac{1}{2}}\|\nabla u\|_2
 \left(\int_{\R^3}
 \frac{|\nabla c|^4}{(c+1)^3}\,\dd x\right)^{\frac{1}{2}}.
\end{aligned}
\end{equation}
To estimate the last integral, we use
\begin{equation}\label{eq:log-gradient-divergence}
\begin{aligned}
\frac{|\nabla c|^4}{(c+1)^3}
=\diver\!\left(\frac{|\nabla c|^2\nabla c}{(c+1)^2}\right)
 -\frac{2\nabla^2\log(c+1):(\nabla c\otimes\nabla c)}{c+1}-\frac{|\nabla c|^2\Delta\log(c+1)}{c+1}.
\end{aligned}
\end{equation}
Integrating over $\R^3$ and using
$|\Delta\log(c+1)|\le\sqrt3|\nabla^2\log(c+1)|$,
we obtain
\[
\begin{aligned}
\int_{\R^3}\frac{|\nabla c|^4}{(c+1)^3}\,\dd x
&\le(2+\sqrt3)\int_{\R^3}
 \left|\nabla^2\log\left(c+1\right)\right|
 \frac{|\nabla c|^2}{c+1}\,\dd x\\
&\le(2+\sqrt3)
 \left(\int_{\R^3}\left(c+1\right)
 \left|\nabla^2\log\left(c+1\right)\right|^2\dd x\right)^{\frac{1}{2}}
 \left(\int_{\R^3}\frac{|\nabla c|^4}{(c+1)^3}\,\dd x\right)^{\frac{1}{2}}.
\end{aligned}
\]
Consequently,
\begin{equation}\label{eq:global-quartic}
\int_{\R^3}\frac{|\nabla c|^4}{(c+1)^3}\,\dd x
\le(2+\sqrt3)^2\int_{\R^3}\left(c+1\right)
 \left|\nabla^2\log\left(c+1\right)\right|^2\dd x.
\end{equation}
Substituting \eqref{eq:global-quartic} into
\eqref{eq:global-strain-holder} and using Young's inequality gives
\[
\begin{aligned}
\left|\int_{\R^3}
 \frac{\nabla c\otimes\nabla c}{c+1}:\nabla u\,\dd x\right|
&\le\frac{1}{2}\int_{\R^3}\left(c+1\right)
 \left|\nabla^2\log\left(c+1\right)\right|^2\dd x\\
&\quad+\frac{(2+\sqrt3)^2}{2}
 \left(A+1\right)\|\nabla u\|_2^2.
\end{aligned}
\]
Combining this with \eqref{eq:global-mixed}, we obtain
\begin{equation}\label{eq:entropy-strain}
\begin{aligned}
&\frac{\dd}{\dd t}\int_{\R^3}
 \left[n\log n+\frac{|\nabla c|^2}{2(c+1)}\right]\dd x
 +\frac{1}{2}\int_{\R^3}\frac{|\nabla n|^2}{n}\,\dd x\\
&\quad+\frac{1}{2}\int_{\R^3}\left(c+1\right)
 \left|\nabla^2\log\left(c+1\right)\right|^2\dd x
 +\frac{1}{2}\int_{\R^3}\frac{n|\nabla c|^2}{c+1}\,\dd x\\
&\le C\left(A+1\right)\|\nabla u\|_2^2,
\end{aligned}
\end{equation}
where $C>0$ is an absolute constant.

\smallskip
\noindent\textit{\bf Step 4: Kinetic energy and the density moment.}
By mass conservation, the Gagliardo--Nirenberg inequality,
and the Sobolev inequality, we have
\begin{equation}\label{eq:force-interpolation}
\begin{aligned}
\|n\|_{\frac{6}{5}}
&=\|\sqrt n\|_{\frac{12}{5}}^2
 \le C\|\sqrt n\|_2^{\frac{3}{2}}\|\nabla\sqrt n\|_2^{\frac{1}{2}}
 \le CM^{\frac{3}{4}}
 \left(\int_{\R^3}\frac{|\nabla n|^2}{n}\,\dd x\right)^{\frac{1}{4}},\\
\|u\|_6&\le C\|\nabla u\|_2.
\end{aligned}
\end{equation}
Taking the $L^2$ inner product of $\eqref{eq:system}_3$ with
$u$ and using $\diver u=0$, we obtain
\[
\frac{1}{2}\frac{\dd}{\dd t}\|u\|_2^2+\|\nabla u\|_2^2
=\int_{\R^3}nF\cdot u\,\dd x.
\]
By H\"older's inequality, \eqref{eq:force-interpolation},
and Young's inequality, for every $\epsilon>0$,
\[
\begin{aligned}
\left|\int_{\R^3}nF\cdot u\,\dd x\right|
&\le C\|F\|_\infty M^{\frac{3}{4}}
 \left(\int_{\R^3}\frac{|\nabla n|^2}{n}\,\dd x\right)^{\frac{1}{4}}
 \|\nabla u\|_2\\
&\le\frac{1}{4}\|\nabla u\|_2^2
 +C\|F\|_\infty^2M^{\frac{3}{2}}
 \left(\int_{\R^3}\frac{|\nabla n|^2}{n}\,\dd x\right)^{\frac{1}{2}}\\
&\le\frac{1}{4}\|\nabla u\|_2^2
 +\epsilon\int_{\R^3}\frac{|\nabla n|^2}{n}\,\dd x
 +C_{\epsilon,M,\|F\|_\infty}.
\end{aligned}
\]
Consequently,
\begin{equation}\label{eq:kinetic}
\frac{1}{2}\frac{\dd}{\dd t}\|u\|_2^2
+\frac{3}{4}\|\nabla u\|_2^2
\le\epsilon\int_{\R^3}\frac{|\nabla n|^2}{n}\,\dd x
 +C_{\epsilon,M,\|F\|_\infty}.
\end{equation}
Multiplying $\eqref{eq:system}_1$ by $\langle x\rangle$
and integrating by parts gives
\begin{equation}\label{eq:moment-identity}
\frac{\dd}{\dd t}\int_{\R^3}\langle x\rangle n\,\dd x
=\int_{\R^3}n\Delta\langle x\rangle\,\dd x
 +\int_{\R^3}n(\nabla c+u)\cdot\nabla\langle x\rangle\,\dd x.
\end{equation}
Since $|\nabla\langle x\rangle|\le1$ and
$0\le\Delta\langle x\rangle\le3$, we have
\[
\frac{\dd}{\dd t}\int_{\R^3}\langle x\rangle n\,\dd x
\le3M+\int_{\R^3}n|\nabla c|\,\dd x
 +\int_{\R^3}n|u|\,\dd x.
\]
Using $0\le c\le A$ and Young's inequality, for every
$\epsilon_1>0$ we obtain
\[
\begin{aligned}
\int_{\R^3}n|\nabla c|\,\dd x
\le\left[M\left(A+1\right)\right]^{\frac{1}{2}}
 \left(\int_{\R^3}\frac{n|\nabla c|^2}{c+1}\,\dd x\right)^{\frac{1}{2}}\le\epsilon_1\int_{\R^3}
 \frac{n|\nabla c|^2}{c+1}\,\dd x
 +\frac{M(A+1)}{4\epsilon_1}.
\end{aligned}
\]
Moreover, \eqref{eq:force-interpolation} and Young's
inequality give, for every $\epsilon_2,\epsilon_3>0$,
\[
\begin{aligned}
\int_{\R^3}n|u|\,\dd x
\le CM^{\frac{3}{4}}
 \left(\int_{\R^3}\frac{|\nabla n|^2}{n}\,\dd x\right)^{\frac{1}{4}}
 \|\nabla u\|_2\le\epsilon_2\int_{\R^3}\frac{|\nabla n|^2}{n}\,\dd x
 +\epsilon_3\|\nabla u\|_2^2+C_{M,\epsilon_2,\epsilon_3}.
\end{aligned}
\]
Thus
\begin{equation}\label{eq:moment}
\begin{aligned}
\frac{\dd}{\dd t}\int_{\R^3}\langle x\rangle n\,\dd x
\le\epsilon_1\int_{\R^3}\frac{n|\nabla c|^2}{c+1}\,\dd x
 +\epsilon_2\int_{\R^3}\frac{|\nabla n|^2}{n}\,\dd x+\epsilon_3\|\nabla u\|_2^2
 +C_{A,M,\epsilon_1,\epsilon_2,\epsilon_3}.
\end{aligned}
\end{equation}

\smallskip
\noindent\textit{\bf Step 5: The combined energy estimate.}
For every $z\ge0$, $z\log z+\langle x\rangle z\ge-e^{-\langle x\rangle-1}$. Applying this inequality with $z=n$ and integrating gives
\begin{equation}\label{eq:negative-global}
\begin{aligned}
\int_{\R^3}(n\log n+\langle x\rangle n)\,\dd x
&\ge-C_w,\\
\int_{\R^3}(n\log n)_-\,\dd x
&\le\int_{\R^3}\langle x\rangle n\,\dd x+C_w,
\end{aligned}
\end{equation}
where $C_w=\int_{\R^3}e^{-\langle x\rangle-1}\,\dd x<\infty$.
Choose $B=B(A)\ge2$ such that
$3B/4-C(A+1)\ge1$, where $C$ is the constant in
\eqref{eq:entropy-strain}. Add $B$ times \eqref{eq:kinetic}
and twice \eqref{eq:moment} to \eqref{eq:entropy-strain}.
Taking $\epsilon=1/(8B)$, $\epsilon_1=\epsilon_2=1/16$,
and $\epsilon_3=1/4$, we obtain
\begin{equation}\label{eq:closed-energy}
\begin{aligned}
&\frac{\dd}{\dd t}\int_{\R^3}
 \left[n\log n+\frac{|\nabla c|^2}{2(c+1)}
 +\frac{B}{2}|u|^2+2\langle x\rangle n\right]\dd x\\
&\quad+\frac{1}{4}\int_{\R^3}
 \left[\frac{|\nabla n|^2}{n}
 +\left(c+1\right)
 \left|\nabla^2\log\left(c+1\right)\right|^2
 +\frac{n|\nabla c|^2}{c+1}
 +|\nabla u|^2\right]\dd x
\le C_{A,M,\|F\|_\infty}.
\end{aligned}
\end{equation}
By \eqref{eq:negative-global},
\[
\int_{\R^3}(n\log n+2\langle x\rangle n)\,\dd x+C_w
\ge\int_{\R^3}(n\log_+n+\langle x\rangle n)\,\dd x.
\]
The initial quantities are finite:
\[
\begin{aligned}
\int_{\R^3}n_0\log_+n_0\,\dd x
&\le M\log_+\|n_0\|_\infty,\\
\frac{1}{2}\int_{\R^3}\frac{|\nabla c_0|^2}{c_0+1}\,\dd x
&\le2\int_{\R^3}|\nabla\sqrt{c_0}|^2\,\dd x=I_0,\\
\int_{\R^3}\langle x\rangle n_0\,\dd x
&\le M+M^{\frac{1}{2}}
 \left(\int_{\R^3}|x|^2n_0\,\dd x\right)^{\frac{1}{2}}.
\end{aligned}
\]
Integrating \eqref{eq:closed-energy} in time and using
$B\ge2$ therefore gives \eqref{eq:energy-budget}.

To prove \eqref{eq:global-signal-norms}, we use
\begin{equation}\label{eq:signal-Hessian-comparison}
\begin{aligned}
\nabla^2c
&=\left(c+1\right)\nabla^2\log\left(c+1\right)
 +\frac{\nabla c\otimes\nabla c}{c+1},\\
\frac{|\nabla^2c|^2}{c+1}
&\le2\left(c+1\right)
 \left|\nabla^2\log\left(c+1\right)\right|^2
 +\frac{2|\nabla c|^4}{(c+1)^3}.
\end{aligned}
\end{equation}
Combining \eqref{eq:global-quartic},
\eqref{eq:signal-Hessian-comparison}, and $0\le c\le A$,
we obtain
\[
\|\nabla^2c\|_2^2+\|\nabla c\|_4^4
\le C_A\int_{\R^3}\left(c+1\right)
 \left|\nabla^2\log\left(c+1\right)\right|^2\dd x.
\]
Moreover,
\[
\begin{aligned}
\|\nabla c\|_2^2
&\le\left(A+1\right)
 \int_{\R^3}\frac{|\nabla c|^2}{c+1}\,\dd x,\\
\int_{\R^3}n|\nabla c|^2\,\dd x
&\le\left(A+1\right)
 \int_{\R^3}\frac{n|\nabla c|^2}{c+1}\,\dd x,
\end{aligned}
\]
which and \eqref{eq:energy-budget} prove
\eqref{eq:global-signal-norms}.

\smallskip
\noindent\textit{\bf Step 6: Justification of the preceding calculations.}
Fix $0<\tau<T_{\max}$. By Lemma~\ref{lem:local}, we have
\[
\sup_{0\le t\le\tau}
\left(\|n(t)\|_\infty+\|u(t)\|_\infty
+\|\nabla c(t)\|_\infty\right)<\infty.
\]
We first assume that $n_0>0$.
Choose $0\le\chi\le1$ in $C_c^\infty(B_2)$ with
$\chi=1$ on $B_1$, and let $\chi_R(x)=\chi(x/R)$.
Testing $\eqref{eq:system}_1$ with $\langle x\rangle\chi_R$,
using $|\nabla(\langle x\rangle\chi_R)|
+|\Delta(\langle x\rangle\chi_R)|\le C$, and letting
$R\to\infty$, we obtain
\[
\sup_{0\le t\le\tau}\int_{\mathbb R^3}\langle x\rangle n(t)<\infty.
\]
Together with \eqref{eq:negative-global} and the boundedness
of $n$, this gives
\[
\sup_{0\le t\le\tau}
\int_{\mathbb R^3}n(1+|\log n|)<\infty.
\]
Testing next with $(1+\log n)\chi_R$, the cutoff errors are
bounded by
\[
C\left[R^{-2}
+R^{-1}(\|u\|_\infty+\|\nabla c\|_\infty)\right]
\int_{\mathbb R^3}n(1+|\log n|).
\]
They vanish after time integration as $R\to\infty$.
Using
\[
\int_{\mathbb R^3}\chi_R\nabla n\cdot\nabla c
\le\frac12\int_{\mathbb R^3}\chi_R\frac{|\nabla n|^2}{n}
+\frac M2\|\nabla c\|_\infty^2,
\]
we conclude that
\[
\int_0^\tau\int_{\mathbb R^3}\frac{|\nabla n|^2}{n}<\infty.
\]

For the signal $c$, take
$\rho_R(x)=(1+|x|^2/R^2)^{-2}$.
Testing $\eqref{eq:system}_2$ with $-\rho_R\Delta c$,
with an additional compact cutoff removed first, gives
\[
\frac{d}{dt}\int_{\mathbb R^3}\rho_R|\nabla c|^2
+\int_{\mathbb R^3}\rho_R|\Delta c|^2
\le
C A^2\|n\|_2^2
+C(1+\|u\|_\infty^2)
\int_{\mathbb R^3}\rho_R|\nabla c|^2,
\qquad R\ge1.
\]
Since $\|n\|_2^2\le M\|n\|_\infty$ and
$\nabla c_0\in L^2$, Gronwall's inequality and
$\rho_R\uparrow1$ yield
\[
\sup_{0\le t\le\tau}\|\nabla c(t)\|_2^2
+\int_0^\tau\|\nabla^2c\|_2^2\,dt<\infty.
\]
Applying the Leray projection $\mathbb P$ to $\eqref{eq:system}_3$
and taking the $L^2$ inner product with $u$, we obtain
\[
\frac12\frac{d}{dt}\|u\|_2^2+\|\nabla u\|_2^2
=\int_{\mathbb R^3}nF\cdot u
\le\frac12\|u\|_2^2
+\frac12\|F\|_\infty^2\|n\|_2^2,
\]
where we used $\mathbb Pu=u$ and the self-adjointness of
$\mathbb P$. Since $\|n(t)\|_2^2\le M\|n(t)\|_\infty$,
Gronwall's inequality yields
\[
\sup_{0\le t\le\tau}\|u(t)\|_2^2
+\int_0^\tau\|\nabla u\|_2^2\,dt<\infty.
\]
These bounds, $c+1\ge1$, and the boundedness of
$n$ and $\nabla c$ imply the integrability of all terms
in the preceding identities.
In particular, the remaining cutoff errors satisfy
\[
\begin{aligned}
\left|\int_{\mathbb R^3}
\frac{|\nabla c|^2}{2(c+1)}
(\Delta\chi_R+u\cdot\nabla\chi_R)\right|
&\le C(R^{-2}+R^{-1}\|u\|_\infty)\|\nabla c\|_2^2,\\
\left|\int_{\mathbb R^3}
\frac{|\nabla c|^2\nabla c}{(c+1)^2}
\cdot\nabla\chi_R\right|
&\le CR^{-1}\|\nabla c\|_\infty\|\nabla c\|_2^2.
\end{aligned}
\]
Their time integrals vanish as $R\to\infty$.
Thus the spatial integrations by parts are justified.
Note that these auxiliary bounds do not enter the constant in
\eqref{eq:closed-energy}.

For general $n_0\ge0$, let
$n_0^\varepsilon=n_0+\varepsilon e^{-|x|^2}$.
The initial entropy and moment are uniformly bounded, and
local continuous dependence gives, for sufficiently small
$\varepsilon>0$,
\[
(n^\varepsilon,c^\varepsilon,u^\varepsilon)
\longrightarrow(n,c,u)
\quad\text{in }C([0,\tau];X).
\]
Apply the integrated estimate \eqref{eq:closed-energy}
to these solutions. Its right-hand side is uniform in
$\varepsilon$. By \eqref{eq:negative-global}, Fatou's lemma,
and lower semicontinuity of the dissipations, passage to
$\varepsilon\rightarrow0$ gives the same estimate for
$(n,c,u)$. Finally, letting
$\tau\rightarrow\min\{T,T_{\max}\}$ proves
\eqref{eq:energy-budget} and \eqref{eq:global-signal-norms}.
\end{proof}

\begin{Lem}\label{lem:fluid}
Let $0<T\le T_{\max}$ with $T<\infty$. Then
\begin{align}
&n\in L^1((0,T)\times\R^3)
\cap L^{\frac{5}{3}}((0,T)\times\R^3)
\cap L^2(0,T;L^{\frac{3}{2}}(\R^3)),
\label{eq:n-integrable}\\
&u\in L^\infty(0,T;L^3(\R^3))
\cap L^5((0,T)\times\R^3),\qquad
\nabla u\in L^{\frac{5}{2}}((0,T)\times\R^3).
\label{eq:u-integrable}
\end{align}
Moreover, the pressure defined by
\begin{equation}\label{eq:canonical-pressure}
\pi=\Delta^{-1}\nabla\cdot(nF)
\end{equation}
belongs to $L^2(0,T;L^3(\R^3))$, and
\begin{equation}\label{eq:uniform-fluid-tail}
\lim_{r\rightarrow0}\sup_{x_0\in\R^3}
\left(
U(r;T,x_0)
+\|\nabla u\|_{L^{\frac{5}{2}}(Q_r(T,x_0))}
+P(r;T,x_0)
\right)=0,
\end{equation}
where $U$ and $P$ are defined in \eqref{eq:scale-quantities}.
\end{Lem}

\begin{proof}
{\bf The proof of  \eqref{eq:n-integrable}.}
For any time interval $J$ and
$f\in L^\infty(J;L^2(\R^3))\cap L^2(J;H^1(\R^3))$,
interpolation and the Sobolev inequality give
\[
\|f\|_{\frac{10}{3}}^{\frac{10}{3}}
\le\|f\|_2^{\frac{4}{3}}\|f\|_6^2
\le C\|f\|_2^{\frac{4}{3}}\|\nabla f\|_2^2.
\]
Integrating in time, we obtain
\begin{equation}\label{eq:parabolic-Sobolev}
\int_J\int_{\R^3}|f|^{\frac{10}{3}}\,\dd x\,\dd t
\le C
\left(\mathop{\rm ess\,sup}_{t\in J}
\int_{\R^3}|f(t)|^2\,\dd x\right)^{\frac{2}{3}}
\int_J\int_{\R^3}|\nabla f|^2\,\dd x\,\dd t.
\end{equation}
Applying \eqref{eq:parabolic-Sobolev} to $\sqrt n$ and
using mass conservation in \eqref{eq:mass-max}, we have
\begin{equation}\label{eq:density-Sobolev}
\begin{aligned}
\int_0^T\int_{\R^3}n^{\frac{5}{3}}\,\dd x\,\dd t
&\le CM^{\frac{2}{3}}\int_0^T\|\nabla\sqrt n\|_2^2\,\dd t,\\
\|n(t)\|_{\frac{3}{2}}^2
=\|\sqrt n(t)\|_3^4
&\le C\|\sqrt n(t)\|_2^2
\|\nabla\sqrt n(t)\|_2^2
=CM\|\nabla\sqrt n(t)\|_2^2.
\end{aligned}
\end{equation}
By
\[
\begin{aligned}
\int_0^T\int_{\R^3}n\,\dd x\,\dd t&=MT,\\
\int_0^T\|\nabla\sqrt n\|_2^2\,\dd t
&=\frac{1}{4}\int_0^T\int_{\R^3}
\frac{|\nabla n|^2}{n}\,\dd x\,\dd t<\infty,
\end{aligned}
\]
and \eqref{eq:energy-budget},
\eqref{eq:density-Sobolev} proves \eqref{eq:n-integrable}.
Interpolation further gives
$n\in L^p((0,T)\times\R^3)$ for $1\le p\le5/3$.

{\bf The proof of \eqref{eq:u-integrable}.}
Applying the Leray projection to $\eqref{eq:system}_3$ gives
\[
\partial_tu-\Delta u=\PP(nF).
\]
Then, taking the $L^2$ inner product with $(-\Delta)^{1/2}u$,
we obtain
\[
\frac{1}{2}\frac{\dd}{\dd t}\|u\|_{\dot H^{\frac{1}{2}}}^2
+\|u\|_{\dot H^{\frac{3}{2}}}^2
=
\left\langle
(-\Delta)^{-\frac{1}{4}}\PP(nF),(-\Delta)^{\frac{3}{4}}u
\right\rangle.
\]
By the dual Sobolev inequality and the boundedness of $\PP$, we have
\[
\|(-\Delta)^{-\frac{1}{4}}\PP(nF)\|_2
\le C\|nF\|_{\frac{3}{2}}
\le C\|F\|_\infty\|n\|_{\frac{3}{2}}.
\]
Young's inequality therefore yields
\begin{equation}\label{eq:half-energy}
\frac{\dd}{\dd t}\|u\|_{\dot H^{\frac{1}{2}}}^2
+\|u\|_{\dot H^{\frac{3}{2}}}^2
\le C\|F\|_\infty^2\|n\|_{\frac{3}{2}}^2.
\end{equation}
Integrating \eqref{eq:half-energy} and using
\eqref{eq:n-integrable} and $u_0\in H^4(\R^3)$, we obtain
\[
\sup_{0<t<T}\|u(t)\|_{\dot H^{\frac{1}{2}}}^2
+\int_0^T\|u\|_{\dot H^{\frac{3}{2}}}^2\,\dd t
\le 2\|u_0\|_{\dot H^{\frac{1}{2}}}^2
+C\|F\|_\infty^2
\int_0^T\|n\|_{\frac{3}{2}}^2\,\dd t<\infty.
\]
In particular, $\dot H^{1/2}(\R^3)\hookrightarrow L^3(\R^3)$
gives $u\in L^\infty(0,T;L^3(\R^3))$.
Moreover, Sobolev embedding and interpolation give
\[
\begin{aligned}
\|u\|_5
&\le C\|u\|_{\dot H^{\frac{9}{10}}}
\le C\|u\|_{\dot H^{\frac{1}{2}}}^{\frac{3}{5}}
       \|u\|_{\dot H^{\frac{3}{2}}}^{\frac{2}{5}},\\
\|\nabla u\|_{\frac{5}{2}}
&\le C\|u\|_{\dot H^{\frac{13}{10}}}
\le C\|u\|_{\dot H^{\frac{1}{2}}}^{\frac{1}{5}}
       \|u\|_{\dot H^{\frac{3}{2}}}^{\frac{4}{5}}.
\end{aligned}
\]
Hence
\begin{equation}\label{eq:fluid-interpolation}
\begin{aligned}
\|u\|_5^5
&\le C\|u\|_{\dot H^{\frac{1}{2}}}^3
       \|u\|_{\dot H^{\frac{3}{2}}}^2,\\
\|\nabla u\|_{\frac{5}{2}}^{\frac{5}{2}}
&\le C\|u\|_{\dot H^{\frac{1}{2}}}^{\frac{1}{2}}
       \|u\|_{\dot H^{\frac{3}{2}}}^2.
\end{aligned}
\end{equation}
Integrating \eqref{eq:fluid-interpolation} proves
\eqref{eq:u-integrable}.

{\bf The proof of \eqref{eq:uniform-fluid-tail}.}
Taking the divergence of $\eqref{eq:system}_3$ gives
$\Delta\pi=\nabla\cdot(nF)$.
By the Hardy--Littlewood--Sobolev inequality, we have
\begin{equation}\label{eq:pressure-potential}
\|\Delta^{-1}\nabla\cdot f\|_3
\le C\|f\|_{\frac{3}{2}}.
\end{equation}
Applying \eqref{eq:pressure-potential} with $f=nF$ and using
\eqref{eq:n-integrable}, we obtain
\begin{align} \label{eq:pressure estimate}
    \|\pi\|_{L^2(0,T;L^3)}
\le C\|F\|_\infty\|n\|_{L^2(0,T;L^{\frac{3}{2}})}<\infty.
\end{align}
Finally, for $0<r<\sqrt T$, the definitions of $U$ and $P$ in \eqref{eq:scale-quantities}
give
\[
\begin{aligned}
U(r;T,x_0)
&\le\|u\|_{L^5((T-r^2,T)\times\R^3)},\\
\|\nabla u\|_{L^{\frac{5}{2}}(Q_r(T,x_0))}
&\le\|\nabla u\|_{L^{\frac{5}{2}}((T-r^2,T)\times\R^3)},\\
P(r;T,x_0)
&\le2\|\pi\|_{L^2(T-r^2,T;L^3(\R^3))}.
\end{aligned}
\]
The right-hand sides are independent of $x_0$ and tend to
zero as $r\downarrow0$ by \eqref{eq:u-integrable} and \eqref{eq:pressure estimate}. This proves \eqref{eq:uniform-fluid-tail}.
\end{proof}

\section{New local epsilon regularity}\label{sec:epsilon}

In this section we prove Theorem~\ref{thm:epsilon}.
For $z_0=(t_0,x_0)$ and $R>0$, define
\begin{equation}\label{eq:scaling}
\begin{aligned}
n_R(s,y)&=R^2n(t_0+R^2s,x_0+Ry),&
c_R(s,y)&=c(t_0+R^2s,x_0+Ry),\\
u_R(s,y)&=Ru(t_0+R^2s,x_0+Ry),&
\pi_R(s,y)&=R^2\pi(t_0+R^2s,x_0+Ry),\\
F_R(s,y)&=RF(t_0+R^2s,x_0+Ry).
\end{aligned}
\end{equation}
The rescaled functions satisfy \eqref{eq:system} with force
$F_R$, and $Q_R(z_0)$ corresponds to $Q_1(0,0)$.
Moreover,
$\|F_R\|_{L^\infty(Q_1)}
=R\|F\|_{L^\infty(Q_R(z_0))}$,
and $\eqref{eq:scale-quantities}_{2,3}$ are invariant under this scaling.
It therefore suffices to prove the local estimate for
$R=1$ and $z_0=(0,0)$.

\subsection{Joint local entropy and fixed-scale energy}

Let \(3<q\le10/3\), and define
\begin{equation}\label{eq:S-Pi}
S=\iint_{Q_1}(n^{\frac{q}{2}}+|\nabla\sqrt{c+1}|^q+|u|^q),\qquad
\Pi=\|\pi-(\pi)_{B_1}(t)\|_{L_t^2L_x^3(Q_1)}.
\end{equation}

\begin{Prop}\label{prop:point}
Let $3<q\le10/3$ and let $(n,c,u)$ be a classical solution of
\eqref{eq:system} in $Q_1$. {Suppose that $n$, $c$, $\nabla c$, and $u$ extend continuously to $(0,0)$.} Assume that
$n>0$, $0\le c\le A$, and $\|F\|_{L^\infty(Q_1)}\le G$.
There exist $\varepsilon_0=\varepsilon_0(A,G,q)>0$ and
$C=C(A,G,q)>0$ such that, if $S+\Pi^2\le\varepsilon_0$, then
\begin{equation}\label{eq:point-bound}
n(0,0)+|\nabla\sqrt{c+1}(0,0)|^2+|u(0,0)|^2
\le C\bigl(S^{\frac{3}{4q}}+\Pi\bigr).
\end{equation}
\end{Prop}

\begin{proof}
\smallskip
\noindent\textit{\bf Step 1: Local energy estimates.}
Recall $L_u=\partial_t-\Delta+u\cdot\nabla$. Since $\nabla c=2\sqrt{c+1}\nabla\sqrt{c+1}$ and $1\le\sqrt{c+1}\le\sqrt{1+A}$, equation $\eqref{eq:system}_2$ gives
\begin{equation}\label{eq:local-signal-equation}
\begin{aligned}
L_u\sqrt{c+1}
&=\frac{|\nabla\sqrt{c+1}|^2}{\sqrt{c+1}}-\frac{nc}{2\sqrt{c+1}},\\
L_u(\partial_i\sqrt{c+1})
&=\partial_i(L_u\sqrt{c+1})-\sum_{\ell=1}^3(\partial_i u_\ell)\partial_\ell\sqrt{c+1}.
\end{aligned}
\end{equation}
Using $\eqref{eq:local-signal-equation}$, we compute
\[
\begin{aligned}
L_u(2|\nabla\sqrt{c+1}|^2)
={}&4\nabla\sqrt{c+1}\cdot\nabla(L_u\sqrt{c+1})-4|\nabla^2\sqrt{c+1}|^2\\
&-4(\nabla\sqrt{c+1}\otimes\nabla\sqrt{c+1}):\nabla u\\
={}&-4|\nabla^2\sqrt{c+1}|^2
+\frac{8}{\sqrt{c+1}}\nabla^2\sqrt{c+1}:(\nabla\sqrt{c+1}\otimes\nabla\sqrt{c+1})\\
&-\frac{4|\nabla\sqrt{c+1}|^4}{c+1}
-2\frac{c}{\sqrt{c+1}}\nabla\sqrt{c+1}\cdot\nabla n\\
&-2n\left(1+\frac{1}{c+1}\right)|\nabla\sqrt{c+1}|^2
-4(\nabla\sqrt{c+1}\otimes\nabla\sqrt{c+1}):\nabla u.
\end{aligned}
\]
Since
\[
\nabla^2\sqrt{c+1}-\frac{\nabla\sqrt{c+1}\otimes\nabla\sqrt{c+1}}{\sqrt{c+1}}
=\frac{\sqrt{c+1}}{2}\nabla^2\log(c+1),
\]
the first three terms in the last expansion combine into
\[
-4\left|\nabla^2\sqrt{c+1}-\frac{\nabla\sqrt{c+1}\otimes\nabla\sqrt{c+1}}{\sqrt{c+1}}\right|^2
=-(c+1)|\nabla^2\log(c+1)|^2.
\]
To sum up, we have
\begin{equation}
    \begin{aligned}\label{eq:nabla sqrt 1+c}
    L_u(2|\nabla\sqrt{c+1}|^2)
&=-(c+1)|\nabla^2\log(c+1)|^2
-2\frac{c}{\sqrt{c+1}}\nabla\sqrt{c+1}\cdot\nabla n\\
&\quad-2n\left(1+\frac{1}{c+1}\right)|\nabla\sqrt{c+1}|^2-4(\nabla\sqrt{c+1}\otimes\nabla\sqrt{c+1}):\nabla u.
\end{aligned}
\end{equation}
By the chain rule, we derive
\begin{equation}\label{eq:nlogn}
    \begin{aligned}
L_u(n\log n)
&=(1+\log n)L_u n-\frac{|\nabla n|^2}{n}\\
&=-\diver\bigl(n(1+\log n)\nabla c\bigr)
+\nabla n\cdot\nabla c-4|\nabla \sqrt{n}|^2.
\end{aligned}
\end{equation}
Noting that $\nabla c=2\sqrt{c+1}\nabla\sqrt{c+1}$, we have
\begin{equation}\label{eq:nabla n cdot nabla c}
    \begin{aligned}
\nabla n\cdot\nabla c-2\frac{c}{\sqrt{c+1}}\nabla\sqrt{c+1}\cdot\nabla n&=\frac{2}{\sqrt{c+1}}\nabla\sqrt{c+1}\cdot\nabla n\\
&\le2|\nabla\sqrt n|^2+\frac{2n|\nabla\sqrt{c+1}|^2}{c+1}.
\end{aligned}
\end{equation}
Combining \eqref{eq:nabla sqrt 1+c} to \eqref{eq:nabla n cdot nabla c}, we obtain
\begin{equation}\label{eq:local-entropy}
\begin{aligned}
&L_u(n\log n+2|\nabla\sqrt{c+1}|^2)
+2|\nabla\sqrt n|^2+(c+1)|\nabla^2\log(c+1)|^2+2n|\nabla\sqrt{c+1}|^2\\
&\le-\diver\bigl(n(1+\log n)\nabla c\bigr)-4(\nabla\sqrt{c+1}\otimes\nabla\sqrt{c+1}):\nabla u.
\end{aligned}
\end{equation}

It remains to control the last term involving $\nabla u$. By Young's inequality and $c+1\le1+A$, we have
\begin{equation}\label{eq:strain-local}
4|\nabla u||\nabla\sqrt{c+1}|^2
\le\frac{|\nabla\sqrt{c+1}|^4}{c+1}+4(1+A)|\nabla u|^2.
\end{equation}
Thus, it suffices to estimate the quartic term by the Hessian dissipation. For each nonnegative $\psi\in C_c^\infty(B_1)$, we first consider
\[
2\int_{B_1}\frac{|\nabla\sqrt{c+1}|^4}{c+1}\psi\,\dd x.
\]
To relate this integral to $\nabla^2\log(c+1)$, we differentiate the vector field $|\nabla\sqrt{c+1}|^2\nabla\sqrt{c+1}/\sqrt{c+1}$. Note that
\[
\begin{aligned}
&\diver\left(\frac{|\nabla\sqrt{c+1}|^2\nabla\sqrt{c+1}}{\sqrt{c+1}}\right)\\
={}&\frac{2\nabla^2\sqrt{c+1}:(\nabla\sqrt{c+1}\otimes\nabla\sqrt{c+1})+|\nabla\sqrt{c+1}|^2\Delta\sqrt{c+1}}{\sqrt{c+1}}-\frac{|\nabla\sqrt{c+1}|^4}{c+1}.
\end{aligned}
\]
Then, using
\[
\nabla^2\sqrt{c+1}=\frac{\sqrt{c+1}}{2}\nabla^2\log(c+1)+\frac{\nabla\sqrt{c+1}\otimes\nabla\sqrt{c+1}}{\sqrt{c+1}},
\]
we obtain
\begin{equation}\label{eq:local-quartic-divergence}
\begin{aligned}
&\diver\left(\frac{|\nabla\sqrt{c+1}|^2\nabla\sqrt{c+1}}{\sqrt{c+1}}\right)\\
={}&\nabla^2\log(c+1):(\nabla\sqrt{c+1}\otimes\nabla\sqrt{c+1})+\frac{1}{2}\Delta\log(c+1)|\nabla\sqrt{c+1}|^2+2\frac{|\nabla\sqrt{c+1}|^4}{c+1}.
\end{aligned}
\end{equation}
Multiplying \eqref{eq:local-quartic-divergence} by $\psi$ and integrating by parts, we get
\[
\begin{aligned}
2\int_{B_1}\frac{|\nabla\sqrt{c+1}|^4}{c+1}\psi\,\dd x
={}&-\frac{1}{2}\int_{B_1}\nabla^2\log(c+1):\\
&\quad\bigl(2\nabla\sqrt{c+1}\otimes\nabla\sqrt{c+1}+|\nabla\sqrt{c+1}|^2I_3\bigr)\psi\,\dd x-J_\psi,
\end{aligned}
\]
where
\[
J_\psi:=\int_{B_1}\frac{|\nabla\sqrt{c+1}|^2\nabla\sqrt{c+1}}{\sqrt{c+1}}\cdot\nabla\psi\,\dd x.
\]
Since
\[
|2\nabla\sqrt{c+1}\otimes\nabla\sqrt{c+1}+|\nabla\sqrt{c+1}|^2I_3|^2=11|\nabla\sqrt{c+1}|^4,
\]
Cauchy--Schwarz and Young's inequalities yield
\[
\begin{aligned}
2\int_{B_1}\frac{|\nabla\sqrt{c+1}|^4}{c+1}\psi\,\dd x
&\le\frac{\sqrt{11}}{2}\int_{B_1}|\nabla^2\log(c+1)||\nabla\sqrt{c+1}|^2\psi\,\dd x+|J_\psi|\\
&\le\int_{B_1}\frac{|\nabla\sqrt{c+1}|^4}{c+1}\psi\,\dd x+\frac{11}{16}\int_{B_1}(c+1)|\nabla^2\log(c+1)|^2\psi\,\dd x+|J_\psi|,
\end{aligned}
\]
which implies
\begin{equation}\label{eq:quartic-local}
\int_{B_1}\frac{|\nabla\sqrt{c+1}|^4}{c+1}\psi\,\dd x
\le\frac{11}{16}\int_{B_1}(c+1)|\nabla^2\log(c+1)|^2\psi\,\dd x+|J_\psi|.
\end{equation}

Let \(g=\sqrt{c+1}\). Using
\(\nabla^2 g=\frac{g}{2}\nabla^2\log(c+1)
             +g^{-1}\nabla g\otimes\nabla g\)
and the preceding estimate, we obtain
\[
\begin{aligned}
\int_{B_1}\left(\frac16|\nabla^2 g|^2
       +\frac{|\nabla g|^4}{g^2}\right)\psi\,dx
&\le \frac1{12}\int_{B_1}g^2|\nabla^2\log(c+1)|^2\psi\,dx
     +\frac43\int_{B_1}\frac{|\nabla g|^4}{g^2}\psi\,dx \\
&\le \int_{B_1}g^2|\nabla^2\log(c+1)|^2\psi\,dx
     +\frac43|J_\psi|,
\end{aligned}
\]
where \(\frac1{12}+\frac43\cdot\frac{11}{16}=1\).
Since \(g\ge1\), the cutoff term satisfies
\(
 |J_\psi|\le\int_{B_1}|\nabla g|^3|\nabla\psi|\,dx.
\)
We now test the local entropy inequality $\eqref{eq:local-entropy}$ with a nonnegative $\psi\in C_c^\infty((-1,0]\times B_1)$. Multiplying by $\psi$, integrating over $B_1$, and using integration by parts, we obtain
\[
\begin{aligned}
&\frac{\dd}{\dd t}\int_{B_1}(n\log n+2|\nabla\sqrt{c+1}|^2)\psi\,\dd x\\
&+\int_{B_1}\left(2|\nabla\sqrt n|^2+(c+1)|\nabla^2\log(c+1)|^2+2n|\nabla\sqrt{c+1}|^2\right)\psi\,\dd x\\
&\le\int_{B_1}(n\log n+2|\nabla\sqrt{c+1}|^2)(\psi_t+\Delta\psi+u\cdot\nabla\psi)\,\dd x\\
&\quad+\int_{B_1}n(1+\log n)\nabla c\cdot\nabla\psi\,\dd x -4\int_{B_1}(\nabla\sqrt{c+1}\otimes\nabla\sqrt{c+1}):\nabla u\,\psi\,\dd x.
\end{aligned}
\]
The fluid equation $\eqref{eq:system}_3$ also gives
\begin{equation}\label{eq:local-Stokes-energy}
(\partial_t-\Delta)|u|^2+2|\nabla u|^2=-2\diver(\pi u)+2nF\cdot u.
\end{equation}
Testing $\eqref{eq:local-Stokes-energy}$ with $4(1+A)\psi$, integrating over $B_1$, and adding the resulting inequality to the preceding estimate, we use $\eqref{eq:strain-local}$--$\eqref{eq:quartic-local}$ to obtain
\begin{equation}\label{eq:joint-local}
\begin{aligned}
&\int_{B_1}\bigl(n\log n+2|\nabla\sqrt{c+1}|^2+4(1+A)|u|^2\bigr)\psi(\tau)\,\dd x\\
&\quad+c_A\int_{-1}^{\tau}\int_{B_1}
\bigl(|\nabla\sqrt n|^2+|\nabla^2\sqrt{c+1}|^2+|\nabla u|^2\bigr)\psi\,\dd x\,\dd t\\
&\le C_A\int_{-1}^{\tau}\int_{B_1}
(n|\log n|+|\nabla\sqrt{c+1}|^2+|u|^2)|\psi_t+\Delta\psi|\,\dd x\,\dd t\\
&\quad+C_A\int_{-1}^{\tau}\int_{B_1}
\bigl[(n|\log n|+n)|\nabla\sqrt{c+1}|+|\nabla\sqrt{c+1}|^3\\
&\qquad\qquad\qquad\qquad+(n|\log n|+|\nabla\sqrt{c+1}|^2)|u|\bigr]|\nabla\psi|\,\dd x\,\dd t\\
&\quad+8(1+A)\int_{-1}^{\tau}\int_{B_1}\pi u\cdot\nabla\psi\,\dd x\,\dd t+8(1+A)\int_{-1}^{\tau}\int_{B_1}nF\cdot u\,\psi\,\dd x\,\dd t.
\end{aligned}
\end{equation}
Finally, testing $\eqref{eq:system}_1$ with $\psi$ and using $\diver u=0$, we obtain
\begin{equation}\label{eq:local-mass}
\begin{aligned}
\int_{B_1}n\psi(\tau)\,\dd x
={}&\int_{-1}^{\tau}\int_{B_1}n(\psi_t+\Delta\psi)\,\dd x\,\dd t+\int_{-1}^{\tau}\int_{B_1}n(\nabla c+u)\cdot\nabla\psi\,\dd x\,\dd t.
\end{aligned}
\end{equation}

\smallskip
\noindent\textit{\bf Step 2: Energy estimate on $Q_{1/2}$.}
Taking $\varepsilon_0\le1$, we have $S\le1$.
Let $q'=q/(q-1)$.
Since $q>3$, we have $q'<q/2$ and
\begin{equation}\label{eq:local-entropy-growth}
|z\log z|^{q'}\le C_q(z+z^{\frac{q}{2}}),
\qquad
|z\log z|\le C_q(z^{\frac{3}{4}}+z^{\frac{q}{2}}),
\qquad z>0.
\end{equation}
By \eqref{eq:S-Pi}, H\"older's inequality, and
\eqref{eq:local-entropy-growth}, we have
\begin{equation}\label{eq:H-fixed}
\begin{aligned}
\|n\log n\|_{L^{q'}(Q_1)}
&\le C_q(S^{\frac{2}{q}}+S)^{\frac{1}{q'}}
\le C_qS^{\frac{2(q-1)}{q^2}},\\
\|n\log n\|_{L^1(Q_1)}
&\le C_q(S^{\frac{3}{2q}}+S)
\le C_qS^{\frac{3}{2q}}.
\end{aligned}
\end{equation}
Consequently,
\[
\begin{aligned}
\iint_{Q_1}n|\log n|(|\nabla\sqrt{c+1}|+|u|)
&\le C_qS^{\frac{3q-2}{q^2}},\\
\iint_{Q_1}
(n|\nabla\sqrt{c+1}|+n|u|+|\nabla\sqrt{c+1}|^3+|\nabla\sqrt{c+1}|^2|u|)
&\le C_qS^{\frac{3}{q}},\\
\iint_{Q_1}(|\nabla\sqrt{c+1}|^2+|u|^2)
&\le C_qS^{\frac{2}{q}},\\
\iint_{Q_1}n|F||u|
&\le C_qG S^{\frac{3}{q}}.
\end{aligned}
\]
All these terms are bounded by
$C_{G,q}S^{3/(2q)}$.
Choose $\psi\in C_c^\infty((-1,0]\times B_1)$ such that
\[
0\le\psi\le1,\qquad
\psi=1\ \text{on }Q_{\frac{1}{2}},\qquad
|\partial_t\psi|+|\nabla\psi|+|\nabla^2\psi|\le C.
\]
By $\diver u=0$, we have
\[
\int_{B_1}(\pi)_{B_1}(t)u\cdot\nabla\psi\,\dd x=0.
\]
Hence
\begin{equation}\label{eq:pressure-fixed}
\begin{aligned}
\left|\int_{-1}^{\tau}\int_{B_1}
\pi u\cdot\nabla\psi\,\dd x\,\dd t\right|
\le C\Pi\|u\|_{L^2(-1,0;L^{\frac{3}{2}}(B_1))}\le C\Pi S^{\frac{1}{q}}
\le C(\Pi^2+S^{\frac{2}{q}}).
\end{aligned}
\end{equation}

We next estimate the negative part $(n\log n)_-=\max\{-n\log n,0\}$.
By \eqref{eq:local-mass}, we get
\[
\sup_{-1<\tau<0}\int_{B_1}n\psi(\tau)\,\dd x
\le C_A(S^{\frac{2}{q}}+S^{\frac{3}{q}})
\le C_A S^{\frac{2}{q}}.
\]
Since $(n\log n)_-\le Cn^{3/4}$, every nonnegative integrable
weight $\omega$ satisfies
\begin{equation}\label{eq:negative-local}
\int (n\log n)_-\omega
\le C\left(\int n\omega\right)^{\frac{3}{4}}
       \left(\int\omega\right)^{\frac{1}{4}}.
\end{equation}
Applying \eqref{eq:negative-local} with $\omega=\psi(\tau)$,
we obtain
\[
\sup_{-1<\tau<0}\int_{B_1}(n\log n)_-\psi(\tau)\,\dd x
\le C_A S^{\frac{3}{2q}}.
\]
Since $n\log n+2(n\log n)_-=n|\log n|$, adding
$2\int_{B_1}(n\log n)_-\psi(\tau)\,\dd x$ to both sides of
\eqref{eq:joint-local} and using
\eqref{eq:H-fixed}--\eqref{eq:negative-local}, we obtain
\[
\begin{aligned}
&\int_{B_1}\bigl(n|\log n|+2|\nabla\sqrt{c+1}|^2+4(1+A)|u|^2\bigr)
\psi(\tau)\,\dd x\\
&\quad+c_A\int_{-1}^{\tau}\int_{B_1}
\bigl(|\nabla\sqrt n|^2+|\nabla^2\sqrt{c+1}|^2+|\nabla u|^2\bigr)
\psi\,\dd x\,\dd t
\le C_{A,G,q}(S^{\frac{3}{2q}}+\Pi^2).
\end{aligned}
\]
Combining this estimate with \eqref{eq:local-mass}, we conclude that
\begin{equation}\label{eq:fixed-energy}
\begin{aligned}
&\sup_{-\frac{1}{4}<t<0}\int_{B_{\frac{1}{2}}}
(n+n|\log n|+|\nabla\sqrt{c+1}|^2+|u|^2)\,\dd x\\
&\quad+\iint_{Q_{\frac{1}{2}}}
\bigl(|\nabla\sqrt n|^2+|\nabla^2\sqrt{c+1}|^2+|\nabla u|^2\bigr)
\,\dd x\,\dd t
\le d_0,
\end{aligned}
\end{equation}
where $d_0=C_{A,G,q}(S^{3/(2q)}+\Pi^2)$.

\smallskip
\noindent\textit{\bf Step 3: Interpolation and the induction.}
Define
\begin{equation}\label{eq:Er}
\begin{aligned}
E(r)&:=r^{-3}\bigg[
\sup_{-r^2<t<0}\int_{B_r}
(n+n|\log n|+|\nabla\sqrt{c+1}|^2+|u|^2)\,\dd x\\
&\qquad+\iint_{Q_r}
\bigl(|\nabla\sqrt n|^2+|\nabla^2\sqrt{c+1}|^2+|\nabla u|^2\bigr)
\,\dd x\,\dd t\bigg].
\end{aligned}
\end{equation}
Set
\begin{equation}\label{eq:dyadic-parameters}
r_k=2^{-k-3}\quad(k\ge0),\qquad
r_{-1}=\frac{1}{4},\qquad
m_0=C_0d_0^{\frac{1}{2}},
\end{equation}
where $C_0\ge1$ is fixed.
We prove $E(r_k)\le m_0$ for all $k\ge-1$.
By \eqref{eq:fixed-energy}, we have
\[
E(r_{-1})+E(r_0)\le C d_0\le m_0
\]
when $d_0$ is sufficiently small.
Fix $j\ge1$ and assume that
\begin{equation}\label{eq:dyadic-hypothesis}
E(r_k)\le m_0,\qquad -1\le k\le j-1.
\end{equation}
For each fixed $t\in(-r^2,0)$, applying the Sobolev inequality on $B_1$ to $g(y)=f(ry)$, we have
\begin{align*}
    \|g\|_{L^6(B_1)}^2&\le C\left(\|\nabla_y g\|_{L^2(B_1)}^2+\|g\|_{L^2(B_1)}^2\right),\\
\|f\|_{L^6(B_r)}^2&\le C\left(\|\nabla f\|_{L^2(B_r)}^2+r^{-2}\|f\|_{L^2(B_r)}^2\right).
\end{align*}
Thus, it follows that
\begin{equation}\label{ineq:f-10/3}
    \begin{aligned}
    \iint_{Q_r}|f|^{\frac{10}{3}}\,\dd x\,\dd t
\le C\left(\sup_{-r^2<t<0}\|f(t)\|_{L^2(B_r)}^2\right)^{\frac{2}{3}}\left(\iint_{Q_r}|\nabla f|^2\,\dd x\,\dd t+r^{-2}\iint_{Q_r}|f|^2\,\dd x\,\dd t\right).
\end{aligned}
\end{equation}
For $r_{j-1}\leq r \leq r_{-1}$, applying this estimate to $\sqrt n$, $\nabla\sqrt{c+1}$, and $u$ in \eqref{eq:Er}, by \eqref{eq:dyadic-hypothesis} we obtain
\begin{equation}\label{eq:dyadic-Sob}
\begin{cases}
\|n\|_{L^{\frac{5}{3}}(Q_r)}
\le Cm_0r^3,\\
\|\nabla\sqrt{c+1}\|_{L^{\frac{10}{3}}(Q_r)}+\|u\|_{L^{\frac{10}{3}}(Q_r)}
\le Cm_0^{\frac{1}{2}}r^{\frac{3}{2}}.
\end{cases}
\end{equation}
Moreover, by interpolation and the Sobolev inequality, we have
\[
\begin{aligned}
\|n(t)\|_{L^{\frac{3}{2}}(B_r)}^2
=\|\sqrt n(t)\|_{L^3(B_r)}^4\le C\|\sqrt n(t)\|_{L^2(B_r)}^2
\left(\|\nabla\sqrt n(t)\|_{L^2(B_r)}^2
+r^{-2}\|\sqrt n(t)\|_{L^2(B_r)}^2\right).
\end{aligned}
\]
Integrating in time and using \eqref{eq:Er} and \eqref{eq:dyadic-hypothesis}, we have
\[
\begin{aligned}
\|n\|_{L^2(-r^2,0;L^{\frac{3}{2}}(B_r))}^2
&\le C(m_0r^3)(m_0r^3+r^{-2}r^2m_0r^3)
\le Cm_0^2r^6,\\
\|u\|_{L^2(-r^2,0;L^{\frac{3}{2}}(B_r))}^2
&\le Cr\iint_{Q_r}|u|^2
\le Cr^3\sup_{-r^2<t<0}\int_{B_r}|u|^2
\le Cm_0r^6.
\end{aligned}
\]
Thus
\begin{equation}\label{eq:dyadic-mixed}
\begin{aligned}
\|n\|_{L^2(-r^2,0;L^{\frac{3}{2}}(B_r))}
\le Cm_0r^3,\qquad
\|u\|_{L^2(-r^2,0;L^{\frac{3}{2}}(B_r))}
\le Cm_0^{\frac{1}{2}}r^3.
\end{aligned}
\end{equation}
Let $m_0<1$,  since $|z\log z|^{3/2}\le C(z^{4/3}+z^{5/3})$, by
\eqref{eq:dyadic-Sob}, we have
\begin{equation}\label{eq:dyadic-entropy-bound}
\begin{aligned}
    \|n\log n\|_{L^{\frac{3}{2}}(Q_r)}\le C\left( \iint_{Q_{r}}n^{\frac{4}{3}}+\iint_{Q_{r}}n^{\frac{5}{3}}\right)^{\frac{2}{3}}
\leq C\left(|Q_r|^{\frac{1}{5}}\|n\|_{L^{\frac{5}{3}}(Q_r)}^{\frac{4}{3}}+m_0^{\frac{5}{3}}r^5\right)^{\frac{2}{3}}
\le Cm_0^{\frac{8}{9}}r^{\frac{10}{3}}.
\end{aligned}
\end{equation}
By \eqref{eq:dyadic-Sob}, \eqref{eq:dyadic-entropy-bound}, and $|Q_r|=Cr^5$, H\"older's inequality gives
\begin{equation}\label{eq:dyadic-flux}
\begin{aligned}
& r^{-4}\iint_{Q_r}\Big[n|\log n|(|\nabla\sqrt{c+1}|+|u|)+(n|\nabla\sqrt{c+1}|+n|u|+|\nabla\sqrt{c+1}|^3+|\nabla\sqrt{c+1}|^2|u|)\Big]\,\dd x\,\dd t\\
&\quad+r^{-3}\iint_{Q_r}n|F||u|\,\dd x\,\dd t\\
&\le Cr^{-4}\Bigg[|Q_r|^{\frac{1}{30}}\|n\log n\|_{L^{\frac{3}{2}}(Q_r)}
\bigl(\|\nabla\sqrt{c+1}\|_{L^{\frac{10}{3}}(Q_r)}+\|u\|_{L^{\frac{10}{3}}(Q_r)}\bigr)\\
&\quad+|Q_r|^{\frac{1}{10}}\bigl(\|n\|_{L^{\frac{5}{3}}(Q_r)}+\|\nabla\sqrt{c+1}\|_{L^{\frac{10}{3}}(Q_r)}^2\bigr)\bigl(\|\nabla\sqrt{c+1}\|_{L^{\frac{10}{3}}(Q_r)}+\|u\|_{L^{\frac{10}{3}}(Q_r)}\bigr)\Bigg]\\
&\quad+Cr^{-3}G|Q_r|^{\frac{1}{10}}\|n\|_{L^{\frac{5}{3}}(Q_r)}\|u\|_{L^{\frac{10}{3}}(Q_r)}\\
&\le Cr\bigl(m_0^{\frac{25}{18}}+m_0^{\frac{3}{2}}\bigr)+CG r^2m_0^{\frac{3}{2}}.
\end{aligned}
\end{equation}

\smallskip
\noindent\textit{\bf Step 4: The pressure estimates.}
Choose $\chi\in C_c^\infty(B_{1/2})$ with
$\chi=1$ on $B_{3/8}$, and write
\begin{equation}\label{eq:pressure-decomposition}
\pi^{\rm loc}=\Delta^{-1}\nabla\cdot(\chi nF),
\qquad
\pi^{\rm har}=\pi-\pi^{\rm loc}.
\end{equation}
By \eqref{eq:pressure-potential} and 
\eqref{eq:fixed-energy}, we have
\[
\|\pi^{\rm loc}\|_{L^2(-\frac{1}{4},0;L^3(\R^3))}
\le CG\|n\|_{L^2(-\frac{1}{4},0;L^{\frac{3}{2}}(B_{\frac{1}{2}}))}
\le CGd_0.
\]
Replacing $\pi$ by $\pi-(\pi)_{B_1}(t)$, we may assume that
$(\pi)_{B_1}(t)=0$.
Since $\Delta\pi=\nabla\cdot(nF)$,
$\pi^{\rm har}$ is harmonic in $B_{3/8}$ by \eqref{eq:pressure-decomposition}.
By the interior estimate for harmonic functions, we have
\begin{equation}\label{eq:harmonic-pressure}
\begin{aligned}
\|\nabla\pi^{\rm har}\|_{L^2(-\frac14,0;L^\infty(B_{\frac14}))}
&\le C\|\pi^{\rm har}\|_{L^2(-\frac14,0;L^3(B_{\frac38}))}\\
&\le C\|\pi\|_{L^2(-\frac14,0;L^3(B_{\frac38}))}
+C\|\pi^{\rm loc}\|_{L^2(-\frac14,0;L^3(\mathbb R^3))}\\
&\le C\Pi
+CG\|n\|_{L^2(-\frac14,0;L^{3/2}(B_{\frac12}))}\\
&\le C(\Pi+Gd_0).
\end{aligned}
\end{equation}
Write
$\pi^{\rm loc}=\pi^{\rm near}+\pi^{\rm far}$, where
\[
\pi^{\rm near}
=\Delta^{-1}\nabla\cdot
\bigl(\chi nF\,\mathbf{1}_{B_{2r_k}}\bigr),
\qquad
\pi^{\rm far}
=\Delta^{-1}\nabla\cdot
\bigl(\chi nF\,\mathbf{1}_{\R^3\setminus B_{2r_k}}\bigr).
\]
By \eqref{eq:pressure-potential} and
\eqref{eq:dyadic-mixed}, we have
\begin{equation}\label{eq:near-pressure}
\begin{aligned}
\|\pi^{\rm near}\|_{L^2(-r_k^2,0;L^3(B_{r_k}))}
\le CG\|n\|_{L^2(-r_k^2,0;L^{\frac{3}{2}}(B_{2r_k}))}
\le CGm_0r_k^3.
\end{aligned}
\end{equation}
For $x\in B_{r_k}$ and $-r_k^2<t<0$, using the Newtonian
kernel $\Phi(z)=-1/(4\pi|z|)$, we have
\[
\nabla\pi^{\rm far}(t,x)
=\int_{B_{\frac12}\setminus B_{2r_k}}
\nabla^2\Phi(x-y)\,\chi(y)n(t,y)F(t,y)\,\dd y,
\]
which implies
\[
\begin{aligned}
|\nabla\pi^{\rm far}(t,x)|
&\le C\int_{B_{\frac12}\setminus B_{2r_k}}
\frac{n(t,y)|F(t,y)|}{|x-y|^3}\,\dd y\\
&\le CG\int_{B_{\frac12}\setminus B_{2r_k}}
\frac{n(t,y)}{|x-y|^3}\,\dd y\\
&\le CG\int_{B_{\frac12}\setminus B_{2r_k}}
\frac{n(t,y)}{|y|^3}\,\dd y,
\end{aligned}
\]
where we used $|\nabla^2\Phi(z)|\le C|z|^{-3}$ and
$|x-y|\ge |y|-|x|\ge |y|/2$ on the integration region.
If $k=0$, the integration region is the fixed annulus
$B_{1/2}\setminus B_{1/4}$, and \eqref{eq:fixed-energy} gives
\[
|\nabla\pi^{\rm far}(t,x)|\le CGd_0\le CGm_0.
\]
For $1\le k\le j-1$, let
$\rho_\ell=2^{\ell+1}r_k$ for $\ell=0,\ldots,k-1$.
Then $\rho_{k-1}=r_0$, and decomposing the integration
region into dyadic annuli, we obtain
\[
\begin{aligned}
|\nabla\pi^{\rm far}(t,x)|
\le{}&CG\sum_{\ell=0}^{k-1}\rho_\ell^{-3}
\int_{B_{2\rho_\ell}}n(t,y)\,\dd y+CG\int_{B_{\frac{1}{2}}\setminus B_{2r_0}}
\frac{n(t,y)}{|y|^3}\,\dd y.
\end{aligned}
\]
For each $\ell=0,\ldots,k-1$, we have
\[
2\rho_\ell=2^{\ell+2}r_k=r_{k-\ell-2},
\qquad -1\le k-\ell-2\le j-1.
\]
Thus, by the induction hypothesis \eqref{eq:dyadic-hypothesis},
\[
\int_{B_{2\rho_\ell}}n(t,y)\,\dd y
\le m_0(2\rho_\ell)^3.
\]
Consequently,
\[
\begin{aligned}
\sum_{\ell=0}^{k-1}\rho_\ell^{-3}
\int_{B_{2\rho_\ell}}n(t,y)\,\dd y
\le\sum_{\ell=0}^{k-1}\rho_\ell^{-3}m_0(2\rho_\ell)^3=8m_0\sum_{\ell=0}^{k-1}1=8km_0\le Cm_0\left(1+\log\frac{r_0}{r_k}\right).
\end{aligned}
\]
Moreover, since $|y|\ge2r_0$ on
$B_{1/2}\setminus B_{2r_0}$, \eqref{eq:fixed-energy}
gives
\[
\int_{B_{\frac{1}{2}}\setminus B_{2r_0}}
\frac{n(t,y)}{|y|^3}\,\dd y
\le Cr_0^{-3}\int_{B_{\frac{1}{2}}}n(t,y)\,\dd y
\le Cd_0.
\]
Since $d_0\le m_0$, both cases yield
\begin{equation}\label{eq:far-pressure-gradient}
|\nabla\pi^{\rm far}(t,x)|
\le CGm_0\left(1+\log\frac{r_0}{r_k}\right).
\end{equation}
For $f\in L^3(B_{r_k})$, we have
\[
\|f-(f)_{B_{r_k}}\|_{L^3(B_{r_k})}
\le 2\|f\|_{L^3(B_{r_k})}.
\]
Moreover, if $f$ is Lipschitz continuous on $B_{r_k}$, then
\[
\|f-(f)_{B_{r_k}}\|_{L^3(B_{r_k})}
\le Cr_k^2\|\nabla f\|_{L^\infty(B_{r_k})}.
\]
Recall that
$\pi=\pi^{\rm near}+\pi^{\rm far}+\pi^{\rm har}$.
By the triangle inequality, we obtain
\[
\begin{aligned}
&\|\pi-(\pi)_{B_{r_k}}(t)\|_{L^2(-r_k^2,0;L^3(B_{r_k}))}\le\sum_{\alpha\in\{\mathrm{near},\mathrm{far},\mathrm{har}\}}
\|\pi^\alpha-(\pi^\alpha)_{B_{r_k}}(t)\|_{L^2(-r_k^2,0;L^3(B_{r_k}))}.
\end{aligned}
\]
We estimate the three terms separately. For the near pressure, using the first inequality above and \eqref{eq:near-pressure}, we have
\[
\begin{aligned}
&\|\pi^{\rm near}-(\pi^{\rm near})_{B_{r_k}}(t)\|_{L^2(-r_k^2,0;L^3(B_{r_k}))}\le2\|\pi^{\rm near}\|_{L^2(-r_k^2,0;L^3(B_{r_k}))}
\le CGm_0r_k^3.
\end{aligned}
\]
For the far pressure, using the second inequality and
\eqref{eq:far-pressure-gradient}, we obtain
\[
\begin{aligned}
&\|\pi^{\rm far}-(\pi^{\rm far})_{B_{r_k}}(t)\|_{L^2(-r_k^2,0;L^3(B_{r_k}))}\le Cr_k^2\|\nabla\pi^{\rm far}\|_{L^2(-r_k^2,0;L^\infty(B_{r_k}))}\le CGm_0r_k^3\left(1+\log\frac{r_0}{r_k}\right).
\end{aligned}
\]
Similarly, by \eqref{eq:harmonic-pressure},
\[
\begin{aligned}
&\|\pi^{\rm har}-(\pi^{\rm har})_{B_{r_k}}(t)\|_{L^2(-r_k^2,0;L^3(B_{r_k}))}\le Cr_k^2\|\nabla\pi^{\rm har}\|_{L^2(-r_k^2,0;L^\infty(B_{\frac{1}{4}}))}\le Cr_k^2(\Pi+Gd_0).
\end{aligned}
\]
Combining the above estimates and absorbing the near-pressure contribution into the far-pressure bound, we obtain
\begin{equation}\label{eq:pressure-decay}
\begin{aligned}
\|\pi-(\pi)_{B_{r_k}}(t)\|_{L^2(-r_k^2,0;L^3(B_{r_k}))}
\le{}&CGm_0r_k^3\left(1+\log\frac{r_0}{r_k}\right)+Cr_k^2(\Pi+Gd_0).
\end{aligned}
\end{equation}
Next, by H\"older's inequality in space and time, it follows that
\[
\begin{aligned}
&\iint_{Q_{r_k}}|\pi-(\pi)_{B_{r_k}}(t)||u|\,\dd x\,\dd t\le\|\pi-(\pi)_{B_{r_k}}(t)\|_{L^2(-r_k^2,0;L^3(B_{r_k}))}
\|u\|_{L^2(-r_k^2,0;L^{\frac{3}{2}}(B_{r_k}))}.
\end{aligned}
\]
By \eqref{eq:dyadic-mixed} and \eqref{eq:pressure-decay}, we obtain
\begin{equation}\label{eq:pressure-flux}
\begin{aligned}
r_k^{-4}\iint_{Q_{r_k}}
|\pi-(\pi)_{B_{r_k}}(t)||u|\,\dd x\,\dd t\le CGm_0^{\frac{3}{2}}r_k^2\left(1+\log\frac{r_0}{r_k}\right)
+Cm_0^{\frac{1}{2}}r_k(\Pi+Gd_0).
\end{aligned}
\end{equation}

\smallskip
\noindent\textit{\bf Step 5: The induction estimate.}
Define
\begin{equation}\label{eq:backward-kernel}
\Psi_j(t,x)=[4\pi(r_j^2-t)]^{-\frac{3}{2}}
\exp\left(-\frac{|x|^2}{4(r_j^2-t)}\right),
\qquad \psi_j=\Psi_j\xi,
\end{equation}
where $\xi\in C_c^\infty((-1/16,0]\times B_{1/4})$ satisfies
\[
0\le\xi\le1,\qquad \xi=1\ \text{on }Q_{\frac{1}{8}},
\qquad |\partial_t\xi|+|\nabla\xi|+|\nabla^2\xi|\le C.
\]
Following the backward heat-kernel construction in
\cite{Caffarelli-1982}, we have
\begin{equation}\label{eq:kernel-test}
\psi_j\ge cr_j^{-3}\quad\text{on }Q_{r_j},
\qquad \int_{\R^3}\psi_j(t,x)\,\dd x\le1,
\qquad \partial_t\Psi_j+\Delta\Psi_j=0,
\end{equation}
and
\begin{equation}\label{eq:kernel-annular-bounds}
\begin{cases}
|(\partial_t+\Delta)\psi_j|\le C
&\text{on }Q_{\frac{1}{4}},\\
|\psi_j|\le Cr_k^{-3},
|\nabla\psi_j|\le Cr_k^{-4}
&\text{on }Q_{r_k}\setminus Q_{r_{k+1}},
\quad 0\le k\le j-2,\\
|\psi_j|\le Cr_{j-1}^{-3},
|\nabla\psi_j|\le Cr_{j-1}^{-4}
&\text{on }Q_{r_{j-1}},
\end{cases}
\end{equation}
where the constants are independent of $j$ and $k$.
To estimate the pressure term, choose $\chi_k\in C_c^\infty((-r_k^2,0]\times B_{r_k})$
such that
\begin{align}\label{nabla-psik}
    0\le\chi_k\le1,\qquad
\chi_k=1\ \text{on }Q_{r_{k+1}},\qquad
|\nabla\chi_k|\le Cr_k^{-1}.
\end{align}
We decompose $\psi_j$ as
\begin{equation}\label{eq:pressure-partition}
\psi_j
=(1-\chi_0)\psi_j
+\sum_{k=0}^{j-2}(\chi_k-\chi_{k+1})\psi_j
+\chi_{j-1}\psi_j.
\end{equation}
Each term in \eqref{eq:pressure-partition} is compactly
supported in $B_{1/4}$ in space. By $\diver u=0$ and integration by parts, at each fixed time, we have
\[
\begin{aligned}
\int_{B_{\frac{1}{4}}}\pi u\cdot\nabla\psi_j\,\dd x
={}&\int_{B_{\frac{1}{4}}}
\bigl(\pi-(\pi)_{B_1}(t)\bigr)
u\cdot\nabla\bigl((1-\chi_0)\psi_j\bigr)\,\dd x\\
&+\sum_{k=0}^{j-2}\int_{B_{\frac{1}{4}}}
\bigl(\pi-(\pi)_{B_{r_k}}(t)\bigr)
u\cdot\nabla\bigl((\chi_k-\chi_{k+1})\psi_j\bigr)\,\dd x\\
&+\int_{B_{\frac{1}{4}}}
\bigl(\pi-(\pi)_{B_{r_{j-1}}}(t)\bigr)
u\cdot\nabla(\chi_{j-1}\psi_j)\,\dd x.
\end{aligned}
\]
We estimate these three contributions separately.
Since $(1-\chi_0)\psi_j$ vanishes on $Q_{r_1}$,
the Gaussian bounds give
$|\nabla((1-\chi_0)\psi_j)|\le C$, uniformly in $j$.
By H\"older's inequality, \eqref{eq:S-Pi} and \eqref{eq:fixed-energy}, we have
\[
\begin{aligned}
\left|\int_{-\frac{1}{16}}^{\tau}\int_{B_{\frac{1}{4}}}
\bigl(\pi-(\pi)_{B_1}(t)\bigr)
u\cdot\nabla\bigl((1-\chi_0)\psi_j\bigr)
\,\dd x\,\dd t\right|\le C\Pi
\|u\|_{L^2(-\frac{1}{16},0;L^{\frac{3}{2}}(B_{\frac{1}{4}}))}
\le C\Pi d_0^{\frac{1}{2}}.
\end{aligned}
\]
For $0\le k\le j-2$, \eqref{eq:kernel-annular-bounds} and \eqref{nabla-psik} implies
\[
\begin{aligned}
\left|\nabla\bigl((\chi_k-\chi_{k+1})\psi_j\bigr)\right|
&=\left|(\chi_k-\chi_{k+1})\nabla\psi_j
+\psi_j(\nabla\chi_k-\nabla\chi_{k+1})\right|\le Cr_k^{-4}
\quad\text{on }Q_{r_k}\setminus Q_{r_{k+2}}.
\end{aligned}
\]
This gradient vanishes outside the indicated region.
For the last term, we similarly have
\[
|\nabla(\chi_{j-1}\psi_j)|\le Cr_{j-1}^{-4},
\qquad
\operatorname{supp}\nabla(\chi_{j-1}\psi_j)
\subset (-r_{j-1}^2,0]\times B_{r_{j-1}}.
\]
Consequently,
\[
\begin{aligned}
\left|\int_{-\frac{1}{16}}^{\tau}\int_{B_{\frac{1}{4}}}
\pi u\cdot\nabla\psi_j\,\dd x\,\dd t\right|\le C\Pi d_0^{\frac{1}{2}}
+C\sum_{k=0}^{j-1}r_k^{-4}
\iint_{Q_{r_k}}
|\pi-(\pi)_{B_{r_k}}(t)||u|\,\dd x\,\dd t.
\end{aligned}
\]
Applying \eqref{eq:pressure-flux} with $R=r_k$ and using
$\log(r_0/r_k)=k\log2$, we obtain
\[
\begin{aligned}
\sum_{k=0}^{j-1}r_k^{-4}
\iint_{Q_{r_k}}
|\pi-(\pi)_{B_{r_k}}(t)||u|\,\dd x\,\dd t
&\le
CGm_0^{\frac{3}{2}}\sum_{k=0}^{j-1}r_k^2(1+k)
+Cm_0^{\frac{1}{2}}(\Pi+Gd_0)\sum_{k=0}^{j-1}r_k\\
&\le
C\bigl[Gm_0^{\frac{3}{2}}+m_0^{\frac{1}{2}}(\Pi+Gd_0)\bigr].
\end{aligned}
\]
Thus
\begin{equation}\label{eq:pressure-sum}
\begin{aligned}
\left|\int_{-\frac{1}{16}}^{\tau}\int_{B_{\frac{1}{4}}}
\pi u\cdot\nabla\psi_j\,\dd x\,\dd t\right|
\le C\left[
\Pi d_0^{\frac{1}{2}}
+Gm_0^{\frac{3}{2}}
+m_0^{\frac{1}{2}}(\Pi+Gd_0)
\right].
\end{aligned}
\end{equation}

Applying \eqref{eq:dyadic-flux} on the same annuli and
\eqref{eq:fixed-energy} on the fixed outer region, we have
\begin{equation}\label{eq:local-nlogn}
    \begin{aligned}
&\int_{-\frac{1}{16}}^{\tau}\int_{B_{\frac{1}{4}}}
(n|\log n|+|\nabla\sqrt{c+1}|^2+|u|^2)
|(\partial_t+\Delta)\psi_j|\,\dd x\,\dd t\\
&\quad+\int_{-\frac{1}{16}}^{\tau}\int_{B_{\frac{1}{4}}}
\bigl[(n|\log n|+n)|\nabla\sqrt{c+1}|+|\nabla\sqrt{c+1}|^3+(n|\log n|+|\nabla\sqrt{c+1}|^2)|u|\bigr]
|\nabla\psi_j|\,\dd x\,\dd t\\
&\quad+\int_{-\frac{1}{16}}^{\tau}\int_{B_{\frac{1}{4}}}
n|F||u|\psi_j\,\dd x\,\dd t\\
&\le C_{A,G}\bigl(d_0+m_0^{\frac{25}{18}}+m_0^{\frac{3}{2}}\bigr).
\end{aligned}
\end{equation}
Testing \eqref{eq:local-mass} with $\psi_j$ also gives
\begin{equation}\label{eq:weighted-mass}
\sup_{-\frac{1}{16}<\tau<0}\int_{B_{\frac{1}{4}}}
n(\tau,x)\psi_j(\tau,x)\,\dd x
\le C_A(d_0+m_0^{\frac{3}{2}}).
\end{equation}

By \eqref{eq:negative-local}, \eqref{eq:kernel-test},
and \eqref{eq:weighted-mass}, we obtain
\begin{equation}\label{eq:weighted-negative}
\int_{B_{\frac{1}{4}}}\bigl(n(\tau,x)\log n(\tau,x)\bigr)_-\psi_j(\tau,x)\,\dd x
\le C_A(d_0+m_0^{\frac{3}{2}})^{\frac{3}{4}}.
\end{equation}
Apply \eqref{eq:joint-local} with $\psi=\psi_j$ and add
$2\int_{B_{1/4}}(n\log n)_-\psi_j(\tau)\,\dd x$ to both sides.
The terminal entropy becomes
$\int_{B_{1/4}}n|\log n|\psi_j(\tau)\,\dd x$.
For $-r_j^2<\tau<0$, \eqref{eq:kernel-test} gives
\[
\begin{aligned}
&r_j^{-3}\int_{B_{r_j}}
(n+n|\log n|+|\nabla\sqrt{c+1}|^2+|u|^2)(\tau,x)\,\dd x\\
&\le C\int_{B_{\frac{1}{4}}}
(n+n|\log n|+|\nabla\sqrt{c+1}|^2+|u|^2)(\tau,x)\psi_j(\tau,x)\,\dd x,
\end{aligned}
\]
and
\[
\begin{aligned}
&r_j^{-3}\int_{-r_j^2}^{\tau}\int_{B_{r_j}}
\bigl(|\nabla\sqrt n|^2+|\nabla^2\sqrt{c+1}|^2+|\nabla u|^2\bigr)
\,\dd x\,\dd t\\
&\le C\int_{-\frac{1}{16}}^{\tau}\int_{B_{\frac{1}{4}}}
\bigl(|\nabla\sqrt n|^2+|\nabla^2\sqrt{c+1}|^2+|\nabla u|^2\bigr)
\psi_j\,\dd x\,\dd t.
\end{aligned}
\]
Taking the time supremum for the terminal terms and letting
$\tau\uparrow0$ for the dissipation, by \eqref{eq:Er}, \eqref{eq:pressure-sum}--\eqref{eq:weighted-negative}, we obtain
\begin{equation}\label{eq:induction}
\begin{aligned}
E(r_j)\le C_{A,G,q}\bigl[
d_0+m_0^{\frac{25}{18}}+m_0^{\frac{3}{2}}
+\Pi m_0^{\frac{1}{2}}+Gd_0m_0^{\frac{1}{2}}+(d_0+m_0^{\frac{3}{2}})^{\frac{3}{4}}\bigr],
\end{aligned}
\end{equation}
where we used $\Pi d_0^{1/2}\le Cd_0$ by
\eqref{eq:fixed-energy}.
Recalling \eqref{eq:dyadic-parameters}, we have
$d_0=C_0^{-2}m_0^2$ and $\Pi\le C C_0^{-1}m_0$.
For $m_0\le1$, it follows that
\[
\begin{aligned}
&d_0+m_0^{\frac{25}{18}}+m_0^{\frac{3}{2}}
+\Pi m_0^{\frac{1}{2}}+Gd_0m_0^{\frac{1}{2}}
+(d_0+m_0^{\frac{3}{2}})^{\frac{3}{4}}\\
&\qquad\le C_{G,C_0}
\left(m_0^2+m_0^{\frac{25}{18}}+m_0^{\frac{3}{2}}
+m_0^{\frac{5}{2}}+m_0^{\frac{9}{8}}\right)
\le C_{G,C_0}m_0^{\frac{9}{8}}.
\end{aligned}
\]
Thus \eqref{eq:induction} yields
\[
E(r_j)\le C_{A,G,q,C_0}m_0^{\frac{9}{8}}\le m_0
\]
when $d_0$ is sufficiently small.
This completes the induction.

\smallskip
\noindent\textit{\bf Step 6: The pointwise estimate.}
Set $t_j=-r_j^2/2$. Since $-r_j^2<t_j<0$, the definition
\eqref{eq:Er} and $E(r_j)\le m_0$ imply
\[
r_j^{-3}\int_{B_{r_j}}
\bigl(n(t_j,x)+|\nabla\sqrt{c+1}(t_j,x)|^2+|u(t_j,x)|^2\bigr)\,\dd x
\le m_0.
\]
All points $(t_j,x)$ with $x\in B_{r_j}$ converge uniformly to $(0,0)$.
Continuity at this single point therefore gives
\[
\begin{aligned}
&|B_1|\bigl(n(0,0)+|\nabla\sqrt{c+1}(0,0)|^2+|u(0,0)|^2\bigr)\\
&\qquad=\lim_{j\to\infty}r_j^{-3}\int_{B_{r_j}}
\bigl(n(t_j,x)+|\nabla\sqrt{c+1}(t_j,x)|^2+|u(t_j,x)|^2\bigr)\,\dd x
\le m_0.
\end{aligned}
\]
Consequently,
\[
\begin{aligned}
n(0,0)+|\nabla\sqrt{c+1}(0,0)|^2+|u(0,0)|^2
&\le Cm_0\\
&\le C_{A,G,q}(S^{\frac{3}{2q}}+\Pi^2)^{\frac12}
\le C_{A,G,q}(S^{\frac{3}{4q}}+\Pi),
\end{aligned}
\]
which proves \eqref{eq:point-bound}.
\end{proof}

\subsection{Proof of Theorem~\ref{thm:epsilon}}

\begin{proof}[Proof of Theorem~\ref{thm:epsilon}]
Next we prove a pointwise estimate, which is divided into four steps.

\smallskip
\noindent\textit{\bf Step 1: Interior estimates for the signal.}
We first prove that, if $\|u\|_{L^5(Q_1)}$ is sufficiently small, then
\begin{equation}\label{eq:interior-MR}
\|\nabla^2c\|_{L^p(Q_{\frac{7}{8}})}+\|c_t\|_{L^p(Q_{\frac{7}{8}})}
\le C_pA\bigl(1+\|n\|_{L^p(Q_1)}\bigr).
\end{equation}
Fix $-(7/8)^2<\tau<0$ and define
\begin{equation}\label{eq:MR-truncated-norm}
Q_r^\tau=(-r^2,\tau)\times B_r,\qquad
X_\tau(r)=\|\nabla^2c\|_{L^p(Q_r^\tau)}+\|c_t\|_{L^p(Q_r^\tau)}.
\end{equation}
For $7/8\le r<R<1$, choose $\eta\in C_c^\infty((-R^2,0]\times B_R)$ such that
\[
0\le\eta\le1,\qquad \eta=1\ \text{on }Q_r,\qquad
|\nabla\eta|\le C(R-r)^{-1},\qquad
|\eta_t|+|\nabla^2\eta|\le C(R-r)^{-2}.
\]
To apply the heat equation estimates, set $w=\eta c$ in $B_R$ and extend it by zero outside $B_R$. By $\eqref{eq:system}_2$, we have
\begin{equation}\label{eq:cutoff-signal-heat}
w_t-\Delta w=-u\cdot\nabla w-\eta nc+c(\eta_t-\Delta\eta+u\cdot\nabla\eta)-2\nabla\eta\cdot\nabla c.
\end{equation}
Since $\eta$ vanishes near $t=-R^2$, the extended function $w$ has zero initial data at $t=-1$. Let $g=w_t-\Delta w$. Then
\[
w(t)=\int_{-1}^t e^{(t-s)\Delta}g(s)\,\dd s,\qquad -1<t<\tau.
\]
The maximal $L^p$ regularity of the heat equation gives
\begin{equation}\label{eq:causal-heat-MR}
\|w_t\|_{L^p}+\|\nabla^2w\|_{L^p}\le C_p\|g\|_{L^p}.
\end{equation}
Moreover, the parabolic Sobolev embedding yields
\begin{equation}\label{eq:parabolic-gradient-embedding}
\|\nabla w\|_{L^{\frac{5p}{5-p}}}\le C_p\|g\|_{L^p},\qquad 1<p<5,
\end{equation}
where these norms are taken over $(-1,\tau)\times\R^3$. The constants are independent of $\tau$; see \cite[Chapter~IV, \S9]{LSU1968}.

We now estimate the right-hand side of \eqref{eq:cutoff-signal-heat}. Since
\[
\frac{1}{p}=\frac{1}{5}+\frac{5-p}{5p},
\]
H\"older's inequality and \eqref{eq:parabolic-gradient-embedding} give
\[
\|u\cdot\nabla w\|_{L^p}
\le\|u\|_{L^5(Q_1)}\|\nabla w\|_{L^{\frac{5p}{5-p}}}
\le C_p\|u\|_{L^5(Q_1)}\|g\|_{L^p}.
\]
Using $0\le c\le A$, the bounds for $\eta$, and $\|u\|_{L^p(Q_R^\tau)}\le C_p\|u\|_{L^5(Q_1)}$, we obtain
\[
\begin{aligned}
\|g\|_{L^p}
&\le C_p\|u\|_{L^5(Q_1)}\|g\|_{L^p}
+A\|n\|_{L^p(Q_1)}+C_pA(R-r)^{-2}\\
&\quad+C_pA(R-r)^{-1}\|u\|_{L^5(Q_1)}
+C_p(R-r)^{-1}\|\nabla c\|_{L^p(Q_R^\tau)}.
\end{aligned}
\]
Choose $\varepsilon_*>0$ such that
\[
\|u\|_{L^5(Q_1)}\le\varepsilon_*\le\min\{1,(2C_p)^{-1}\}.
\]
Then the first term on the right can be absorbed into the left. Since $R-r<1$, we have
\[
\|g\|_{L^p}\le C_pA\bigl(\|n\|_{L^p(Q_1)}+(R-r)^{-2}\bigr)+C_p(R-r)^{-1}\|\nabla c\|_{L^p(Q_R^\tau)}.
\]
As $w=c$ on $Q_r^\tau$, \eqref{eq:MR-truncated-norm} and \eqref{eq:causal-heat-MR} imply
\[
X_\tau(r)\le C_pA\bigl(\|n\|_{L^p(Q_1)}+(R-r)^{-2}\bigr)+C_p(R-r)^{-1}\|\nabla c\|_{L^p(Q_R^\tau)}.
\]
It remains to control $\|\nabla c\|_{L^p(Q_R^\tau)}$. By interpolation on $B_R$,
\[
\|\nabla c\|_{L^p(Q_R^\tau)}
\le\epsilon\|\nabla^2c\|_{L^p(Q_R^\tau)}+C_p\epsilon^{-1}\|c\|_{L^p(Q_R^\tau)},\qquad 0<\epsilon\le1.
\]
Since $\|c\|_{L^p(Q_R^\tau)}\le C_pA$, the preceding estimate becomes
\[
\begin{aligned}
X_\tau(r)
&\le C_p\epsilon(R-r)^{-1}X_\tau(R)+C_pA\|n\|_{L^p(Q_1)}+C_pA\bigl((R-r)^{-2}+\epsilon^{-1}(R-r)^{-1}\bigr).
\end{aligned}
\]
For $0<\theta<1$, choosing $\epsilon=\theta(R-r)/C_p$ gives
\begin{equation}\label{eq:MR-nested-estimate}
X_\tau(r)\le\theta X_\tau(R)+C_{\theta,p}A\bigl(1+\|n\|_{L^p(Q_1)}\bigr)(R-r)^{-2}.
\end{equation}
Apply \eqref{eq:MR-nested-estimate} only for
$7/8\le r<R\le15/16$. For each fixed $\tau<0$, classicality in the open
cylinder $Q_1$ gives
\[
0\le X_\tau(r)\le X_\tau(15/16)<\infty
\qquad(7/8\le r\le15/16).
\]
Thus by Lemma~\ref{lem:Gia} we have
\[
X_\tau(7/8)\le C_pA\bigl(1+\|n\|_{L^p(Q_1)}\bigr),
\]
with $C_p$ independent of $\tau$, which proves \eqref{eq:interior-MR}.

We now use \eqref{eq:interior-MR} to estimate $\nabla c$. The local Gagliardo--Nirenberg inequality and $0\le c\le A$ give
\[
\|\nabla c(t)\|_{L^{2p}(B_{\frac{3}{4}})}^{2p}
\le C_pA^p\left(\|\nabla^2c(t)\|_{L^p(B_{\frac{7}{8}})}^p+A^p\right).
\]
Integrating in time and using \eqref{eq:interior-MR}, we obtain
\begin{equation}\label{eq:c-gradient-L2p}
\|\nabla c\|_{L^{2p}(Q_{\frac{3}{4}})}
\le C_pA\bigl(1+\|n\|_{L^p(Q_1)}\bigr)^{\frac{1}{2}}.
\end{equation}
For the range $3/2<p\le5/3$, we have
\[
\frac{5p}{5-p}<2p,\qquad
\frac{1}{p}=\frac{1}{5}+\frac{5-p}{5p}.
\]
Therefore, H\"older's inequality and \eqref{eq:c-gradient-L2p} yield, whenever $\|n\|_{L^p(Q_1)}\le1$,
\begin{equation}\label{eq:drift-source}
\begin{aligned}
\|u\cdot\nabla c\|_{L^p(Q_{\frac{3}{4}})}
\le\|u\|_{L^5(Q_{\frac{3}{4}})}\|\nabla c\|_{L^{\frac{5p}{5-p}}(Q_{\frac{3}{4}})}\le C_p\|u\|_{L^5(Q_1)}\|\nabla c\|_{L^{2p}(Q_{\frac{3}{4}})}\le C_{A,p}\|u\|_{L^5(Q_1)}.
\end{aligned}
\end{equation}

\smallskip
\noindent\textit{\bf Step 2: The signal gradient at a smaller scale.}
Fix $-1/4<\tau<0$ and set
\[
Q_{3/4}^{\tau}=(-9/16,\tau)\times B_{3/4},\qquad
\partial_pQ_{3/4}^{\tau}
=(\{-9/16\}\times\overline B_{3/4})
\cup([-9/16,\tau)\times\partial B_{3/4}).
\]
On this strict-top cylinder, let $a^{\tau}$ solve
\begin{align}\label{heat-a}
(a^{\tau})_t-\Delta a^{\tau}=0,\qquad
 a^{\tau}=c\quad\text{on }\partial_pQ_{3/4}^{\tau},
\end{align}
and put $w^{\tau}=c-a^{\tau}$. The lower-face and lateral data for
$a^{\tau}$ agree at their intersection, since they are restrictions of
the same classical function $c$. No data are prescribed on the upper face.
The maximum principle gives $0\le a^{\tau}\le A$, and hence
\begin{equation}\label{eq:caloric-decomposition}
\begin{aligned}
(w^{\tau})_t-\Delta w^{\tau}&=-nc-u\cdot\nabla c
 &&\text{in }Q_{3/4}^{\tau},\\
w^{\tau}&=0
 &&\text{on }\partial_pQ_{3/4}^{\tau},\qquad |w^{\tau}|\le A.
\end{aligned}
\end{equation}
For $-1/4<\tau_1<\tau_2<0$, uniqueness of the heat problem gives
\[
a^{\tau_2}=a^{\tau_1},\qquad w^{\tau_2}=w^{\tau_1}
\quad\text{on }Q_{3/4}^{\tau_1}.
\]
These consistent solutions define $a,w$ on the open cylinder $Q_{3/4}$,
with $c=a+w$, $0\le a\le A$, and $|w|\le A$.

Dirichlet maximal $L^p$ regularity
\cite[Chapter~IV, \S9]{LSU1968} gives
\[
\|(w^{\tau})_t\|_{L^p(Q_{3/4}^{\tau})}
+\|D^2w^{\tau}\|_{L^p(Q_{3/4}^{\tau})}
\le C_p\|nc+u\cdot\nabla c\|_{L^p(Q_{3/4}^{\tau})}.
\]
The constant is independent of $\tau$: extend the forcing by zero after
$\tau$, apply the zero-initial, zero-Dirichlet estimate on the fixed
interval $(-9/16,0)$, and restrict the resulting solution to
$(-9/16,\tau)$. Using \eqref{eq:drift-source} and then letting
$\tau\uparrow0$, we obtain
\begin{equation}\label{ineq:wt-D2w}
\begin{aligned}
\|w_t\|_{L^p(Q_{3/4})}+\|D^2w\|_{L^p(Q_{3/4})}
&\le C_p\|nc+u\cdot\nabla c\|_{L^p(Q_{3/4})}\\
&\le C_{A,p}\left(\|n\|_{L^p(Q_1)}+\|u\|_{L^5(Q_1)}\right).
\end{aligned}
\end{equation}
We next estimate $\nabla w$. 
We have \[
\begin{aligned}
\int_{B_{3/4}}|\nabla w|^{2p}\,\dd x
&=-\int_{B_{3/4}}w\,\diver\bigl(|\nabla w|^{2p-2}\nabla w\bigr)\,\dd x\\
&\le C_p\|w\|_\infty\int_{B_{3/4}}|\nabla w|^{2p-2}|D^2w|\,\dd x\\
&\le C_p\|w\|_\infty\left(\int_{B_{3/4}}|\nabla w|^{2p}\,\dd x\right)^{\frac{p-1}{p}}
\left(\int_{B_{3/4}}|D^2w|^p\,\dd x\right)^{\frac1p}.
\end{aligned}
\]
Dividing by the gradient factor and raising both sides to the power $p$, then integrating in time, we obtain from $|w|\le A$ and \eqref{ineq:wt-D2w} that
\begin{equation}\label{eq:w-small}
\begin{aligned}
\iint_{Q_{3/4}}|\nabla w|^{2p}\,\dd x\,\dd t
&\le C_pA^p\iint_{Q_{3/4}}|D^2w|^p\,\dd x\,\dd t\\
&\le C_{A,p}\left(\|n\|_{L^p(Q_1)}^p+\|u\|_{L^5(Q_1)}^p\right).
\end{aligned}
\end{equation}

On the other hand, the interior gradient estimate for the heat equation and $0\le a\le A$ give
\begin{equation}\label{eq:a-small}
\|\nabla a\|_{L^\infty(Q_{1/2})}\le CA.
\end{equation}
Since $c=a+w$, we have
\[
|\nabla\sqrt{c+1}|^{2p}
=\frac{|\nabla a+\nabla w|^{2p}}{2^{2p}(c+1)^p}
\le C_p\bigl(|\nabla a|^{2p}+|\nabla w|^{2p}\bigr).
\]
Thus, by \eqref{eq:w-small} and \eqref{eq:a-small}, for $0<\kappa\le1/4$,
\begin{equation}\label{eq:gradient-reduction}
\begin{aligned}
&\kappa^{2p-5}\iint_{Q_\kappa}\bigl(n^p+|\nabla\sqrt{c+1}|^{2p}\bigr)\,\dd x\,\dd t\\
&\le C_{A,p}\kappa^{2p-5}\left(\|n\|_{L^p(Q_1)}^p+\|u\|_{L^5(Q_1)}^p\right)+C_pA^{2p}\kappa^{2p}.
\end{aligned}
\end{equation}
For the velocity, H\"older's inequality and $|Q_\kappa|=C\kappa^5$ give
\[
\begin{aligned}
\kappa^{2p-5}\iint_{Q_\kappa}|u|^{2p}\,\dd x\,\dd t
&\le\kappa^{2p-5}|Q_\kappa|^{1-\frac{2p}{5}}\|u\|_{L^5(Q_\kappa)}^{2p}\\
&\le C_p\|u\|_{L^5(Q_1)}^{2p}.
\end{aligned}
\]
For the pressure, subtracting $(\pi)_{B_1}(t)$ and applying Jensen's inequality, we obtain
\[
\|\pi-(\pi)_{B_\kappa}(t)\|_{L^2(-\kappa^2,0;L^3(B_\kappa))}
\le2\|\pi-(\pi)_{B_1}(t)\|_{L^2(-\kappa^2,0;L^3(B_\kappa))}\le2\Pi.
\]
Recall the scaling \eqref{eq:scaling},
$u_\kappa(s,y)=\kappa u(\kappa^2s,\kappa y),~
\pi_\kappa(s,y)=\kappa^2\pi(\kappa^2s,\kappa y)$.
Under the change of variables $x=\kappa y$ and $t=\kappa^2s$, we have
    \begin{align}\label{uk_2p}
\iint_{Q_1}|u_\kappa|^{2p}\,\dd y\,\dd s
&=\kappa^{2p-5}\iint_{Q_\kappa}|u|^{2p}\,\dd x\,\dd t\le C_p\|u\|_{L^5(Q_1)}^{2p}.
\end{align}
Similarly, since
$(\pi_\kappa)_{B_1}(s)
=\kappa^2(\pi)_{B_\kappa}(\kappa^2s)$,
a change of variables gives
\begin{equation}\label{eq:Pi}
    \begin{aligned}
\|\pi_\kappa-(\pi_\kappa)_{B_1}(s)\|_{L^2(-1,0;L^3(B_1))}&=\|\pi-(\pi)_{B_\kappa}(t)\|_{L^2(-\kappa^2,0;L^3(B_\kappa))}\\
&\le2\|\pi-(\pi)_{B_1}(t)\|_{L^2(-1,0;L^3(B_1))}=2\Pi,
\end{aligned}
\end{equation}
where the inequality follows from Jensen's inequality and $Q_\kappa\subset Q_1$.

\smallskip
\noindent\textit{\bf Step 3: the pointwise estimate.}
Let $\varepsilon_0>0$ be the smallness threshold in Proposition \ref{prop:point} with exponent $2p$. First, choose $\kappa=\kappa(A,G,p)\in(0,1/4]$ such that
\[
C_pA^{2p}\kappa^{2p}\le\frac{\varepsilon_0}{4}.
\]
Then choose the threshold $\varepsilon_*$ in Theorem \ref{thm:epsilon} sufficiently small so that
$S+\Pi^2\le\varepsilon_0$ by \eqref{eq:gradient-reduction} to \eqref{eq:Pi}.
Moreover, $\|F_\kappa\|_\infty\le\kappa G\le G$. For this pointwise estimate, suppose that $n$, $c$, $\nabla c$,
and $u$ extend continuously to $(0,0)$.
Applying that proposition, rescaling, and using
$|\nabla c|^2=4(c+1)|\nabla\sqrt{c+1}|^2$, we obtain
\begin{equation}\label{eq:point-density-only}
n(0,0)+|\nabla c(0,0)|^2+|u(0,0)|^2
\le C_{A,G,p}.
\end{equation}

\smallskip
\noindent\textit{\bf Step 4: the estimate on $Q_{R/2}(z_0)$.}
Let $z=(t,x)\in Q_{R/2}(z_0)$ and set $\rho=R/2$.
Then $Q_\rho(z)\subset Q_R(z_0)$, and the solution is
classical in a neighborhood of $z$.
The local quantities satisfy
\[
\begin{aligned}
\Acal_p(\rho;z)
\le2^{5-2p}\Acal_p(R;z_0),
U(\rho;z)\le U(R;z_0),
P(\rho;z)\le2P(R;z_0).
\end{aligned}
\]
The pressure estimate follows by subtracting the average
over $B_R(x_0)$ and applying Jensen's inequality to the
average over $B_\rho(x)$.
Also $\rho\|F\|_\infty\le G$.

Taking $\varepsilon_*$ sufficiently small, we may apply
\eqref{eq:point-density-only} to the solution rescaled \eqref{eq:scaling} from $Q_\rho(z)$ to $Q_1$. Hence
\[
n(z)+|\nabla c(z)|^2+|u(z)|^2
\le C_{A,G,p}\rho^{-2}
\le C_{A,G,p}R^{-2}.
\]
Taking the supremum over $z\in Q_{R/2}(z_0)$ proves
\eqref{eq:epsilon-conc}.
\end{proof}

\section{Terminal signal and normalized density compactness}
\label{sec:terminal-normalization}



Note that if $c$ is only known to be classical for $t<T$, then its value
at $(T,\cdot)$ is not yet defined. In the following Proposition \ref{prop:terminal}, we first construct a terminal
value $c_*(T,x_0)$ ($\forall x_0\in\mathbb R^3$) and prove that
\[
c(T+r^2s,x_0+ry)\longrightarrow c_*(T,x_0)
\quad\text{strongly in }L^a(Q_1),\qquad r\downarrow0,
\]
for every $1\le a<\infty$. Thus, after rescaling smaller
backward cylinders around $(T,x_0)$ to $Q_1$, the signal
converges to a constant.
We then establish \eqref{eq:normalized-gain} for the normalized
density, with a constant independent of the normalization
parameter $\lambda$. These estimates will be used in
Section~\ref{sec:concentration} to exclude density concentration.

\subsection{The terminal signal and its backward limit}
\label{sec:terminal}

\begin{Prop}\label{prop:terminal}
Let $0<T<\infty$ and let $u$ be divergence-free with $u\in L^\infty(0,T;L^3(\mathbb R^3))$. Suppose that
\[
c_t-\Delta c+u\cdot\nabla c=-nc,\qquad 0\le c\le A,\qquad n\ge0,\qquad n\in L^1((0,T)\times\mathbb R^3).
\]
Assume that $(n,c,u)$ is classical on $(0,T)\times\mathbb R^3$ and that $n,u$ are bounded on $[t_0,t_1]\times\mathbb R^3$ for every $0<t_0<t_1<T$. Then there exists an upper semicontinuous function $c_*:(0,T]\times\mathbb R^3\to[0,A]$ such that $c_*=c$ for $t<T$ and, for every $x_0\in\mathbb R^3$,
\[
\limsup_{\substack{(t,x)\to(T,x_0)\\t<T}}c(t,x)=c_*(T,x_0).
\]
Moreover, for every $1\le a<\infty$,
\begin{equation}\label{eq:constant-tangent}
c(T+r^2s,x_0+ry)\longrightarrow c_*(T,x_0)\quad\text{strongly in }L^a(Q_1),\qquad r\downarrow0.
\end{equation}
\end{Prop}

\begin{proof}
If $A=0$, the conclusion is immediate. Next we only assume 
that $A>0$.

\smallskip
\noindent\textit{\bf Step 1: The kernel representation.}
To construct the kernel at time $T$, extend $u$ by $0$ for $t\ge T$. This preserves $\nabla\cdot u=0$
and the $L^\infty_tL^3_x$ bound.
 The operator
$\partial_t-\Delta+u\cdot\nabla$ has a positive fundamental
solution $\Gamma(t,x;s,y)$ satisfying
\begin{equation}\label{eq:Gaussian}
\frac{e^{-C\frac{|x-y|^2}{t-s}}}{C(t-s)^{\frac{3}{2}}}
\le \Gamma(t,x;s,y)
\le \frac{C e^{-\frac{|x-y|^2}{C(t-s)}}}{(t-s)^{\frac{3}{2}}},
\qquad 0<s<t,
\end{equation}
where $C$ depends only on the dimension and
$\|u\|_{L^\infty(0,T;L^3(\mathbb R^3))}$ (see, for example,  Qian--Xi \cite[Theorem 5.3]{QianXi2019-jlms} or \cite[Theorem 1.2]{QianXi2019}).
Moreover, for fixed $(s,y)$, the kernel is continuous in $(t,x)$
when $t>s$ satisfying
\begin{align}\label{eq:fundemental}
    \int_{\mathbb R^3}\Gamma(t,x;s,y)\,dy
=\int_{\mathbb R^3}\Gamma(t,x;s,y)\,dx=1
\end{align}
and
\begin{align}\label{eq:Tau-exchange}
    \Gamma(t,x;s,y)
=\int_{\mathbb R^3}
\Gamma(t,x;\sigma,z)\Gamma(\sigma,z;s,y)\,dz,
\qquad s<\sigma<t.
\end{align}
For $\|g\|_\infty<\infty$, define
\begin{align}\label{eq:drift-evolution}
&\bigl[\mathcal U(t,s)g(s)\bigr](x)
:=\int_{\mathbb R^3}\Gamma(t,x;s,y)g(s,y)\,dy.
\end{align}
Fix $t_b=T/2$ and let $f=nc$. Since $0\le c\le A$, $n\geq 0$, and $n\in L^1((0,T)\times\mathbb R^3)$, we have $f\ge0$ and
\[
\|f\|_{L^1((0,T)\times\mathbb R^3)}
\le A\|n\|_{L^1((0,T)\times\mathbb R^3)}<\infty.
\]
For \(t_b<t\le T\), define the following quantities, initially
allowing the nonnegative integral defining \(V_0(T,x)\) to take
the value \(+\infty\):
\begin{equation}
    \begin{aligned}\label{eq:H-V}
&H_0(t,x)
:=\mathcal U(t,t_b)c(t_b)(x),\\
&V_0(t,x)
:=\int_{t_b}^{t}\int_{\mathbb R^3}
\Gamma(t,x;s,y)f(s,y)\,dy\,ds=\int_{t_b}^{t}\bigl[\mathcal U(t,s)f(s)\bigr](x)\,ds.
\end{aligned}
\end{equation}
 By the assumptions, $u$ and $f$ are bounded on $[t_b,T']\times\mathbb R^3$, where $T'<T$.  For \eqref{eq:Gaussian}, \eqref{eq:fundemental} and any compact set $K\subset\mathbb R^3$, we have
\[
\sup_{x\in K}|\mathcal U(t,t_b)c(t_b)(x)-c(t_b,x)|
\le \sup_{\substack{x\in K\\|x-y|\le\delta}}|c(t_b,y)-c(t_b,x)|+CA|t-t_b|^{-\frac32}\int_{|z|>\delta}e^{-\frac{|z|^2}{C|t-t_b|}}\,dz,
\]
which implies that
\[
\mathcal U(t,t_b)c(t_b)\longrightarrow c(t_b)\quad\text{locally uniformly as }t\downarrow t_b.
\]
Moreover, \eqref{eq:fundemental} gives
\[
\left\|\int_{t_b}^{t}\mathcal U(t,s)f(s)\,ds\right\|_\infty\le (t-t_b)\|f\|_{L^\infty([t_b,T']\times\mathbb R^3)}.
\]
Consequently,
\[
\widetilde c(t):=\mathcal U(t,t_b)c(t_b)-\int_{t_b}^{t}\mathcal U(t,s)f(s)\,ds
\]
is bounded and converges locally uniformly to $c(t_b)$ as $t\downarrow t_b$.
Since $f\in L^1\cap L^\infty$ on this strip, $f\in L^2$ there.
Let $v_R$ denote the same expression with $c(t_b)$ replaced by
$\mathbf{1}_{B_R}c(t_b)$. By the $L^2$ Duhamel formula,
\[
(\partial_t-\Delta+u\cdot\nabla)v_R=-f.
\]
Moreover, by \eqref{eq:Gaussian},
\[
|v_R(t,x)-\widetilde c(t,x)|
\le A\int_{|y|>R}\Gamma(t,x;t_b,y)\,dy\longrightarrow0
\]
locally uniformly for $t_b<t\le T'$. Passing to the limit in the
weak formulation gives
$(\partial_t-\Delta+u\cdot\nabla)(\widetilde c-c)=0$.
Since $\widetilde c-c$ is bounded and has zero initial trace,
the maximum principle gives $\widetilde c=c$, and hence
\begin{equation}\label{eq:terminal-duhamel}
c(t)=\mathcal U(t,t_b)c(t_b)-\int_{t_b}^{t}\mathcal U(t,s)f(s)\,ds,\qquad t_b<t\le T'.
\end{equation}
To see that the preceding identity holds pointwise, fix \(T'<T\)
and truncate the integral defining \(V_0(t,x)\) to \(s<t-\delta\).
The truncated integral is continuous in \((t,x)\) by kernel
continuity and dominated convergence, since its integrand is bounded
by \(C\delta^{-3/2}f(s,y)\) and \(f\in L^1\).
By the unit mass of the kernel, the omitted part is bounded by
\(\delta\|f\|_{L^\infty([t_b,T']\times\mathbb R^3)}\), uniformly
for \(t_b<t\le T'\). Thus \(V_0\) is continuous for \(t_b<t<T\).
Since \(c\) and \(H_0\) are also continuous there (for \(H_0\), see
Step~2 below), the identity holds everywhere. As \(T'<T\) is
arbitrary, we obtain
\begin{equation}\label{eq:terminal-duhamel-expanded}
c(t,x)+\int_{t_b}^{t}\int_{\mathbb R^3}\Gamma(t,x;s,y)n(s,y)c(s,y)\,dy\,ds=H_0(t,x),\qquad t_b<t<T.
\end{equation}

\smallskip
\noindent\textit{\bf Step 2: The terminal signal value.}
We next show that $H_0$ is continuous for $t>t_b$, including
at $t=T$. Fix $(t_0,x_0)$ with $t_b<t_0\le T$.
For $(t,x)$ sufficiently close to $(t_0,x_0)$,
the bound $|c(t_b)|\le A$ and \eqref{eq:Gaussian} give
\[
\begin{aligned}
|H_0(t,x)-H_0(t_0,x_0)|
\le A\int_{B_R(x_0)}
|\Gamma(t,x;t_b,y)-\Gamma(t_0,x_0;t_b,y)|\,dy+CA\int_{|y-x_0|>R}e^{-\frac{|y-x_0|^2}{C}}\,dy,
\end{aligned}
\]
for $R\ge1$, where $C$ is independent of $R$ and of
$(t,x)$ in this neighborhood.
For fixed $R$, the first term tends to zero by kernel
continuity and dominated convergence. Letting
$(t,x)\to(t_0,x_0)$ and then $R\to\infty$ proves the claim.

Let $\Gamma(t,x;s,y)=0$ for $t\le s$. For fixed $(s,y)$,
positivity and continuity for $t>s$ give
\[
\Gamma(t,x;s,y)
\le \liminf_{j\to\infty}\Gamma(t_j,x_j;s,y),
\qquad\text{whenever }(t_j,x_j)\to(t,x).
\]
Since $f\ge0$, for $t\leq T$ Fatou's lemma yields
\begin{align} \label{eq:lower semicontinuous}
V_0(t,x)
&=\int_{t_b}^{t}\int_{\mathbb R^3}
\Gamma(t,x;s,y)f(s,y)\,dy\,ds \nonumber\\
&\le \int_{t_b}^{T}\int_{\mathbb R^3}
\liminf_{j\to\infty}
\bigl[\Gamma(t_j,x_j;s,y)f(s,y)\bigr]\,dy\,ds\\
&\le \liminf_{j\to\infty}V_0(t_j,x_j). \nonumber
\end{align}
By \eqref{eq:terminal-duhamel-expanded} and the bound of $c$, we have
\[
0\le V_0(t,x)=H_0(t,x)-c(t,x)\le H_0(t,x)\le A,
\qquad t_b<t<T.
\]
By \eqref{eq:lower semicontinuous} and the continuity of $H_0$ up to $t=T$, we obtain
\begin{equation}\label{eq:potential-finite}
\begin{aligned}
0\le V_0(T,x)
\le\liminf_{t\uparrow T}V_0(t,x)
\le\limsup_{t\uparrow T}V_0(t,x)\le H_0(T,x)\le A.
\end{aligned}
\end{equation}
In particular, $V_0(T,x)$ is finite for every $x\in\mathbb R^3$. Define
\begin{equation}\label{eq:cstar}
c_*(t,x):=
\begin{cases}
c(t,x),&0<t<T,\\
H_0(T,x)-V_0(T,x),&t=T.
\end{cases}
\end{equation}
By \eqref{eq:terminal-duhamel-expanded}, we have
$c_*(t,x)=H_0(t,x)-V_0(t,x)$ for $t_b<t<T$, and the same identity holds at $t=T$ by definition. Since $H_0$ is continuous and $V_0$ is lower semicontinuous by \eqref{eq:lower semicontinuous}, for any sequence $(t_j,x_j)\to(T,x)$, we have $t_j>t_b$ for all sufficiently large $j$ and
\[
\begin{aligned}
\limsup_{j\to\infty}c_*(t_j,x_j)
&=\limsup_{j\to\infty}\bigl[H_0(t_j,x_j)-V_0(t_j,x_j)\bigr]\\
&\le H_0(T,x)-\liminf_{j\to\infty}V_0(t_j,x_j)\\
&\le H_0(T,x)-V_0(T,x)=c_*(T,x).
\end{aligned}
\]
Thus, $c_*$ is upper semicontinuous on $(0,T]\times\mathbb R^3$. Moreover, \eqref{eq:potential-finite} gives
\[
0\le c_*(T,x)=H_0(T,x)-V_0(T,x)\le H_0(T,x)\le A.
\]
Together with $0\le c\le A$ for $t<T$, this proves that $0\le c_*\le A$ on $(0,T]\times\mathbb R^3$.

\smallskip
\noindent\textit{\bf Step 3: $L^a$ Convergence on backward cylinders.}
We prove the strong convergence in \eqref{eq:constant-tangent}.
Fix $x_0\in\mathbb R^3$ and write $c_\sharp=c_*(T,x_0)$.
For $0<\tau<T-t_b$, applying $\mathcal U(T,T-\tau)$ to
\eqref{eq:terminal-duhamel-expanded} at $T-\tau$ and using the
semigroup property, we obtain
\[
\begin{aligned}
\mathcal U(T,T-\tau)c(T-\tau)(x_0)
&=\mathcal U(T,T-\tau)H_0(T-\tau)(x_0)\\
&\quad-\int_{t_b}^{T-\tau}\mathcal U(T,T-\tau)\mathcal U(T-\tau,s)f(s)(x_0)\,ds\\
&=H_0(T,x_0)-\int_{t_b}^{T-\tau}\int_{\mathbb R^3}\Gamma(T,x_0;s,y)f(s,y)\,dy\,ds,
\end{aligned}
\]
where we used Tonelli's theorem, since $\Gamma,f\ge0$. Subtracting $c_\sharp=H_0(T,x_0)-V_0(T,x_0)$, we obtain
\begin{equation}\label{eq:tail-potential}
\mathcal U(T,T-\tau)c(T-\tau)(x_0)-c_\sharp
=\int_{T-\tau}^{T}\int_{\mathbb R^3}\Gamma(T,x_0;s,y)f(s,y)\,dy\,ds
\longrightarrow0
\end{equation}
as $\tau\downarrow0$, by \eqref{eq:potential-finite}.

Next, we show that the positive excess vanishes under the terminal evolution.
Since $c_*$ is upper semicontinuous at $(T,x_0)$ and $c_*=c$ for $t<T$, for every $\epsilon>0$ there exists $\delta>0$ such that $c(t,y)\le c_\sharp+\epsilon$ whenever $|y-x_0|+|t-T|<\delta$ and $t<T$. Choosing $d>0$ with $d+d^2<\delta$ and $d^2<T-t_b$, we obtain
\begin{equation}\label{eq:terminal-upper-bound}
c(t,y)\le c_\sharp+\epsilon\quad\text{if }|y-x_0|<d,\quad T-d^2<t<T.
\end{equation}
Thus, for $0<\tau<d^2$, by \eqref{eq:terminal-upper-bound} and $0\le c,c_\sharp\le A$,
\begin{equation}\label{eq:c-c_st}
\bigl(c(T-\tau,y)-c_\sharp\bigr)_+
\le\epsilon\mathbf1_{B_d(x_0)}(y)+A\mathbf1_{\mathbb R^3\setminus B_d(x_0)}(y).
\end{equation}
Moreover, by the Gaussian upper bound in \eqref{eq:Gaussian},
\begin{align}
\int_{|y-x_0|\ge d}\Gamma(T,x_0;T-\tau,y)\,dy
&\le C\tau^{-\frac32}\int_{|y-x_0|\ge d}e^{-\frac{|y-x_0|^2}{C\tau}}\,dy \nonumber\\
&\le C\tau^{-\frac32}e^{-\frac{d^2}{2C\tau}}\int_{\mathbb R^3}e^{-\frac{|y-x_0|^2}{2C\tau}}\,dy \label{eq:outer}\\
&\le Ce^{-\frac{d^2}{C\tau}}. \nonumber
\end{align}
Consequently, by \eqref{eq:drift-evolution}, \eqref{eq:c-c_st}, \eqref{eq:outer} and the unit-mass property,
\begin{equation}\label{eq:terminal-positive-excess}
\begin{aligned}
\mathcal U(T,T-\tau)\bigl(c(T-\tau)-c_\sharp\bigr)_+(x_0)
&\le\epsilon+A\int_{|y-x_0|\ge d}\Gamma(T,x_0;T-\tau,y)\,dy\\
&\le\epsilon+CAe^{-\frac{d^2}{C\tau}}.
\end{aligned}
\end{equation}
Letting $\tau\downarrow0$ and then $\epsilon\downarrow0$, we conclude that
\[
\mathcal U(T,T-\tau)\bigl(c(T-\tau)-c_\sharp\bigr)_+(x_0)\longrightarrow0.
\]
Then we obtain the weighted convergence of the absolute difference.
Using $|z|=2z_+-z$, the unit-mass property and \eqref{eq:tail-potential}, we obtain
\begin{equation}\label{eq:weighted-convergence}
\begin{aligned}
\omega(\tau)
&:=\mathcal U(T,T-\tau)|c(T-\tau)-c_\sharp|(x_0)\\
&=2\mathcal U(T,T-\tau)\bigl(c(T-\tau)-c_\sharp\bigr)_+(x_0)\\
&\quad-\int_{T-\tau}^{T}\int_{\mathbb R^3}\Gamma(T,x_0;s,y)f(s,y)\,dy\,ds
\longrightarrow0
\end{aligned}
\end{equation}
as $\tau\downarrow0$.

{We now convert the weighted convergence into local $L^1$ convergence away from the terminal time slice.}
Fix $0<\sigma<1$. If $\sigma r^2\le\tau\le r^2$ and $|x-x_0|\le r$, the Gaussian lower bound in \eqref{eq:Gaussian} gives
\begin{equation}\label{eq:terminal-kernel-lower}
\Gamma(T,x_0;T-\tau,x)\ge C^{-1}\tau^{-\frac32}e^{-\frac{Cr^2}{\tau}}\ge C_\sigma^{-1}r^{-3}.
\end{equation}
For $r^2<T-t_b$, set $c_r(s,y)=c(T+r^2s,x_0+ry)$. Combining \eqref{eq:weighted-convergence} and \eqref{eq:terminal-kernel-lower}, we have
\begin{equation}\label{eq:tangent-away-top}
\begin{aligned}
\int_{-1}^{-\sigma}\int_{B_1}|c_r-c_\sharp|\,dy\,ds
&=r^{-5}\int_{\sigma r^2}^{r^2}\int_{B_r(x_0)}|c(T-\tau,x)-c_\sharp|\,dx\,d\tau\\
&\le C_\sigma r^{-2}\int_{\sigma r^2}^{r^2}\omega(\tau)\,d\tau\\
&\le C_\sigma\sup_{0<\tau\le r^2}\omega(\tau)\longrightarrow0
\end{aligned}
\end{equation}
as $r\downarrow0$.
Finally, we complete the $L^1$ convergence on the whole cylinder and upgrade it to every finite $L^a$.
Since $0\le c_r,c_\sharp\le A$, the remaining terminal slice satisfies
\[
\int_{-\sigma}^{0}\int_{B_1}|c_r-c_\sharp|\,dy\,ds\le A|B_1|\sigma.
\]
Combining this with \eqref{eq:tangent-away-top}, we obtain
\[
\|c_r-c_\sharp\|_{L^1(Q_1)}
\le C_\sigma\sup_{0<\tau\le r^2}\omega(\tau)+A|B_1|\sigma.
\]
Letting $r\downarrow0$ and then $\sigma\downarrow0$, we conclude that $c_r\to c_\sharp$ strongly in $L^1(Q_1)$. Finally, for every $1\le a<\infty$, since $|c_r-c_\sharp|\le A$,
\[
\|c_r-c_\sharp\|_{L^a(Q_1)}^a\le A^{a-1}\|c_r-c_\sharp\|_{L^1(Q_1)}\longrightarrow0.
\]
This proves \eqref{eq:constant-tangent}.
Moreover, \eqref{eq:constant-tangent} with $a=1$ gives
\[
c_\sharp
=\lim_{r\downarrow0}\frac{1}{|Q_r(T,x_0)|}
  \iint_{Q_r(T,x_0)}c(t,x)\,dx\,dt
\le\lim_{r\downarrow0}\sup_{Q_r(T,x_0)}c
\le c_\sharp.
\]
The last inequality follows from the space-time upper semicontinuity
proved in {\bf Step~2}. Since these backward cylinders form a neighborhood
basis from below, we conclude that
\[
\limsup_{\substack{(t,x)\to(T,x_0);t<T}}c(t,x)=c_*(T,x_0).
\]
\end{proof}



\subsection{Entropy uniform under normalization}
\label{sec:normalization}


\begin{Prop}\label{prop:uniform-entropy}
Let $p_0>5/4$ and $A,K_0<\infty$.
There exists $\eta_0=\eta_0(A,p_0,K_0)>0$ such that
the following holds.
Suppose that $\lambda\ge1$ and $(N,c,b)$ is a classical
solution of
\begin{equation}\label{eq:normalized-system}
\begin{cases}
N_s+b\cdot\nabla N=\Delta N-\nabla\cdot(N\nabla c),\\
c_s+b\cdot\nabla c=\Delta c-\lambda Nc, \quad \nabla\cdot b=0
\end{cases}
\qquad\text{in }Q_1.
\end{equation}
Assume that $N>0$, $0\le c\le A$, and
\begin{equation}\label{eq:normalized-hyp}
\iint_{Q_1}N^{p_0}\le K_0,\qquad
\|b\|_{L^5(Q_1)}^2
+\|\nabla b\|_{L^{\frac{5}{2}}(Q_1)}\le\eta_0.
\end{equation}
Then
\begin{equation}\label{eq:normalized-gain}
\sup_{-\frac{1}{2}<s<0}\int_{B_{\frac{1}{2}}}N(s)
+\iint_{Q_{\frac{1}{2}}}
\left(
|\nabla\sqrt N|^2+N|\nabla c|^2+N^{\frac{5}{3}}
\right)
\le C(A,p_0,K_0),
\end{equation}
and
\begin{equation}\label{eq:normalized-signal-gain}
\begin{aligned}
\frac{1}{\lambda}\sup_{-\frac{1}{2}<s<0}
\int_{B_{\frac{1}{2}}}|\nabla c(s)|^2
+\iint_{Q_{\frac{1}{2}}}
\left(
|\nabla c|^2+\lambda Nc^2
+\frac{|\nabla^2c|^2+|\nabla c|^4}{\lambda}
\right)
\le C(A,p_0,K_0).
\end{aligned}
\end{equation}
The constants are independent of $\lambda$.
\end{Prop}

\begin{proof}
We first impose $\eta_0\le1$, so that
\[
\|b\|_{L^5(Q_1)}\le1,\qquad
\|\nabla b\|_{L^{\frac{5}{2}}(Q_1)}\le1.
\]
Also, $\eta_0$ will be chosen smaller, if necessary, to absorb
the terms involving these norms.
Fix $-1<\tau<0$. In the following, we derive the estimates on
$(-1,\tau)\times B_1$ with constants independent of $\tau$,
and then let $\tau\uparrow0$.

\smallskip
\noindent\textit{\bf Step 1: The signal energy.}
Multiplying $\eqref{eq:normalized-system}_2$ by $c$, we obtain
\begin{equation}\label{eq:normalized-signal-budget-identity}
(\partial_s-\Delta+b\cdot\nabla)\frac{c^2}{2}
=-|\nabla c|^2-\lambda Nc^2.
\end{equation}
Choose $0\le\chi\le1$ in
$C_c^\infty((-1,1)\times B_1)$ such that
$\chi=1$ on $(-15/16,0)\times B_{15/16}$.
Using $\nabla\cdot b=0$, we have
\[
\begin{aligned}
\frac{1}{2}\int_{B_1}c^2(\tau)\chi(\tau)
+\int_{-1}^{\tau}\int_{B_1}
\chi\left(|\nabla c|^2+\lambda Nc^2\right)
&=\frac{1}{2}\int_{-1}^{\tau}\int_{B_1}
c^2\left(\chi_s+\Delta\chi+b\cdot\nabla\chi\right)\\
&\le CA^2\left(1+\|b\|_{L^5(Q_1)}\right),
\end{aligned}
\]
which implies
\begin{equation}\label{eq:outer-signal}
\int_{-\frac{15}{16}}^{0}\int_{B_{\frac{15}{16}}}
\left(|\nabla c|^2+\lambda Nc^2\right)
\le CA^2.
\end{equation}

\smallskip
\noindent\textit{\bf Step 2: The shifted entropy identity.}
We have
\begin{equation}\label{eq:normalized-v-bounds}
1\le c+1\le A+1,\qquad
N\log N-N+1\ge0.
\end{equation}
By $\eqref{eq:normalized-system}_1$, we get
\begin{equation}\label{eq:normalized-density-identity}
\begin{aligned}
(\partial_s-\Delta+b\cdot\nabla)(N\log N-N+1)
&=(\log N)(\partial_s-\Delta+b\cdot\nabla)N
-\frac{|\nabla N|^2}{N}\\
&=-\nabla\cdot(N\log N\nabla c)
+\nabla N\cdot\nabla c-\frac{|\nabla N|^2}{N}.
\end{aligned}
\end{equation}
Using $(\partial_s-\Delta+b\cdot\nabla)c=-\lambda Nc$ and 
\eqref{eq:Fisher-identity}, we have
\[
\begin{aligned}
&(\partial_s-\Delta+b\cdot\nabla)\frac{|\nabla c|^2}{2\lambda(c+1)}+\frac{c+1}{\lambda}|\nabla^2\log(c+1)|^2
+\frac{N|\nabla c|^2}{2(c+1)}
+\frac{N|\nabla c|^2}{2(c+1)^2}\\
&=-\frac{c}{c+1}
\nabla N\cdot\nabla c
-\frac{(\nabla c\otimes\nabla c):\nabla b}{\lambda(c+1)}.
\end{aligned}
\]
Adding this identity to
\eqref{eq:normalized-density-identity}, we obtain
\begin{equation}\label{eq:normalized-identity}
\begin{aligned}
&(\partial_s-\Delta+b\cdot\nabla)
\left(N\log N-N+1+\frac{|\nabla c|^2}{2\lambda(c+1)}\right)\\
&\quad+\frac{|\nabla N|^2}{N}
+\frac{c+1}{\lambda}|\nabla^2\log(c+1)|^2
+\frac{N|\nabla c|^2}{2(c+1)}
+\frac{N|\nabla c|^2}{2(c+1)^2}\\
&=-\nabla\cdot(N\log N\nabla c)
+\frac{\nabla N\cdot\nabla c}{c+1}-\frac{(\nabla c\otimes\nabla c):\nabla b}{\lambda(c+1)}.
\end{aligned}
\end{equation}
By Young's inequality, we get
\begin{equation}\label{eq:normalized-mixed-Young}
\frac{|\nabla N||\nabla c|}{c+1}
\le\frac{1}{2}\frac{|\nabla N|^2}{N}
+\frac{N|\nabla c|^2}{2(c+1)^2}.
\end{equation}
Combining \eqref{eq:normalized-identity} and \eqref{eq:normalized-mixed-Young} gives
\begin{equation}\label{eq:normalized-inequality}
\begin{aligned}
&(\partial_s-\Delta+b\cdot\nabla)
\left(N\log N-N+1+\frac{|\nabla c|^2}{2\lambda(c+1)}\right)\\
&\quad+\frac{1}{2}\frac{|\nabla N|^2}{N}
+\frac{c+1}{\lambda}|\nabla^2\log(c+1)|^2
+\frac{N|\nabla c|^2}{2(c+1)}\\
&\le-\nabla\cdot(N\log N\nabla c)
+|\nabla b|\,\frac{|\nabla c|^2}{\lambda(c+1)}.
\end{aligned}
\end{equation}

\smallskip
\noindent\textit{\bf Step 3: Localized Fisher dissipation.}
For $\eta\in C_c^\infty(B_1)$, using
$\nabla c=(c+1)\nabla\log(c+1)$ and integrating by parts, we obtain
\[
\begin{aligned}
\int_{B_1}\eta^2\frac{|\nabla c|^4}{(c+1)^3}
&=\int_{B_1}\eta^2|\nabla\log(c+1)|^2
\nabla\log(c+1)\cdot\nabla c\\
&=-\int_{B_1}\left(c+1\right)\nabla\cdot
\left(\eta^2|\nabla\log(c+1)|^2\nabla\log(c+1)\right)\\
&=-2\int_{B_1}\frac{\eta^2}{c+1}
\nabla^2\log(c+1):(\nabla c\otimes\nabla c)\\
&\qquad-\int_{B_1}\eta^2\frac{|\nabla c|^2}{c+1}
\Delta\log(c+1)
-2\int_{B_1}\eta\frac{|\nabla c|^2}{(c+1)^2}
\nabla c\cdot\nabla\eta,
\end{aligned}
\]
where we used
$\nabla(|\nabla\log(c+1)|^2)
=2\nabla^2\log(c+1)\,\nabla\log(c+1)$.
Thus, $|\Delta\log(c+1)|\le\sqrt3|\nabla^2\log(c+1)|$
and Young's inequality imply
\[
\begin{aligned}
\int_{B_1}\eta^2\frac{|\nabla c|^4}{(c+1)^3}
&\le(2+\sqrt3)\int_{B_1}
\frac{\eta^2|\nabla c|^2}{c+1}|\nabla^2\log(c+1)|
+2\int_{B_1}|\eta|\frac{|\nabla c|^3}{(c+1)^2}|\nabla\eta|\\
&\le\frac{1}{2}\int_{B_1}\eta^2\frac{|\nabla c|^4}{(c+1)^3}
+C\int_{B_1}\eta^2\left(c+1\right)|\nabla^2\log(c+1)|^2+C\int_{B_1}\frac{|\nabla c|^2}{c+1}|\nabla\eta|^2.
\end{aligned}
\]
Consequently,
\begin{equation}\label{eq:localized-quartic-estimate}
\begin{aligned}
\int_{B_1}\eta^2\frac{|\nabla c|^4}{(c+1)^3}
\le C\int_{B_1}\eta^2\left(c+1\right)|\nabla^2\log(c+1)|^2+C\int_{B_1}\frac{|\nabla c|^2}{c+1}|\nabla\eta|^2.
\end{aligned}
\end{equation}
Moreover,
\begin{equation}\label{eq:W-differentiation}
\begin{aligned}
\nabla\left(\frac{\nabla c}{\sqrt{\lambda(c+1)}}\right)
&=\frac{\nabla^2c}{\sqrt{\lambda(c+1)}}
-\frac{\nabla c\otimes\nabla c}{2\sqrt\lambda(c+1)^{\frac{3}{2}}}\\
&=\sqrt{\frac{c+1}{\lambda}}\nabla^2\log(c+1)
+\frac{\nabla c\otimes\nabla c}{2\sqrt\lambda(c+1)^{\frac{3}{2}}}.
\end{aligned}
\end{equation}
Combining \eqref{eq:localized-quartic-estimate} and
\eqref{eq:W-differentiation}, we obtain
\begin{equation}\label{eq:W-gradient}
\begin{aligned}
\int_{B_1}\left|\nabla\left(
\frac{\eta\nabla c}{\sqrt{\lambda(c+1)}}\right)\right|^2\le\frac{C}{\lambda}\int_{B_1}\eta^2
\left(c+1\right)|\nabla^2\log(c+1)|^2
+C\int_{B_1}\frac{|\nabla c|^2}{\lambda(c+1)}|\nabla\eta|^2.
\end{aligned}
\end{equation}

\smallskip
\noindent\textit{\bf Step 4: Localization and absorption.}
We first introduce spatial and temporal cutoffs and define the localized entropy quantity.
Choose $\eta\in C_c^\infty(B_{7/8})$ with $0\le\eta\le1$ and $\eta=1$ on $B_{3/4}$. Choose a nondecreasing function $\theta\in C^\infty(\mathbb R)$ with $0\le\theta\le1$, $\theta(s)=0$ for $s\le-7/8$ and $\theta(s)=1$ for $s\ge-3/4$. Let $\psi=\theta\eta^2$, $Z=\sqrt\theta\,\eta\nabla c/\sqrt{\lambda(c+1)}$, and define
\begin{equation}\label{eq:Ytau}
\begin{aligned}
Y_\tau
&=\sup_{-1<s\le\tau}\int_{B_1}\left(N\log N-N+1+\frac{|\nabla c|^2}{2\lambda(c+1)}\right)(s)\psi(s)\\
&\quad+\int_{-1}^{\tau}\int_{B_1}\psi\left(\frac{|\nabla N|^2}{N}+\frac{c+1}{\lambda}|\nabla^2\log(c+1)|^2+\frac{N|\nabla c|^2}{c+1}\right).
\end{aligned}
\end{equation}

Next, we establish the energy bounds for $Z$ and its spatial gradient.
By $\lambda\ge1$, $c+1\ge1$, and \eqref{eq:outer-signal}, we have
\begin{equation}\label{eq:W-outer-budget}
\int_{-\frac{15}{16}}^{0}\int_{B_{\frac{15}{16}}}\frac{|\nabla c|^2}{\lambda(c+1)}
\le\int_{-\frac{15}{16}}^{0}\int_{B_{\frac{15}{16}}}|\nabla c|^2
\le CA^2.
\end{equation}
Since $N\log N-N+1\ge0$, \eqref{eq:Ytau} gives
\begin{equation}\label{eq:Z1}
\sup_{-1<s\le\tau}\int_{B_1}|Z(s)|^2
=\sup_{-1<s\le\tau}\int_{B_1}\psi(s)\frac{|\nabla c(s)|^2}{\lambda(c(s)+1)}
\le2Y_\tau.
\end{equation}
Moreover, integrating \eqref{eq:W-gradient} with the time weight $\theta$ and using \eqref{eq:W-outer-budget}, we obtain
\begin{align}\label{eq:Z2}
\int_{-1}^{\tau}\int_{B_1}|\nabla Z|^2
&=\int_{-1}^{\tau}\theta\int_{B_1}\left|\nabla\left(\frac{\eta\nabla c}{\sqrt{\lambda(c+1)}}\right)\right|^2\nonumber\\
&\le C\int_{-1}^{\tau}\int_{B_1}\psi\frac{c+1}{\lambda}|\nabla^2\log(c+1)|^2
+C\int_{-1}^{\tau}\int_{B_1}\theta\frac{|\nabla c|^2}{\lambda(c+1)}|\nabla\eta|^2\nonumber\\
&\le CY_\tau+CA^2.
\end{align}
Applying \eqref{eq:parabolic-Sobolev} to $Z$ and using \eqref{eq:Z1}--\eqref{eq:Z2}, we obtain
\begin{equation}\label{eq:Z-Sob}
\begin{aligned}
\|Z\|_{L^{\frac{10}{3}}((-1,\tau)\times B_1)}^2
&\le C\left(\sup_{-1<s\le\tau}\int_{B_1}|Z(s)|^2\right)^{\frac25}
\left(\int_{-1}^{\tau}\int_{B_1}|\nabla Z|^2\right)^{\frac35}\le C\left(Y_\tau+A^2\right).
\end{aligned}
\end{equation}
We also record the growth estimates for the density entropy.
Since $p_0>5/4$, we have
\[
z\log z-z+1+(z\log z-z+1)^{\frac54}+z(\log z)^2
\le C_{p_0}(1+z^{p_0}),\qquad z>0.
\]
Consequently, by \eqref{eq:normalized-hyp},
\begin{equation}\label{eq:h-growth}
\iint_{Q_1}\left(N\log N-N+1+N(\log N)^2\right)
+\|N\log N-N+1\|_{L^{\frac54}(Q_1)}
\le C_{p_0}(1+K_0).
\end{equation}
Now we test the normalized entropy inequality against $\psi$ and integrate by parts.
Multiplying \eqref{eq:normalized-inequality} by $\psi$ and integrating by parts, we obtain
\begin{equation}\label{eq:H}
\begin{aligned}
&\int_{B_1}\left(N\log N-N+1+\frac{|\nabla c|^2}{2\lambda(c+1)}\right)(\tau)\psi(\tau)\\
&\quad+\int_{-1}^{\tau}\int_{B_1}\psi\left(\frac12\frac{|\nabla N|^2}{N}+\frac{c+1}{\lambda}|\nabla^2\log(c+1)|^2+\frac{N|\nabla c|^2}{2(c+1)}\right)\\
&\le\int_{-1}^{\tau}\int_{B_1}\left(N\log N-N+1+\frac{|\nabla c|^2}{2\lambda(c+1)}\right)(\psi_s+\Delta\psi+b\cdot\nabla\psi)\\
&\quad+\int_{-1}^{\tau}\int_{B_1}N\log N\nabla c\cdot\nabla\psi
+\int_{-1}^{\tau}\int_{B_1}\psi|\nabla b|\frac{|\nabla c|^2}{\lambda(c+1)}.
\end{aligned}
\end{equation}
We estimate the cutoff and chemotactic flux terms on the right-hand side.
Since $\psi_s$ and $\Delta\psi$ are bounded, \eqref{eq:W-outer-budget} and \eqref{eq:h-growth} yield
\begin{equation}\label{eq:H1}
\int_{-1}^{\tau}\int_{B_1}\left(N\log N-N+1+\frac{|\nabla c|^2}{2\lambda(c+1)}\right)(|\psi_s|+|\Delta\psi|)
\le C(A,p_0,K_0).
\end{equation}
Using $\nabla\psi=2\theta\eta\nabla\eta$ and Young's inequality, we have
\begin{equation}\label{eq:normalized-chem-flux}
|N\log N\nabla c\cdot\nabla\psi|
\le\frac14\psi\frac{N|\nabla c|^2}{c+1}
+4\theta(c+1)N(\log N)^2|\nabla\eta|^2.
\end{equation}
By $c+1\le A+1$, \eqref{eq:h-growth}, and \eqref{eq:normalized-hyp}, the second term satisfies
\begin{equation}\label{eq:H3}
\int_{-1}^{\tau}\int_{B_1}4\theta(c+1)N(\log N)^2|\nabla\eta|^2
\le C_A\iint_{Q_1}N(\log N)^2
\le C(A,p_0,K_0).
\end{equation}
For the density transport term, H\"older's inequality and \eqref{eq:h-growth} give
\begin{equation}\label{eq:density-transport}
\int_{-1}^{\tau}\int_{B_1}(N\log N-N+1)|b||\nabla\psi|
\le C\|b\|_{L^5(Q_1)}\|N\log N-N+1\|_{L^{\frac54}(Q_1)}
\le C(A,p_0,K_0).
\end{equation}
We next control the Fisher transport and strain terms by the Sobolev norm of $Z$.
Using $|Z|^2=\psi|\nabla c|^2/[\lambda(c+1)]$, Young's inequality and \eqref{eq:Z-Sob}, we obtain
\begin{equation}\label{eq:Fisher-transport}
\begin{aligned}
\int_{-1}^{\tau}\int_{B_1}\frac{|\nabla c|^2}{2\lambda(c+1)}|b||\nabla\psi|&=\int_{-1}^{\tau}\int_{B_1}\theta\eta|b|\frac{|\nabla c|^2}{\lambda(c+1)}|\nabla\eta|\\
&\le\frac12\int_{-1}^{\tau}\int_{B_1}|b|^2|Z|^2
+\frac12\int_{-1}^{\tau}\int_{B_1}\theta\frac{|\nabla c|^2}{\lambda(c+1)}|\nabla\eta|^2\\
&\le C\|b\|_{L^5(Q_1)}^2\|Z\|_{L^{\frac{10}{3}}((-1,\tau)\times B_1)}^2+CA^2\\
&\le C\|b\|_{L^5(Q_1)}^2(Y_\tau+A^2)+CA^2.
\end{aligned}
\end{equation}
Similarly, the strain term satisfies
\begin{equation}\label{eq:normalized-strain}
\begin{aligned}
\int_{-1}^{\tau}\int_{B_1}\psi|\nabla b|\frac{|\nabla c|^2}{\lambda(c+1)}
&=\int_{-1}^{\tau}\int_{B_1}|\nabla b||Z|^2\\
&\le\|\nabla b\|_{L^{\frac52}(Q_1)}
\|Z\|_{L^{\frac{10}{3}}((-1,\tau)\times B_1)}^2\\
&\le C\|\nabla b\|_{L^{\frac52}(Q_1)}(Y_\tau+A^2).
\end{aligned}
\end{equation}
For $-1<\sigma\le\tau<0$, combining
\eqref{eq:H}--\eqref{eq:H3} with
\eqref{eq:density-transport},
\eqref{eq:normalized-chem-flux},
\eqref{eq:Fisher-transport}, and
\eqref{eq:normalized-strain}, and using $Y_\sigma\le Y_\tau$,
we obtain
\begin{equation}\label{eq:normalized-coefficient-budget}
\begin{aligned}
&\int_{B_1}
 \left(N\log N-N+1+
 \frac{|\nabla c|^2}{2\lambda(c+1)}\right)
 (\sigma,x)\psi(\sigma,x)\,dx\\
&\quad+\int_{-1}^{\sigma}\int_{B_1}\psi
 \left(
 \frac12\frac{|\nabla N|^2}{N}
 +\frac{c+1}{\lambda}|\nabla^2\log(c+1)|^2
 +\frac14\frac{N|\nabla c|^2}{c+1}
 \right)\,dx\,ds\\
&\le C+
 C\left(\|b\|_{L^5(Q_1)}^2+
 \|\nabla b\|_{L^{5/2}(Q_1)}\right)Y_\tau,
\end{aligned}
\end{equation}
where $C\ge1$ depends only on $A,p_0,K_0$; here we used
$\eta_0\le1$ and \eqref{eq:normalized-hyp}.
All terms on the left-hand side are nonnegative.
Taking the supremum in $\sigma$ bounds the time-slice term
in \eqref{eq:Ytau}; setting $\sigma=\tau$ bounds its dissipation
term by four times the same right-hand side. Hence
\begin{equation}\label{eq:Y-absorb}
Y_\tau
\le C_0+
C_1\left(\|b\|_{L^5(Q_1)}^2+
\|\nabla b\|_{L^{5/2}(Q_1)}\right)Y_\tau,
\end{equation}
where $C_0=C_1=5C$ are independent of $\tau$ and $\lambda$.
 Choose
\begin{equation}\label{eq:normalized-smallness-choice}
\eta_0=\min\{1,(2C_1)^{-1}\}.
\end{equation}
Then \eqref{eq:normalized-hyp} and \eqref{eq:Y-absorb} yield
\begin{equation}\label{eq:normalized-Y-bound}
\sup_{-1<\tau<0}Y_\tau\le2C_0.
\end{equation}

\smallskip
\noindent\textit{\bf Step 5: Density estimates.}
For any $-\frac12<\tau<0$, since $\psi=1$ on $(-\frac34,0)\times B_{\frac34}$, by \eqref{eq:Ytau}, \eqref{eq:normalized-Y-bound}, together with
\[
N\le N\log N-N+1+e^2,\qquad
|\nabla\sqrt N|^2=\frac{|\nabla N|^2}{4N},
\]
we have
\begin{align}\label{eq:est of N}
\sup_{-\frac34<s\le\tau}\int_{B_{\frac34}}N(s)
+\int_{-\frac34}^{\tau}\int_{B_{\frac34}}
\left(|\nabla\sqrt N|^2+\frac{N|\nabla c|^2}{c+1}\right)
\le C(A,p_0,K_0).
\end{align}
Using $c+1\le A+1$, we also have
\begin{equation}\label{eq:normalized-unweighted-coupling}
\int_{-\frac12}^{\tau}\int_{B_{\frac12}}N|\nabla c|^2
\le(A+1)\int_{-\frac12}^{\tau}\int_{B_{\frac12}}
\frac{N|\nabla c|^2}{c+1}
\le C(A,p_0,K_0).
\end{equation}
Choose $0\le\xi\le1$ in $C_c^\infty(B_{\frac34})$ with $\xi=1$ on $B_{\frac12}$. Applying \eqref{eq:parabolic-Sobolev} to $\xi\sqrt N$ gives
\begin{equation}\label{eq:normalized-density-Sobolev}
\begin{aligned}
\int_{-\frac12}^{\tau}\int_{B_{\frac12}}N^{\frac53}
&\le\int_{-\frac12}^{\tau}\int_{B_1}|\xi\sqrt N|^{\frac{10}{3}}\\
&\le C\left(\sup_{-\frac12<s\le\tau}\int_{B_1}\xi^2N(s)\right)^{\frac23}
\int_{-\frac12}^{\tau}\int_{B_1}|\nabla(\xi\sqrt N)|^2\\
&\le C(A,p_0,K_0),
\end{aligned}
\end{equation}
where we used $|\nabla(\xi\sqrt N)|^2\le2\xi^2|\nabla\sqrt N|^2+2N|\nabla\xi|^2$ and \eqref{eq:est of N}. Combining \eqref{eq:est of N}, \eqref{eq:normalized-unweighted-coupling} and \eqref{eq:normalized-density-Sobolev}, we obtain
\begin{equation}\label{eq:normalized-density-gain-tau}
\begin{aligned}
\sup_{-\frac12<s\le\tau}\int_{B_{\frac12}}N(s)
+\int_{-\frac12}^{\tau}\int_{B_{\frac12}}
\left(|\nabla\sqrt N|^2+N|\nabla c|^2+N^{\frac53}\right)
\le C(A,p_0,K_0).
\end{aligned}
\end{equation}


\smallskip
\noindent\textit{\bf Step 6: Signal estimates.}
For any $-\frac12<\tau<0$, multiplying \eqref{eq:localized-quartic-estimate} by $\theta/\lambda$, integrating in time up to $\tau$, and using \eqref{eq:W-outer-budget} and \eqref{eq:normalized-Y-bound}, we obtain
\begin{equation}\label{eq:normalized-quartic-bound}
\begin{aligned}
\frac{1}{\lambda}\int_{-\frac34}^{\tau}\int_{B_{\frac34}}
\frac{|\nabla c|^4}{(c+1)^3}
&\le\frac{C}{\lambda}\int_{-1}^{\tau}\int_{B_1}
\psi(c+1)|\nabla^2\log(c+1)|^2\\
&\quad+C\int_{-1}^{\tau}\int_{B_1}
\theta\frac{|\nabla c|^2}{\lambda(c+1)}|\nabla\eta|^2\\
&\le C(A,p_0,K_0).
\end{aligned}
\end{equation}
Moreover,
\[
\nabla^2c=(c+1)\nabla^2\log(c+1)+\frac{\nabla c\otimes\nabla c}{c+1},
\]
and hence
\[
\frac{|\nabla^2c|^2}{\lambda(c+1)}
\le\frac{2(c+1)}{\lambda}|\nabla^2\log(c+1)|^2
+\frac{2}{\lambda}\frac{|\nabla c|^4}{(c+1)^3}.
\]
Combining this inequality with \eqref{eq:normalized-Y-bound}, \eqref{eq:normalized-quartic-bound} and $c+1\le A+1$, we obtain
\[
\frac{1}{\lambda}\int_{-\frac12}^{\tau}\int_{B_{\frac12}}
\left(|\nabla^2c|^2+|\nabla c|^4\right)
\le C(A,p_0,K_0).
\]
Also, by \eqref{eq:Ytau} and \eqref{eq:normalized-Y-bound},
\[
\begin{aligned}
\frac{1}{\lambda}\sup_{-\frac12<s\le\tau}
\int_{B_{\frac12}}|\nabla c(s)|^2
&\le(A+1)\sup_{-\frac12<s\le\tau}
\int_{B_{\frac12}}\frac{|\nabla c(s)|^2}{\lambda(c(s)+1)}\\
&\le C(A,p_0,K_0).
\end{aligned}
\]
Together with \eqref{eq:outer-signal}, these estimates yield
\begin{equation}\label{eq:normalized-signal-gain-tau}
\begin{aligned}
&\frac{1}{\lambda}\sup_{-\frac12<s\le\tau}
\int_{B_{\frac12}}|\nabla c(s)|^2\\
&\quad+\int_{-\frac12}^{\tau}\int_{B_{\frac12}}
\left(|\nabla c|^2+\lambda Nc^2+
\frac{|\nabla^2c|^2+|\nabla c|^4}{\lambda}\right)
\le C(A,p_0,K_0).
\end{aligned}
\end{equation}

Finally, letting $\tau\uparrow0$ in
\eqref{eq:normalized-density-gain-tau} and
\eqref{eq:normalized-signal-gain-tau}, the monotone convergence theorem applies to all nonnegative space-time integrals. Moreover, the time-slice suprema increase to the corresponding suprema over $-\frac12<s<0$. Hence we obtain \eqref{eq:normalized-gain} and \eqref{eq:normalized-signal-gain}.
\end{proof}


\section{Exclusion of density concentration}
\label{sec:concentration}

Next we prove decay of the density scale quantity at every
candidate singular point $(T,x_0)$. We treat separately the
cases $c_*(T,x_0)>0$ and $c_*(T,x_0)=0$, using the terminal
signal properties established in the preceding section.

\subsection{Positive terminal signal: normalization and recurrence}\label{sec:positive}

Suppose, for contradiction, that $T=T_{\max}<\infty$.
We may assume $n_0\not\equiv0$, so that $n>0$ for $0<t<T$
by the strong maximum principle.
Let
\begin{equation*}
p=\frac{19}{12},\qquad p'=\frac{19}{7},\qquad
\gamma=5-2p=\frac{11}{6}.
\end{equation*}
Fix $x_0\in\mathbb R^3$, set $z_0=(T,x_0)$, and define, for $0<r<\sqrt T$,
\begin{equation}\label{eq:terminal-density-functional}
\Acal(r;T,x_0)
:=r^{-\gamma}\int_{T-r^2}^{T}\int_{B_r(x_0)}n^p.
\end{equation}
By \eqref{eq:n-integrable}, $\Acal(r;T,x_0)<\infty$ for each fixed
$r$.

\begin{Prop}\label{prop:positive}
If \(c_*(T,x_0)>0\), then
\begin{equation}\label{eq:positive-decay}
\lim_{r\downarrow0}\Acal(r;T,x_0)=0 .
\end{equation}
\end{Prop}

\begin{proof}
\noindent\textit{\bf Step 1: Normalization and uniform integrability.}
Let $(n_r,c_r,u_r)$ be given by \eqref{eq:scaling} with
$R=r$ and $t_0=T$, and let
\begin{equation}\label{eq:lambda-N}
\lambda_r=\max\{1,\Acal(r;T,x_0)^{\frac{1}{p}}\},
\qquad
N_r=\frac{n_r}{\lambda_r}.
\end{equation}
Then
\begin{equation}\label{eq:rescaled-normalized-system}
\begin{cases}
(N_r)_s+u_r\cdot\nabla N_r
=\Delta N_r-\nabla\cdot(N_r\nabla c_r),\\
(c_r)_s+u_r\cdot\nabla c_r
=\Delta c_r-\lambda_rN_rc_r,
\end{cases}
\qquad \nabla\cdot u_r=0.
\end{equation}
By a change of variables, we have
\begin{equation}\label{eq:Nr-bound}
\iint_{Q_1}N_r^p
=\lambda_r^{-p}r^{2p-5}
\iint_{Q_r(T,x_0)}n^p
=\frac{\Acal(r;T,x_0)}{\max\{1,\Acal(r;T,x_0)\}}
\le1.
\end{equation}
Moreover,
\begin{equation}\label{eq:critical-drift-scaling}
\begin{aligned}
\|u_r\|_{L^5(Q_1)}
&=\|u\|_{L^5(Q_r(T,x_0))},\\
\|\nabla_yu_r\|_{L^{\frac{5}{2}}(Q_1)}
&=\|\nabla_xu\|_{L^{\frac{5}{2}}(Q_r(T,x_0))}.
\end{aligned}
\end{equation}
Thus \eqref{eq:uniform-fluid-tail} gives
\begin{align} \label{eq:u_r}
    \|u_r\|_{L^5(Q_1)}^2
+\|\nabla u_r\|_{L^{\frac{5}{2}}(Q_1)}
\longrightarrow0
\qquad\text{as }r\downarrow0.
\end{align}
Applying Proposition~\ref{prop:uniform-entropy} with
$p_0=p$, $K_0=1$, and $\lambda=\lambda_r$, we obtain
\begin{equation}\label{eq:Nr-gain}
\sup_{0<r<r_0}\iint_{Q_{\frac{1}{2}}}N_r^{\frac{5}{3}}
\le C(A,p)
\end{equation}
for some $r_0>0$.

\smallskip
\noindent\textit{\bf Step 2: Convergence of the normalized density.}
Let $r_k\downarrow0$ be arbitrary. By \eqref{eq:Nr-bound},
after passing to a subsequence, we have
\begin{equation}\label{eq:normalized-weak-limit}
N_{r_k}\rightharpoonup N
\quad\text{in }L^p(Q_1),
\qquad N\ge0.
\end{equation}
By Lemma~\ref{lem:local}, $(n,c,u)\in C([0,\tau];X)$ for every $\tau<T$. Hence, by \eqref{eq:Xnorm}, $n,u,\nabla c$ are bounded on $[t_b,\tau]\times\mathbb R^3$ for every $0<t_b<\tau<T$. The classical regularity on positive-time intervals follows from Step 4 of Appendix~\ref{app:local}.
Thus by Proposition~\ref{prop:terminal} with $a=p'$ in \eqref{eq:constant-tangent}, we obtain
\begin{align} \label{eq:c_rk}
    c_{r_k}\longrightarrow c_*(T,x_0)
\quad\text{strongly in }L^{p'}(Q_1).
\end{align}
For $0\le\varphi\in C_c^\infty(Q_1)$, testing
$\eqref{eq:rescaled-normalized-system}_2$ with $\varphi$
and using $\nabla\cdot u_{r_k}=0$, we have
\begin{equation}\label{eq:signal-test}
\begin{aligned}
0\le\iint_{Q_1}\lambda_{r_k}N_{r_k}c_{r_k}\varphi
&=\iint_{Q_1}
\bigl(c_{r_k}-c_*(T,x_0)\bigr)
(\varphi_s+\Delta\varphi)+\iint_{Q_1}c_{r_k}u_{r_k}\cdot\nabla\varphi\\
&\le
\|c_{r_k}-c_*(T,x_0)\|_{L^{p'}(Q_1)}
\|\varphi_s+\Delta\varphi\|_{L^p(Q_1)}\\
&\quad+A\|u_{r_k}\|_{L^5(Q_1)}
\|\nabla\varphi\|_{L^{\frac{5}{4}}(Q_1)}
\longrightarrow0,
\end{aligned}
\end{equation}
where we used $\iint_{Q_1}(\varphi_s+\Delta\varphi)=0$, \eqref{eq:u_r} and \eqref{eq:c_rk}.
By $\lambda_{r_k}\ge1$ and \eqref{eq:signal-test}, we have
\begin{align} \label{eq:ad}
    0\le\iint_{Q_1}N_{r_k}c_{r_k}\varphi
\le\iint_{Q_1}\lambda_{r_k}N_{r_k}c_{r_k}\varphi
\longrightarrow0.
\end{align}
By \eqref{eq:normalized-weak-limit} and \eqref{eq:c_rk}, we have
\[
\begin{aligned}
&\left|\iint_{Q_1}N_{r_k}c_{r_k}\varphi
-c_*(T,x_0)\iint_{Q_1}N\varphi\right|\\
&\quad\le
\|N_{r_k}\|_{L^p(Q_1)}
\|c_{r_k}-c_*(T,x_0)\|_{L^{p'}(Q_1)}
\|\varphi\|_{L^\infty(Q_1)}
+c_*(T,x_0)\left|\iint_{Q_1}(N_{r_k}-N)\varphi\right|
\longrightarrow0,
\end{aligned}
\]
which implies
\[
c_*(T,x_0)\iint_{Q_1}N\varphi=0
\qquad\text{for every }0\le\varphi\in C_c^\infty(Q_1)
\]
by \eqref{eq:ad}.
Since $c_*(T,x_0)>0$, we conclude that $N=0$ almost everywhere
in $Q_1$.

Moreover, testing \eqref{eq:normalized-weak-limit} with $1\in L^{p'}(Q_1)$
and using $N_{r_k}\ge0$, we get
\begin{equation}\label{eq:Nr-L1}
\|N_{r_k}\|_{L^1(Q_1)}
=\iint_{Q_1}N_{r_k}
\longrightarrow0,
\end{equation}
which combined with \eqref{eq:Nr-gain} gives
\begin{equation}\label{eq:Nr-Lp}
\begin{aligned}
\|N_{r_k}\|_{L^p(Q_{\frac{1}{2}})}
\le
\|N_{r_k}\|_{L^1(Q_{\frac{1}{2}})}^{\frac{3}{38}}
\|N_{r_k}\|_{L^{\frac{5}{3}}(Q_{\frac{1}{2}})}^{\frac{35}{38}}\longrightarrow0.
\end{aligned}
\end{equation}
As every sequence $r_k\downarrow0$ has a subsequence
satisfying \eqref{eq:Nr-Lp}, we conclude that
\[
\lim_{r\downarrow0}\iint_{Q_{\frac{1}{2}}}N_r^p=0.
\]

\smallskip
\noindent\textit{\bf Step 3: Decay of $\Acal(r;T,x_0)$.}
By \eqref{eq:lambda-N} and a change of variables, we have
\begin{equation}\label{eq:exact-recurrence}
\begin{aligned}
\iint_{Q_{\frac{1}{2}}}N_r^p
=\frac{r^{-\gamma}}{\max\{1,\Acal(r;T,x_0)\}}
\iint_{Q_{\frac{r}{2}}(T,x_0)}n^p=2^{-\gamma}
\frac{\Acal(\frac{r}{2};T,x_0)}{\max\{1,\Acal(r;T,x_0)\}},
\end{aligned}
\end{equation}
which implies
\begin{equation}\label{eq:recurrence-coefficient}
\frac{\Acal(\frac{r}{2};T,x_0)}{\max\{1,\Acal(r;T,x_0)\}}
\longrightarrow0
\qquad\text{as }r\downarrow0.
\end{equation}
Decreasing $r_0$ if necessary, we have
\begin{equation}\label{eq:dyadic-recurrence}
\Acal(\frac{r}{2};T,x_0)\le\frac{1}{2}\max\{1,\Acal(r;T,x_0)\},
\qquad 0<r\le r_0.
\end{equation}
Iterating \eqref{eq:dyadic-recurrence} yields
\[
\Acal(2^{-j}r_0;T,x_0)
\le\max\{1,2^{-j}\Acal(r_0;T,x_0)\}
\le\max\{1,\Acal(r_0;T,x_0)\},
\qquad j\ge0.
\]
Since $\Acal(r_0)<\infty$, \eqref{eq:recurrence-coefficient}
now implies
\[
\begin{aligned}
\Acal(2^{-j-1}r_0;T,x_0)
&=\frac{\Acal(2^{-j-1}r_0;T,x_0)}{\max\{1,\Acal(2^{-j}r_0;T,x_0)\}}
\max\{1,\Acal(2^{-j}r_0;T,x_0)\}\\
&\le\max\{1,\Acal(r_0;T,x_0)\}
\frac{\Acal(2^{-j-1}r_0;T,x_0)}{\max\{1,\Acal(2^{-j}r_0;T,x_0)\}}
\longrightarrow0.
\end{aligned}
\]
For each $0<r\le r_0$, choose $j\ge0$ such that
$2^{-j-1}r_0<r\le2^{-j}r_0$. Since
$Q_r(z_0)\subset Q_{2^{-j}r_0}(z_0)$, we have
\[
\begin{aligned}
\Acal(r;T,x_0)
&=r^{-\gamma}\iint_{Q_r(z_0)}n^p
\le r^{-\gamma}\iint_{Q_{2^{-j}r_0}(z_0)}n^p\\
&=\left(\frac{2^{-j}r_0}{r}\right)^\gamma
\Acal(2^{-j}r_0;T,x_0)
\le 2^\gamma\Acal(2^{-j}r_0;T,x_0).
\end{aligned}
\]
As $r\downarrow0$, we have $j\to\infty$, and hence
$\Acal(r;T,x_0)\to0$.
This proves \eqref{eq:positive-decay}.
\end{proof}

\subsection{Zero terminal signal: quantitative density decay}\label{sec:small}

We next consider the case $c_*(T,x_0)=0$. Since $c_*$ is {upper semicontinuous in space and time} and $c_*(T,x_0)=0$,
for every $\varepsilon>0$ there exists $r>0$ such that
\begin{align}\label{c leq 1}
    0\le c(t,x)=c_*(t,x)<\varepsilon
\qquad\text{for all }(t,x)\in Q_r(T,x_0).
\end{align}
We derive a local estimate using the weighted
density energy $\int N^p\phi(c)$, following the small-signal method
of Tao \cite{Tao2011}; see also
Lankeit--Winkler \cite[Section~2]{LankeitWinkler2023}.
The estimate allows a divergence-free drift in
$L^\infty_sL^3_y$, with no smallness assumption on its norm.
Fixing $p=19/12$, we use the local energy estimate for $N^{p/2}$
to obtain $N\in L^{5p/3}$ on a smaller cylinder, where
\begin{equation}\label{eq:small-signal-exponent}
\frac{5p}{3}=\frac{95}{36}>\frac{5}{2}.
\end{equation}

\begin{Prop}\label{prop:small}
Fix $p=\frac{19}{12}$. Suppose that $N>0$, $0\le c\le1$, and $b$
are classical on $Q_1=(-1,0)\times B_1$ and satisfy
\begin{equation}\label{eq:small-system}
\begin{cases}
N_t+b\cdot\nabla N=\Delta N-\nabla\cdot(N\nabla c),\\
c_t+b\cdot\nabla c=\Delta c-Nc,\qquad \nabla\cdot b=0.
\end{cases}
\end{equation}
Assume that
\begin{equation}\label{eq:small-signal-hyp}
N\in L^p(Q_1),
\qquad
b\in L^\infty(-1,0;L^3(B_1)).
\end{equation}
Then
\begin{align}
&\sup_{-\frac{9}{16}<t<0}\int_{B_{\frac{3}{4}}}N^p(t)
+\iint_{Q_{\frac{3}{4}}}
\left(|\nabla N^{\frac{p}{2}}|^2+N^p|\nabla c|^2\right)
\le
C\left(1+\|b\|_{L^\infty(-1,0;L^3(B_1))}^2\right)
\iint_{Q_1}N^p,
\label{eq:small-energy}\\
&\iint_{Q_{\frac{2}{3}}}N^{\frac{5p}{3}}
\le C\left[
\left(1+\|b\|_{L^\infty(-1,0;L^3(B_1))}^2\right)
\iint_{Q_1}N^p
\right]^{\frac{5}{3}}.
\label{eq:small-gain}
\end{align}
Consequently, for every $0<\rho\le2/3$,
\begin{equation}\label{eq:small-Morrey}
\rho^{2p-5}\iint_{Q_\rho}N^p
\le
C\rho^{\frac{1}{6}}
\left(1+\|b\|_{L^\infty(-1,0;L^3(B_1))}^2\right)
\iint_{Q_1}N^p.
\end{equation}
\end{Prop}

\begin{proof}
\noindent\textit{\bf Step 1: The weighted identity.}
Define
\begin{equation}\label{eq:cos-weight}
\varphi(c):=(\cos c)^{-(p-1)}.
\end{equation}
For $0\le c\le1$, we have $1\le\varphi\le C$,
$\varphi'\ge0$, and
\begin{equation}\label{eq:cos-coefficient}
\begin{aligned}
\frac{\varphi'}{\varphi}
&=(p-1)\tan c,\\
\frac{\varphi''}{\varphi}
-p\frac{\varphi'}{\varphi}
-p(p-1)\left(\frac{1}{2}-\tan c\right)^2
&=(p-1)\left(1-\frac{p}{4}\right)
=\frac{203}{576}.
\end{aligned}
\end{equation}
Let $0\leq \psi\in C^\infty((-1,0);C_c^\infty(B_1))$. Multiplying the two equations in \eqref{eq:small-system} by
$pN^{p-1}\varphi(c)\psi$ and $N^p\varphi'(c)\psi$,
respectively, and adding, we obtain
\begin{align} \label{eq:Np0}
\frac{\dd}{\dd t}\int\psi N^p\varphi
&=\int\psi_tN^p\varphi
+p\int\psi N^{p-1}\varphi\Delta N
+\int\psi N^p\varphi'\Delta c\\
&\quad-p\int\psi N^{p-1}\varphi\,\nabla\cdot(N\nabla c)
-\int\psi N^{p+1}c\varphi'
-\int\psi b\cdot\nabla(N^p\varphi). \nonumber
\end{align}
Integrating the first diffusion term by parts, we obtain
\[
\begin{aligned}
p\int\psi N^{p-1}\varphi\Delta N
={}&-p(p-1)\int\psi\varphi N^{p-2}|\nabla N|^2-p\int\psi\varphi'N^{p-1}\nabla c\cdot\nabla N
-p\int\varphi N^{p-1}\nabla N\cdot\nabla\psi.
\end{aligned}
\]
Similarly,
\[
\begin{aligned}
\int\psi N^p\varphi'\Delta c
={}&-p\int\psi\varphi'N^{p-1}\nabla N\cdot\nabla c-\int\psi N^p\varphi''|\nabla c|^2
-\int N^p\varphi'\nabla c\cdot\nabla\psi.
\end{aligned}
\]
Adding these identities, the mixed terms combine to give
$-2p\int\psi\varphi'N^{p-1}\nabla N\cdot\nabla c$, while the cutoff terms satisfy
\[
\begin{aligned}
&-p\int\varphi N^{p-1}\nabla N\cdot\nabla\psi
-\int N^p\varphi'\nabla c\cdot\nabla\psi=-\int\nabla(N^p\varphi(c))\cdot\nabla\psi
=\int N^p\varphi(c)\Delta\psi.
\end{aligned}
\]
Consequently, we obtain 
\begin{align}\label{eq:Np1}
p\int\psi N^{p-1}\varphi\Delta N
+\int\psi N^p\varphi'\Delta c
&=-p(p-1)\int\psi\varphi N^{p-2}|\nabla N|^2
-2p\int\psi\varphi'N^{p-1}\nabla N\cdot\nabla c\\
&\qquad-\int\psi N^p\varphi''|\nabla c|^2
+\int N^p\varphi\Delta\psi, \nonumber
\end{align}
Similarly, the chemotaxis term satisfies
\begin{align} \label{eq:Np2}
-p\int\psi N^{p-1}\varphi\,\nabla\cdot(N\nabla c)
&=p(p-1)\int\psi\varphi N^{p-1}\nabla N\cdot\nabla c\\
&\quad+p\int\psi N^p\varphi'|\nabla c|^2
+p\int N^p\varphi\nabla c\cdot\nabla\psi. \nonumber
\end{align}
Using $\nabla\cdot b=0$, we also have
\begin{align} \label{eq:Np3}
    -\int\psi b\cdot\nabla(N^p\varphi)
=\int N^p\varphi b\cdot\nabla\psi.
\end{align}
Combining \eqref{eq:Np0}--\eqref{eq:Np3} yields
\begin{equation}\label{eq:cos-quadratic-form}
\begin{aligned}
\frac{\dd}{\dd t}\int\psi N^p\varphi
&=-p(p-1)\int\psi\varphi N^{p-2}|\nabla N|^2\\
&\quad+p\int\psi N^{p-1}
\bigl((p-1)\varphi-2\varphi'\bigr)
\nabla N\cdot\nabla c\\
&\quad+\int\psi N^p(p\varphi'-\varphi'')|\nabla c|^2
-\int\psi N^{p+1}c\varphi'\\
&\quad+\int N^p\varphi
(\psi_t+\Delta\psi+b\cdot\nabla\psi)
+p\int N^p\varphi\nabla c\cdot\nabla\psi,
\end{aligned}
\end{equation}
which, along with \eqref{eq:cos-coefficient}, yields
\begin{equation}\label{eq:cos-identity}
\begin{aligned}
\frac{\dd}{\dd t}\int\psi N^p\varphi
&+p(p-1)\int\psi\varphi N^{p-2}
\left|\nabla N-\left(\frac{1}{2}-\tan c\right)N\nabla c\right|^2\\
&+\frac{203}{576}\int\psi\varphi N^p|\nabla c|^2
+\int\psi N^{p+1}c\varphi'\\
&=\int N^p\varphi(\psi_t+\Delta\psi+b\cdot\nabla\psi)
+p\int N^p\varphi\nabla c\cdot\nabla\psi.
\end{aligned}
\end{equation}

\smallskip
\noindent\textit{\bf Step 2: The local energy estimate.}
Since $|{1/2}-\tan c|\le C$, we have
\[
|\nabla N^{\frac{p}{2}}|^2
\le C N^{p-2}
\left|\nabla N-\left(\frac{1}{2}-\tan c\right)N\nabla c\right|^2
+C N^p|\nabla c|^2.
\]
Thus, for some absolute $\kappa>0$ and $1\leq\varphi\leq C$, we have
\begin{equation}\label{eq:cos-coercivity}
\begin{aligned}
p(p-1)\varphi N^{p-2}
\left|\nabla N-\left(\frac{1}{2}-\tan c\right)N\nabla c\right|^2
+\frac{203}{1152}\varphi N^p|\nabla c|^2\ge
\kappa\left(|\nabla N^{\frac{p}{2}}|^2+N^p|\nabla c|^2\right).
\end{aligned}
\end{equation}
Choose $0\le\zeta\le1$ in $C_c^\infty(B_1)$ with
$\zeta=1$ on $B_{3/4}$, and a smooth nondecreasing function
$0\le\theta\le1$ such that $\theta=0$ near $t=-1$ and
$\theta=1$ for $t\ge-9/16$. Take $\psi=\theta\zeta^2$, with $|\theta'|+|\nabla\zeta|+|\nabla^2\zeta|\le C$. Young's inequality gives
\begin{equation}\label{eq:cos-cutoff}
pN^p\varphi|\nabla c\cdot\nabla\psi|
\le \frac{203}{1152}\psi\varphi N^p|\nabla c|^2
+C\theta N^p|\nabla\zeta|^2.
\end{equation}
By H\"older's inequality and the Sobolev inequality for
$\zeta N^{p/2}\in H_0^1(B_1)$, we have
\begin{equation}\label{eq:cos-transport}
\begin{aligned}
\left|\int N^p\varphi b\cdot\nabla\psi\right|
&\le C\theta\|b(t)\|_{L^3(B_1)}
\|\zeta N^{\frac{p}{2}}\|_{L^6(B_1)}
\|N^{\frac{p}{2}}\nabla\zeta\|_{L^2(B_1)}\\
&\le C\theta\|b(t)\|_{L^3(B_1)}
\|\nabla(\zeta N^{\frac{p}{2}})\|_{L^2(B_1)}
\|N^{\frac{p}{2}}\nabla\zeta\|_{L^2(B_1)}\\
&\le \frac{\kappa}{2}\theta\int\zeta^2|\nabla N^{\frac{p}{2}}|^2
+C\left(1+\|b(t)\|_{L^3(B_1)}^2\right)
\theta\int N^p|\nabla\zeta|^2.
\end{aligned}
\end{equation}
Moreover,
\[
\left|\int N^p\varphi(\psi_t+\Delta\psi)\right|
\le C\int_{B_1}N^p,
\qquad
\int\psi N^{p+1}c\varphi'\ge0.
\]
Combining these estimates with \eqref{eq:cos-identity}
and integrating from $-1$ to any $\tau\in(-9/16,0)$,
we obtain
\[
\begin{aligned}
&\int_{B_{\frac{3}{4}}}N^p(\tau)
+\frac{\kappa}{2}
\int_{-\frac{9}{16}}^{\tau}\int_{B_{\frac{3}{4}}}
\left(|\nabla N^{\frac{p}{2}}|^2+N^p|\nabla c|^2\right)\\
&\qquad\le
C\left(1+\|b\|_{L^\infty(-1,0;L^3(B_1))}^2\right)
\iint_{Q_1}N^p.
\end{aligned}
\]
Taking the supremum over $\tau$ and then letting
$\tau\uparrow0$ proves \eqref{eq:small-energy}.

\smallskip
\noindent\textit{\bf Step 3: Higher integrability and scale decay.}
Choose $\xi\in C_c^\infty(B_{3/4})$ with
$0\le\xi\le1$, $\xi=1$ on $B_{2/3}$, and
$|\nabla\xi|\le C$. Applying
\eqref{eq:parabolic-Sobolev} to $\xi N^{p/2}$ on $(-4/9,0)$
and using \eqref{eq:small-energy}, we have
\[
\begin{aligned}
\iint_{Q_{\frac{2}{3}}}N^{\frac{5p}{3}}
&\le
C\left(\sup_{-\frac{4}{9}<t<0}\int_{B_{\frac{3}{4}}}\xi^2N^p(t)\right)^{\frac{2}{3}}
\int_{-\frac{4}{9}}^{0}\int_{B_{\frac{3}{4}}}
|\nabla(\xi N^{\frac{p}{2}})|^2\\
&\le C\left[
\left(1+\|b\|_{L^\infty(-1,0;L^3(B_1))}^2\right)
\iint_{Q_1}N^p
\right]^{\frac{5}{3}}.
\end{aligned}
\]
This proves \eqref{eq:small-gain}. For $0<\rho\le2/3$,
H\"older's inequality gives
\[
\begin{aligned}
\rho^{2p-5}\iint_{Q_\rho}N^p
\le \rho^{2p-5}|Q_\rho|^{\frac{2}{5}}
\left(\iint_{Q_\rho}N^{\frac{5p}{3}}\right)^{\frac{3}{5}}\le C\rho^{2p-3}
\left(1+\|b\|_{L^\infty(-1,0;L^3(B_1))}^2\right)
\iint_{Q_1}N^p.
\end{aligned}
\]
By $2p-3=1/6$, this proves \eqref{eq:small-Morrey}
and completes the proof.
\end{proof}



\begin{Cor}\label{cor:zero}
If \(c_*(T,x_0)=0\), then
\[
\lim_{r\downarrow0}\Acal_p(r;T,x_0)=0.
\]
\end{Cor}

\begin{proof}
Since $c_*(T,x_0)=0$ and $c_*$ is upper semicontinuous,
there exists $0<R<\sqrt T$ such that
$0\le c\le1$ in $Q_R(T,x_0)$ by \eqref{c leq 1}. Let
\[
\begin{aligned}
N(s,y)&=R^2n(T+R^2s,x_0+Ry),\\
C(s,y)&=c(T+R^2s,x_0+Ry),\\
b(s,y)&=Ru(T+R^2s,x_0+Ry).
\end{aligned}
\]
Then $(N,C,b)$ satisfies \eqref{eq:small-system}, with
$0\le C\le1$ and
\[
\begin{aligned}
\iint_{Q_1}N^p
&=R^{2p-5}\iint_{Q_R(T,x_0)}n^p
=\Acal_p(R;T,x_0)<\infty,\\
\|b\|_{L^\infty(-1,0;L^3(B_1))}
&=\|u\|_{L^\infty(T-R^2,T;L^3(B_R(x_0)))}\le\|u\|_{L^\infty(0,T;L^3(\R^3))}<\infty.
\end{aligned}
\]
For $0<r\le2R/3$, applying \eqref{eq:small-Morrey}
with $\rho=r/R$ gives
\begin{equation}\label{eq:zero-rate}
\begin{aligned}
\Acal_p(r;T,x_0)
=\left(\frac{r}{R}\right)^{2p-5}
\iint_{Q_{\frac{r}{R}}}N^p\le C\left(1+\|u\|_{L^\infty(0,T;L^3(\R^3))}^2\right)
\Acal_p(R;T,x_0)\left(\frac{r}{R}\right)^{\frac{1}{6}}.
\end{aligned}
\end{equation}
Keeping $R$ fixed and letting $r\downarrow0$ yields
$\Acal_p(r;T,x_0)\to0$.
\end{proof}

\section{Proofs of the main results}

\subsection{Uniform continuation on the whole space}\label{sec:global}

We now prove Theorem~\ref{thm:main}.

\begin{proof}[Proof of Theorem~\ref{thm:main}]
By \eqref{eq:data-smooth}--\eqref{eq:data-fluid}, we have
$(n_0,c_0,u_0)\in X$. Thus, Lemma~\ref{lem:local} yields a unique
solution
\begin{align} \label{eq:local}
    (n,c,u)\in C([0,T_{\max});X)
\cap C_{\mathrm{loc}}^{1,2}((0,T_{\max})\times\R^3),
\qquad
\nabla\pi\in C_{\mathrm{loc}}((0,T_{\max})\times\R^3),
\end{align}
which satisfies \eqref{eq:system} pointwise for $0<t<T_{\max}$. Here $v\in C_{\mathrm{loc}}^{1,2}$ means that $v$, $\partial_t v$, $\nabla_x v$, and $D_x^2v$ are continuous.
Suppose, for contradiction, that $T=T_{\max}<\infty$. First, we may assume that $n_0\not\equiv0$.

\smallskip
\noindent\textit{\bf Step 1: Density decay at every terminal point.}
Testing $\eqref{eq:system}_1$ with $-n_-$, where
$n_-=\max\{-n,0\}$, and using $\nabla\cdot u=0$, we obtain
\[
\begin{aligned}
\frac{1}{2}\frac{\dd}{\dd t}\|n_-\|_2^2
+\|\nabla n_-\|_2^2
=\int n_-\nabla c\cdot\nabla n_-\le\frac{1}{2}\|\nabla n_-\|_2^2
+\frac{1}{2}\|\nabla c\|_\infty^2\|n_-\|_2^2.
\end{aligned}
\]
For every $S<T_{\max}$, \eqref{eq:local} gives
$\nabla c\in L^\infty((0,S)\times\R^3)$.
Since $n_-(0)=0$, Gronwall's inequality yields $n\ge0$.
As $n_0\not\equiv0$, the strong maximum principle gives
$n>0$ for $0<t<T_{\max}$. Recall that the signal equation $\eqref{eq:system}_2$ reads
\[
c_t-\Delta c+u\cdot\nabla c+nc=0.
\]
we note that the constant functions $0$ and $A=\|c_0\|_\infty$ satisfy
\[
0_t-\Delta0+u\cdot\nabla0+n\cdot0=0,
\qquad
A_t-\Delta A+u\cdot\nabla A+nA=nA\ge0.
\]
Since $0\le c_0\le A$, the comparison principle yields
$0\le c\le A$.
Mass conservation gives $\|n(t)\|_{L^1(\R^3)}=\|n_0\|_{L^1(\R^3)}$. These prove \eqref{eq:mass-max}.
Applying Lemmas~\ref{lem:energy} and~\ref{lem:fluid}
on $(0,T)$, we obtain
\begin{equation} \label{eq:integrability}
\begin{gathered}
n\in L^p((0,T)\times\R^3),\\
u\in L^\infty(0,T;L^3(\R^3))
\cap L^5((0,T)\times\R^3),\qquad
\pi\in L^2(0,T;L^3(\R^3)),
\end{gathered}
\end{equation}
where $p=19/12$.
In particular,
\[
\iint_{(0,T)\times\R^3}nc
\le AT\|n_0\|_{L^1(\R^3)}<\infty.
\]
The above estimates verify all the hypotheses of Proposition~\ref{prop:terminal}. Hence, there exists an upper semicontinuous extension $c_*:(0,T]\times\mathbb R^3\to[0,A]$ of $c$ such that $c_*(t,x)=c(t,x)$ for every $0<t<T$.

Fix $x_0\in\R^3$. If $c_*(T,x_0)>0$, use Proposition~\ref{prop:positive}.
If $c_*(T,x_0)=0$, Corollary~\ref{cor:zero} applies. In both cases, we conclude that
\begin{equation}\label{eq:all-density-decay}
\lim_{r_{x_0}\downarrow0}\Acal_p(r_{x_0};T,x_0)=0
\qquad\text{for every }x_0\in\R^3.
\end{equation}

\smallskip
\noindent\textit{\bf Step 2: A uniform estimate outside a large ball.}
Let $\varepsilon_*$ and $C$ be the constants in
Theorem~\ref{thm:epsilon}, with parameters $A$, $\|F\|_\infty$,
and $p=19/12$. By \eqref{eq:integrability}, choose
$0<\rho\le1$, $\rho^2<T/4$, such that
\begin{equation}\label{eq:choose-rho}
\|u\|_{L^5((T-\rho^2,T)\times\R^3)}
+2\|\pi\|_{L^2(T-\rho^2,T;L^3(\R^3))}
<\frac{\varepsilon_*}{2}.
\end{equation}
Since $n^p\in L^1((0,T)\times\R^3)$, we can then choose
$R_\infty>0$ such that
\begin{equation}\label{eq:choose-tail}
\int_0^T\int_{|x|>R_\infty}n^p
<\frac{\varepsilon_*}{2}\rho^\gamma.
\end{equation}
For $|x_0|>R_\infty+\rho$, we have
$B_\rho(x_0)\subset\{|x|>R_\infty\}$ and therefore
\begin{equation}\label{Acal_p}
    \begin{aligned}
    \Acal_p(\rho;T,x_0)
=\rho^{-\gamma}\iint_{Q_\rho(T,x_0)}n^p
<\frac{\varepsilon_*}{2}.
\end{aligned}
\end{equation}
Thus \eqref{eq:epsilon-hyp} holds in $Q_\rho(T,x_0)$ by \eqref{eq:choose-rho} and \eqref{Acal_p}.
Theorem~\ref{thm:epsilon} therefore yields
\[
\sup_{Q_{\frac{\rho}{2}}(T,x_0)}
\left(n+|\nabla c|^2+|u|^2\right)\le C\rho^{-2}.
\]
Hence, taking $x_0=x$, we obtain
\begin{equation}\label{eq:infinity-bound}
n(t,x)+|\nabla c(t,x)|^2+|u(t,x)|^2\le C\rho^{-2}
\end{equation}
for $|x|>R_\infty+\rho$ and $T-\rho^2/4<t<T$.

\smallskip
\noindent\textit{\bf Step 3: A uniform estimate on the whole space.}
For each $x\in\overline{B_{R_\infty+2\rho}}$, by
\eqref{eq:all-density-decay}, we can choose
$r_x>0$ sufficiently small, with $r_x\le\rho$, such that
\[
\mathcal A_p(r_x;T,x)<\frac{\varepsilon_*}{2}.
\]
Also, \eqref{eq:choose-rho} gives
\[
\begin{aligned}
\|u\|_{L^5(Q_{r_x}(T,x))}
+\|\pi-(\pi)_{B_{r_x}(x)}\|
 _{L^2(T-r_x^2,T;L^3(B_{r_x}(x)))}<\frac{\varepsilon_*}{2}.
\end{aligned}
\]
Hence \eqref{eq:epsilon-hyp} holds in $Q_{r_x}(T,x)$. By Theorem~\ref{thm:epsilon}, we have
\[
\sup_{Q_{\frac{r_x}{2}}(T,x)}
\left(n+|\nabla c|^2+|u|^2\right)\le Cr_x^{-2}.
\]
By compactness, there exist $x_1,\ldots,x_J$ and
$r_j=r_{x_j}$ such that
\[
\overline{B_{R_\infty+2\rho}}
\subset\bigcup_{j=1}^J B_{\frac{r_j}{4}}(x_j).
\]
Let
\begin{equation}\label{eq:cover-time-width}
\delta_T=\frac{1}{4}\min\{\rho^2,r_1^2,\ldots,r_J^2\}>0.
\end{equation}
For each $j$,
\[
(T-\delta_T,T)\times B_{\frac{r_j}{4}}(x_j)
\subset Q_{\frac{r_j}{2}}(T,x_j).
\]
Combining the finite cover with \eqref{eq:infinity-bound},
we obtain
\begin{equation}\label{eq:global-top-bound}
\begin{aligned}
\sup_{T-\delta_T<t<T}
\left(\|n(t)\|_\infty+\|\nabla c(t)\|_\infty^2
+\|u(t)\|_\infty^2\right)\le
C\max\{\rho^{-2},r_1^{-2},\ldots,r_J^{-2}\}<\infty.
\end{aligned}
\end{equation}

\smallskip
\noindent\textit{\bf Step 4: Continuation beyond $T$.}
By mass conservation, the maximum principle, and
\eqref{eq:energy-budget}, we have
\[
\sup_{0<t<T}
\left(\|n(t)\|_1+\|c(t)\|_\infty+\|u(t)\|_2\right)<\infty.
\]
Combining this bound with \eqref{eq:global-top-bound}
and \eqref{eq:Xnorm}, we obtain
\[
\sup_{T-\delta_T<t<T}\|(n,c,u)(t)\|_X<\infty,
\]
contradicting the blow-up criterion
\eqref{eq:continuation} in Lemma~\ref{lem:local}.
Hence $T_{\max}=\infty$.
If $n_0\equiv0$, uniqueness gives $n\equiv0$ and
$u(t)=e^{t\Delta}u_0$. Hence
$\|u(t)\|_\infty\le\|u_0\|_\infty$.
By $\eqref{eq:system}_2$ and the heat semigroup estimate,
\[
\|\nabla c(t)\|_\infty
\le\|\nabla c_0\|_\infty
+C\|u_0\|_\infty\int_0^t(t-s)^{-\frac12}
\|\nabla c(s)\|_\infty\,\dd s.
\]
{The Gronwall inequality for the weakly singular
kernel $t^{-1/2}$ bounds $\|\nabla c(t)\|_\infty$
on every finite interval.} Together with
$\|c(t)\|_\infty\le\|c_0\|_\infty$ and
$\|u(t)\|_2\le\|u_0\|_2$, this gives a uniform $X$-norm
bound on every finite interval. The continuation criterion
in Lemma~\ref{lem:local} therefore yields
$T_{\max}=\infty$ in this case as well.
Thus, in both cases, $T_{\max}=\infty$. By
Lemma~\ref{lem:local} and Definition~\ref{def:mild},
$(n,c,u)\in C([0,T];X)$ for every finite $T$.
Consequently, \eqref{eq:Xnorm} gives \eqref{eq:main-bound},
completing the proof.
\end{proof}

\subsection{Fluid-free consequences and stability under damping}\label{ks:section}

\begin{proof}[Proof of Corollary~\ref{cor:ks}]
Under the assumptions of Theorem~\ref{thm:main} on $(n_0,c_0)$, take $F=0$ and $u_0=0$, which satisfy the corresponding assumptions on the external force and initial velocity. By Theorem~\ref{thm:main}, the coupled system admits a unique global bounded mild solution. Moreover, uniqueness for the homogeneous Stokes equation gives $u=0$, and hence the scalar equations reduce to \eqref{ks:system}. The mass conservation, maximum principle and global bounds follow directly from \eqref{eq:mass-max} and \eqref{eq:main-bound}. Finally, uniqueness in $C([0,T];\BUC\times\BUC^1)$ follows by subtracting the scalar mild formulations and applying the heat-semigroup estimates on successive short time intervals.
\end{proof}

\begin{proof}[Proof of Corollary~\ref{cor:stability}]
Fix $0<T<\infty$. By Corollary~\ref{cor:ks}, the undamped solution exists on $[0,T]$. Choose $K\ge1$ such that
\begin{equation}\label{eq:stability-reference-bound}
\sup_{0\le t\le T}\left(\|n(t)\|_\infty+\|c(t)\|_{W^{1,\infty}}\right)\le K.
\end{equation}

\smallskip
\noindent\textit{\bf Step 1: Uniform local existence and continuation criterion.}
Let $T_{\max,\mu}$ denote the maximal existence time of the solution to \eqref{ks:damped}. The local construction in $\BUC\times\BUC^1$ is uniform for $\mu\in[0,1]$ on bounded sets. Indeed, the additional term in the mild density equation is
\[
-\mu\int_0^t e^{(t-s)\Delta}n_\mu(s)^2\,ds.
\]
For two functions $m_1,m_2$ satisfying $\sup_{0\le s\le h}\|m_i(s)\|_\infty\le R$, the heat semigroup contraction gives
\[
\sup_{0\le t\le h}\left\|\mu\int_0^t e^{(t-s)\Delta}(m_1(s)^2-m_2(s)^2)\,ds\right\|_\infty
\le 2Rh\sup_{0\le s\le h}\|m_1(s)-m_2(s)\|_\infty.
\]
The remaining terms are treated by the heat estimates \eqref{eq:heat-bounds}, as in the scalar part of the construction in Appendix~\ref{app:local}. Consequently, the local existence time can be chosen uniformly in $\mu\in[0,1]$ on bounded sets, and the continuation criterion gives
\[
T_{\max,\mu}<\infty\quad\Longrightarrow\quad
\limsup_{t\uparrow T_{\max,\mu}}
\left(\|n_\mu(t)\|_\infty+\|c_\mu(t)\|_{W^{1,\infty}}\right)=\infty.
\]

\smallskip
\noindent\textit{\bf Step 2: Bootstrap interval and difference equations.}
On the common existence interval, put $a=n_\mu-n$, $b=c_\mu-c$, and define
\begin{equation}\label{eq:stability-distance}
D(t)=\|a(t)\|_\infty+\|b(t)\|_\infty+\|\nabla b(t)\|_\infty.
\end{equation}
Since the two solutions have the same initial data, we get $D(0)=0$.
{For every $S<\min\{T,T_{\max,\mu}\}$, local well-posedness gives
$(n,c),(n_\mu,c_\mu)\in C([0,S];\BUC\times\BUC^1)$.
Thus $D$ is continuous on the common preterminal existence interval.} Define
\[
\tau_\mu:=\sup\left\{\tau\in(0,\min\{T,T_{\max,\mu}\}):D(s)<1\ \text{for every }s\in[0,\tau]\right\}.
\]
By local existence, the two solutions exist on a common time interval $[0,h]$ for some $h>0$. Since $D(0)=0$ and $D$ is continuous in time, we have $D(t)<1$ for all sufficiently small $t\ge0$. Hence $\tau_\mu>0$. For $0\le s<\tau_\mu$, \eqref{eq:stability-reference-bound} gives
\begin{align}\label{n mu-c mu}
    \|n_\mu(s)\|_\infty+\|c_\mu(s)\|_{W^{1,\infty}}
\le K+D(s)\le K+1.
\end{align}
Subtracting \eqref{ks:system} from \eqref{ks:damped}, we obtain
\begin{equation}\label{eq:stability-difference-equations}
\begin{cases}
a_t-\Delta a=-\nabla\cdot(a\nabla c_\mu+n\nabla b)-\mu n_\mu^2,\\
b_t-\Delta b=-ac_\mu-nb,\qquad a(0)=b(0)=0.
\end{cases}
\end{equation}

\smallskip
\noindent\textit{\bf Step 3: Derivation of the Volterra inequality.}
For $0\le t<\tau_\mu$, the corresponding mild formulations are
\[
\begin{aligned}
a(t)
={}&-\int_0^t\nabla\cdot e^{(t-s)\Delta}(a\nabla c_\mu+n\nabla b)(s)\,ds
-\mu\int_0^t e^{(t-s)\Delta}n_\mu(s)^2\,ds,\\
b(t)
={}&-\int_0^t e^{(t-s)\Delta}(ac_\mu+nb)(s)\,ds.
\end{aligned}
\]
By the heat-semigroup estimates \eqref{eq:heat-bounds}, \eqref{eq:stability-reference-bound} and \eqref{n mu-c mu}, we have
\[
\begin{aligned}
\|a\nabla c_\mu+n\nabla b\|_\infty
&\le(K+1)\|a\|_\infty+K\|\nabla b\|_\infty
\le(K+1)D,\\
\|ac_\mu+nb\|_\infty
&\le(K+1)\|a\|_\infty+K\|b\|_\infty
\le(K+1)D,\\
\mu\|n_\mu^2\|_\infty&\le\mu(K+1)^2.
\end{aligned}
\]
Applying these estimates to the mild formulations, we obtain
\[
\begin{aligned}
\|a(t)\|_\infty
&\le C(K+1)\int_0^t(t-s)^{-\frac12}D(s)\,\dd s
+\mu(K+1)^2t,\\
\|b(t)\|_\infty
&\le(K+1)\int_0^tD(s)\,\dd s,\\
\|\nabla b(t)\|_\infty
&\le C(K+1)\int_0^t(t-s)^{-\frac12}D(s)\,\dd s.
\end{aligned}
\]
Adding these inequalities and using $t\le T$, we obtain
\begin{equation}\label{ks:volterra}
D(t)\le A_T\mu+B_T\int_0^t
\left(1+(t-s)^{-\frac12}\right)D(s)\,\dd s,
\end{equation}
where $A_T=T(K+1)^2$ and $B_T=(K+1)\max\{1,2C\}$ depend only on $T$, $K$ and the heat-semigroup constants, and are independent of $\mu$.

\smallskip
\noindent\textit{\bf Step 4: Uniform stability estimate by an exponential weight.}
Since
\[
\int_0^T e^{-\omega r}(1+r^{-\frac12})\,dr
\le\omega^{-1}+\sqrt{\pi}\,\omega^{-\frac12}\longrightarrow0
\qquad\text{as }\omega\to\infty,
\]
we choose $\omega>0$, depending only on $B_T$, such that
\begin{equation}\label{eq:Volterra-weight}
B_T\int_0^T e^{-\omega r}(1+r^{-\frac12})\,dr\le\frac12.
\end{equation}
For $0<t_*<\tau_\mu$, define
\[
Z(t_*):=\sup_{0\le t\le t_*}e^{-\omega t}D(t).
\]
Multiplying \eqref{ks:volterra} by $e^{-\omega t}$ gives
\[
\begin{aligned}
e^{-\omega t}D(t)
&\le A_T\mu+B_T\int_0^t e^{-\omega(t-s)}
\left(1+(t-s)^{-\frac12}\right)e^{-\omega s}D(s)\,ds\\
&\le A_T\mu+\frac12Z(t_*).
\end{aligned}
\]
Taking the supremum over $0\le t\le t_*$, we obtain $Z(t_*)\le2A_T\mu$. Consequently,
\begin{equation}\label{eq:Volterra-conclusion}
\sup_{0\le t\le t_*}D(t)\le2A_Te^{\omega T}\mu=:C_T\mu.
\end{equation}
The constant $C_T$ is independent of $\mu$.

\smallskip
\noindent\textit{\bf Step 5: Closing the bootstrap and proving stability.}
Choose
\[
\mu_T=\min\left\{1,\frac{1}{2C_T}\right\}.
\]
For $0\le\mu\le\mu_T$, \eqref{eq:Volterra-conclusion}, valid for every $t_*<\tau_\mu$, yields
\[
D(t)\le\frac12,\qquad 0\le t<\tau_\mu.
\]
If $\tau_\mu<\min\{T,T_{\max,\mu}\}$, continuity would extend the inequality $D<1$ to a larger interval, contradicting the definition of $\tau_\mu$. Hence $\tau_\mu=\min\{T,T_{\max,\mu}\}$. If $T_{\max,\mu}\le T$, the same estimate gives
\[
\sup_{0\le t<T_{\max,\mu}}
\left(\|n_\mu(t)\|_\infty+\|c_\mu(t)\|_{W^{1,\infty}}\right)
\le K+\frac12,
\]
contradicting the continuation criterion. Therefore $T_{\max,\mu}>T$.

Finally, letting $t_*\uparrow T$ in \eqref{eq:Volterra-conclusion} and using continuity at $T$, we obtain
\[
\sup_{0\le t\le T}
\left(\|n_\mu(t)-n(t)\|_\infty+\|c_\mu(t)-c(t)\|_{W^{1,\infty}}\right)
\le C_T\mu.
\]
In particular,
\[
\sup_{0\le t\le T}
\left(\|n_\mu(t)\|_\infty+\|c_\mu(t)\|_{W^{1,\infty}}\right)
\le K+C_T\mu\le K+\frac12.
\]
This proves \eqref{ks:stability}.
\end{proof}

\appendix
\section{Local well-posedness and a uniform restart time}
\label{app:local}

We prove Lemma~\ref{lem:local} in the space \eqref{eq:Xspace}
with the norm \eqref{eq:Xnorm}.

\begin{proof}[Proof of Lemma~\ref{lem:local}]
Fix a starting time $t_*$ and a bound $M_*\ge0$ such that
\[
\|(n(t_*),c(t_*),u(t_*))\|_X\le M_*.
\]
We construct the solution on an interval of length
$h=h(M_*,\|F\|_\infty)>0$, independent of $t_*$.
By time translation, it suffices to start at \(t_*=0\); we retain
\(t_*\) in the mild formulas to emphasize the restart property.
Since \(\diver u=0\), the density equation is
\(n_t-\Delta n=-\diver(n(\nabla c+u))\). The three mild formulas
are therefore
\begin{equation}\label{eq:mild}
\begin{aligned}
n(t)&=e^{(t-t_*)\Delta}n(t_*)
-\int_{t_*}^{t}\nabla\cdot e^{(t-s)\Delta}
       [n(\nabla c+u)](s)\dd s,\\
c(t)&=e^{(t-t_*)\Delta}c(t_*)
-\int_{t_*}^{t}e^{(t-s)\Delta}
       [nc+u\cdot\nabla c](s)\dd s,\\
u(t)&=e^{(t-t_*)\Delta}u(t_*)
+\int_{t_*}^{t}e^{(t-s)\Delta}\PP(nF)(s)\dd s.
\end{aligned}
\end{equation}
Here \(\PP\) denotes the Leray projection. In particular, the
sign of the transport in the density flux is the same as that
of the chemotactic flux in this formula.

\smallskip
\noindent\textit{\bf Step 1: The linear bounds.}
For \(q=1,\infty\), the heat kernel gives
\begin{equation}\label{eq:heat-bounds}
\|e^{s\Delta}f\|_q\le\|f\|_q,\qquad
\|\nabla e^{s\Delta}f\|_q\le Cs^{-\frac{1}{2}}\|f\|_q.
\end{equation}
The heat semigroup is strongly continuous on \(L^1\), \(L^2\),
and \(\BUC\), and on \(\BUC^1\) after commuting first derivatives.
For the Stokes term we use
\begin{equation}\label{eq:Stokes-L2}
\|e^{s\Delta}\PP f\|_2\le\|f\|_2,\qquad
\|e^{s\Delta}\PP f\|_\infty\le Cs^{-\frac{3}{4}}\|f\|_2.
\end{equation}
The second estimate follows by splitting the heat semigroup
into two factors and using \(L^2\) boundedness of \(\PP\)
and the heat \(L^2\)-to-\(L^\infty\) bound. We do not require
boundedness of the Leray projection on \(L^\infty\).

\smallskip
\noindent\textit{\bf Step 2: Estimates for the fixed-point map.}
Consider the closed ball of radius \(R\) in
\(C([t_*,t_*+h];X)\). The density flux obeys
\begin{equation}\label{eq:local-density-flux}
\|n(\nabla c+u)\|_1+\|n(\nabla c+u)\|_\infty
\le (\|n\|_1+\|n\|_\infty)
    (\|\nabla c\|_\infty+\|u\|_\infty)\le CR^2.
\end{equation}
By \eqref{eq:heat-bounds} and \eqref{eq:local-density-flux},
its contribution to the \(L^1\cap\BUC\) norm is at most
\(CR^2\int_0^h s^{-1/2}\dd s\le C\sqrt h R^2\).
The signal source satisfies
\begin{equation}\label{eq:local-signal-source}
\|nc+u\cdot\nabla c\|_\infty\le CR^2.
\end{equation}
Applying \eqref{eq:heat-bounds} to
\eqref{eq:local-signal-source} and its heat-evolved gradient gives a
\(\BUC^1\) contribution bounded by \(C(h+\sqrt h)R^2\).
No spatial integral of \(c\) is used. Finally,
\begin{equation}\label{eq:local-Stokes-force}
\|nF\|_2\le\|F\|_\infty
                  \|n\|_1^{\frac{1}{2}}\|n\|_\infty^{\frac{1}{2}}
\le C\|F\|_\infty R.
\end{equation}
By \eqref{eq:Stokes-L2} and \eqref{eq:local-Stokes-force},
the fluid contribution is at most
\(C\|F\|_\infty(h+h^{1/4})R\). These bounds also show that
the integral terms tend to zero in \(X\) at the starting time.
Their continuity follows by splitting an integral into a short
terminal part, controlled by the integrable kernel, and an earlier
part, on which the semigroup is strongly continuous.

\smallskip
\noindent\textit{\bf Step 3: Difference estimates and contraction.}
For two elements \((n_i,c_i,u_i)\), \(i=1,2\), in the same
ball, the density flux difference has the exact decomposition
\begin{equation}\label{eq:difference-density-flux}
n_1(\nabla c_1+u_1)-n_2(\nabla c_2+u_2)
=(n_1-n_2)(\nabla c_1+u_1)
+n_2\big(\nabla(c_1-c_2)+u_1-u_2\big).
\end{equation}
The \(L^1\) and \(L^\infty\) norms in
\eqref{eq:difference-density-flux} are bounded by \(CR\)
times the \(X\) norm \eqref{eq:Xnorm} of the difference. Likewise,
\begin{equation}\label{eq:difference-signal-source}
\begin{aligned}
n_1c_1-n_2c_2&=(n_1-n_2)c_1+n_2(c_1-c_2),\\
u_1\cdot\nabla c_1-u_2\cdot\nabla c_2
&=(u_1-u_2)\cdot\nabla c_1+u_2\cdot\nabla(c_1-c_2).
\end{aligned}
\end{equation}
The two terms in \eqref{eq:difference-signal-source} have the
same bound in \(L^\infty\).
For the force difference,
\begin{equation}\label{eq:difference-Stokes-force}
\|(n_1-n_2)F\|_2
\le\|F\|_\infty\|n_1-n_2\|_1^{\frac{1}{2}}
                          \|n_1-n_2\|_\infty^{\frac{1}{2}}
\le\frac{\|F\|_\infty}{2}
      (\|n_1-n_2\|_1+\|n_1-n_2\|_\infty).
\end{equation}
Apply \eqref{eq:heat-bounds} and \eqref{eq:Stokes-L2} to
\eqref{eq:difference-density-flux}--\eqref{eq:difference-Stokes-force}.
The fixed-point map then has Lipschitz constant at most
\begin{equation}\label{eq:contraction}
CR(h+\sqrt h)+C\|F\|_\infty(h+h^{\frac{1}{4}}).
\end{equation}
Choose \(R=2(M_*+1)\), increasing the fixed factor if necessary
to bound the linear evolutions, and then take \(0<h\le1\) so
that the nonlinear map preserves the ball and
\eqref{eq:contraction} is at most \(1/2\). Banach's fixed-point
theorem gives existence and uniqueness. Both requirements depend
only on \(M_*\) and \(\|F\|_\infty\). The difference estimates
\eqref{eq:difference-density-flux}--\eqref{eq:difference-Stokes-force}
also yield continuous dependence and uniqueness on longer
bounded intervals by successive applications of this argument.

\smallskip
\noindent\textit{\bf Step 4: Classical regularity at positive times.}
On a compact subinterval of \((t_*,T_{\max})\), the constructed
solution has bounded \(n,c,\nabla c,u\). Thus
\[
n_t-\Delta n=-\diver(n(\nabla c+u)),\qquad
c_t-\Delta c=-nc-u\cdot\nabla c
\]
have, respectively, a bounded divergence source and a bounded
scalar source. Interior heat-kernel difference estimates give
parabolic H\"older bounds for \(n\) and \(\nabla c\), with a
fixed small exponent \(0<\alpha<1/4\). Since
\(n\in L^1\cap L^\infty\), \(nF\) is bounded in every finite
\(L^q\). In particular, the Stokes gradient estimate with \(q=6\)
has singularity \((t-s)^{-3/4}\), which is time integrable.
Difference estimates give H\"older control of \(u\) and
\(\nabla u\) on a smaller time interval.

The density flux is now H\"older continuous. Interior
divergence-form Schauder estimates give \(\nabla n\), and the
signal equation gives \(\nabla^2c\) and \(c_t\). Consequently
\[
n_t-\Delta n=-(u+\nabla c)\cdot\nabla n-n\Delta c
\]
has a H\"older right-hand side, so \(n_t,\nabla^2n\) follow from
the scalar Schauder estimate. The Stokes equations give
\(u_t,\nabla^2u,\nabla\pi\) by the corresponding interior estimate.
For the linear parabolic estimates and the classical regularity
used in this bootstrap, see \cite[Chapter~IV]{LSU1968} and
\cite[Section~5.1]{Lunardi1995}. The assumptions on \(F\) suffice for
this finite bootstrap. With the smooth initial data in the
theorem, the solution attains those data and has the usual
classical regularity at positive times.

\smallskip
\noindent\textit{\bf Step 5: positivity, mass, and continuation.}
For nonnegative density data, the scalar maximum principle
applied to the linear equation for \(n\) with the constructed
coefficients preserves nonnegativity. The signal equation then
preserves \(0\le c\le\|c(t_*)\|_\infty\). One may first apply
this argument to smooth approximations on an interval of the
uniform local existence length and pass to the mild solution
by continuous dependence. For the data in the theorem, the
calculation also follows directly on classical time intervals.
The density flux belongs to \(L^1\) at every time and is bounded
there on each local interval. Integrate its mild formula and
use the zero integral of the heat-kernel gradient to obtain
\[
\int_{\R^3}n(t)=\int_{\R^3}n(t_*).
\]
This proves mass conservation without an assumption on
pointwise decay of the derivatives at infinity.

If a finite maximal time had a uniform preterminal \(X\) bound,
choose \(t_*\) within half the uniform existence length of that
time and repeat the construction. Uniqueness joins the new and
old solutions, and positive-time smoothing makes the extension
classical. This contradicts maximality and proves
\eqref{eq:continuation}. The pressure recovered from Stokes is
determined up to a spatial constant depending on time.
\end{proof}

\section{A technical lemma}

{We use the following iteration lemma
\cite[Chapter~V, Lemma~3.1]{G83}; see also \cite[Lemma~8]{ChYa}.}

\begin{Lem}[Lemma 8 \cite{ChYa}]
\label{lem:Gia}
Let $f(r)$ be a non-negative bounded function on $[R_0,R_1] \subset \R_+$. If there are {nonnegative} constants $A,B,D$ and positive exponents $b<a$ and a parameter $\theta \in (0,1)$ such that for all $R_0 \le \rho < \tau \le R_1$
\[
f(\rho) \le \theta f(\tau) + \frac{A}{(\tau-\rho)^{a}} + \frac{B}{(\tau-\rho)^{b}} + D,
\]
then for all $R_0 \le \rho < \tau \le R_1$
\[
f(\rho) \leq C(a,\theta) \left[\frac{A}{(\tau-\rho)^{a}} + \frac{B}{(\tau-\rho)^{b}} + D\right].
\]
\end{Lem}

\section*{Declarations}
\begin{itemize}
		\item \textbf{Acknowledgments} The authors would like to thank Professor Zhifei Zhang for some helpful communications. 
		W. Wang was supported by National Key R\&D Program of China (No.2023YFA1009200), NSFC under grant 12471219 and Natural Science
Foundation of Liaoning Province (No.2026-MS-001). 
		\item \textbf{Conflict of interest} The authors declare that they have no conflict of interest.
		\item \textbf{Data Availability} Data sharing is not applicable to this article as no datasets were generated or analyzed during the current study.
        \item \textbf{Use of artificial intelligence} During the preparation of this work, the authors used ChatGPT (OpenAI) in order to assist with language editing, manuscript organization, and checking mathematical arguments. After using this tool, the authors reviewed and edited the content as needed and take full responsibility for the content of this manuscript.
\end{itemize}

\providecommand{\bysame}{\leavevmode\hbox to3em{\hrulefill}\thinspace}
\providecommand{\MR}{\relax\ifhmode\unskip\space\fi MR }
\providecommand{\MRhref}[2]{%
  \href{http://www.ams.org/mathscinet-getitem?mr=#1}{#2}
}
\providecommand{\href}[2]{#2}

\end{document}